\documentclass[reqno]{article}
\usepackage[top=2.5cm, bottom=2.5cm, left=3cm, right=3cm]{geometry}
\usepackage{marginnote}

\usepackage{lmodern}

\usepackage{amssymb}
\usepackage{amsmath}
\usepackage{mathtools}
\usepackage{amsthm}
\usepackage{graphicx}
\usepackage{tikz}
\usepackage{tikz-cd}
\usepackage{pgf}
\usepackage{subcaption}
\usepackage{wrapfig}
\usepackage{array}
\usepackage{tabularx}
\usepackage{titlesec}
\usepackage{eepic}
\usepackage{multicol}
\usepackage[linktocpage=true]{hyperref}
\usepackage[hypcap=true]{caption}
\usepackage[cmtip,all]{xy}
\usepackage{enumitem}
\usepackage{leftidx}
\usepackage[mathscr]{euscript}
\usepackage{setspace}
\usepackage{needspace}

\usepackage{parskip}

\makeatletter
\def\thm@space@setup{%
  \thm@preskip=\parskip \thm@postskip=0pt
}
\makeatother

\allowdisplaybreaks

\renewcommand{\Re}{\operatorname{Re}}
\renewcommand{\Im}{\operatorname{Im}}

\usepackage{titlesec}
\titleformat{\section}[block]{\color{black}\large\bfseries\filcenter}{\thesection.}{0.5em}{}
\titleformat{\subsection}[hang]{\bfseries}{}{0.5em}{}

\numberwithin{equation}{section}

\newtheorem{theorem}{Theorem}[section]
\newtheorem{lemma}[theorem]{Lemma}
\newtheorem{proposition}[theorem]{Proposition}
\newtheorem{corollary}[theorem]{Corollary}
\newtheorem{definition}[theorem]{Definition}
\newtheorem{problem}[theorem]{Problem}

\newtheorem{remark}[theorem]{Remark}

\newtheorem{ltheorem}{Theorem}

\newtheorem{lproblem}{Problem}

\newtheorem{Model}{Theorem}

\newtheorem{model}[Model]{Model}

\titleformat{\subsection}[runin]{\bfseries}{}{}{}[.]
\titleformat{\subsubsection}[runin]{\bfseries}{}{}{}[.]
\newcounter{introequation}
\newenvironment{nsecequation}{\refstepcounter{introequation}\equation}{\tag{\theintroequation}\endequation}

\makeatletter
\renewenvironment{proof}[1][\proofname]{%
   \par\pushQED{\qed}\normalfont%
   \topsep6\p@\@plus6\p@\relax
   \trivlist\item[\hskip\labelsep\bfseries#1\@addpunct{.}]%
   \ignorespaces
}{%
   \popQED\endtrivlist\@endpefalse
}
\makeatother

\addtocounter{tocdepth}{0}
\newcommand{\CC}{\mathbb{C}}

\newcommand{\EE}{\mathbb{E}}
\newcommand{\FF}{\mathbb{F}}

\newcommand{\HH}{\mathbb{H}}

\newcommand{\KK}{\mathbb{K}}

\newcommand{\QQ}{\mathbb{Q}}
\newcommand{\RR}{\mathbb{R}}
\newcommand{\TT}{\mathbb{T}}

\newcommand{\XX}{\mathbb{X}}

\newcommand{\ZZ}{\mathbb{Z}}

\newcommand{\B}{\mathcal{B}}

\newcommand{\D}{\mathcal{D}}

\newcommand{\F}{\mathcal{F}}

\def\H{\mathcal{H}}

\newcommand{\VN}{\mathcal{L}}
\newcommand{\M}{\mathcal{M}}
\newcommand{\N}{\mathcal{N}}
\def\O{\mathcal{O}}

\newcommand{\U}{\mathcal{U}}

\newcommand{\X}{\mathcal{X}}

\newcommand{\Ee}{\mathscr{E}}

\newcommand{\sD}{\mathrm{D}}

\newcommand{\sG}{\mathrm{G}}
\newcommand{\sH}{\mathrm{H}}

\newcommand{\sO}{\mathrm{O}}

\newcommand{\sX}{\mathrm{X}}

\newcommand{\into}{\hookrightarrow}
\newcommand{\onto}{\twoheadrightarrow}
\newcommand{\longto}{\longrightarrow}

\newcommand{\equivalent}{\Longleftrightarrow}

\DeclareMathOperator*\wstspan{\overline{\mathrm{span}^{\wast}}}

\def\XXint#1#2#3{{\setbox0=\hbox{$#1{#2#3}{\int}$ }
\vcenter{\hbox{$#2#3$ }}\kern-.6\wd0}}

\def\1{\mathbf{1}}
\def\Id{\mathrm{id}}

\def\M{\mathcal{M}}

\def\N{\mathcal{N}}
\def\Zent{\mathcal{Z}}

\def\Homeo{\mathrm{Homeo}}

\makeatletter
\newsavebox{\@brx}
\newcommand{\llangle}[1][]{\savebox{\@brx}{\(\m@th{#1\langle}\)}%
  \mathopen{\copy\@brx\mkern2mu\kern-0.9\wd\@brx\usebox{\@brx}}}
\newcommand{\rrangle}[1][]{\savebox{\@brx}{\(\m@th{#1\rangle}\)}%
  \mathclose{\copy\@brx\mkern2mu\kern-0.9\wd\@brx\usebox{\@brx}}}
\makeatother

\newcommand{\Vrt}{\mathrm{Vert}}
\newcommand{\Edg}{\mathrm{Edge}}
\newcommand{\ori}{\mathrm{o}}
\newcommand{\tar}{\mathrm{t}}

\def\BMO{\mathrm{BMO}}

\def\sgn{\mathrm{sgn}}

\newcommand{\vertiii}[1]{
 {\left\vert\kern-0.25ex\left\vert\kern-0.25ex\left\vert #1 
  \right\vert\kern-0.25ex\right\vert\kern-0.25ex\right\vert}
}

\def\B{\mathcal{B}}

\def\wast{{{\mathrm{w}}^\ast}}
\def\op{\mathrm{op}}
\def\cb{\mathrm{cb}}
\def\Aut{\mathrm{Aut}}

\def\Iso{\mathrm{Iso}}

\def\SL{\mathrm{SL}}
\def\PSL{\mathrm{PSL}}

\def\PSO{\mathrm{PSO}}
\def\PSU{\mathrm{PSU}}

\def\st{\mathrm{St}}
\def\Aff{\mathrm{Aff}}

\def\acts{\curvearrowright}
\def\BS{\mathrm{BS}}
\DeclareMathOperator*\freePr{\scalebox{1.5}{\raisebox{-0.2ex}{$\ast$}}}

\title{
  Noncommutative Cotlar identities\\
  for groups acting on tree-like structures
}

\author{
  Adri\'an M. Gonz\'alez-P\'erez\thanks{
    Partially supported by the ANR Grant \texttt{ANR-18-CE04-0021} and the \textit{Ramón y Cajal} grant \texttt{RYC2022-037045-I} (Ministerio de Ciencia, Spain).
    Research funded by the Madrid Government program \texttt{V PRICIT},
    grant number \texttt{SI1/PJI/2019-00514}.
  }, \
  Javier Parcet\thanks{
    Partially supported by the Spanish Grant \texttt{PID2019-107914GB-I00} 
    ``Fronteras del An\'alisis Arm\'onico'' (MCIN / PI J. Parcet).
  }
  \ and 
  Runlian Xia\thanks{
    The three authors were partly supported by the Severo Ochoa Grant \texttt{CEX2019-000904-S} (ICMAT), 
    funded by \texttt{MCIN/AEI 10.13039/501100011033}.
  } 
}
\date{}

\begin{document}

\maketitle

\null

\begin{abstract}
  Let $T_m$ be a noncommutative Fourier multiplier. In recent work,
Mei and Ricard introduced a noncommutative analogue of Cotlar's 
identity in order to prove that certain multipliers are bounded on
the noncommutative $L_p$-spaces of a free group. 
Here, we study Cotlar-type identities in full generality, 
giving a closed characterization for them in terms of $m$:
 \[
   \big( m(g h) - m(g) \big) \, \big( m(g^{-1}) - m(h) \big) = 0, 
   \; \forall g  \in \sG\setminus \{e\}, h \in \sG.
 \]
Using a geometric argument, we prove that if $X$ is a tree 
---or more generally an $\RR$-tree--- on which $\sG$ acts and $m$ lifts to a 
function $\widetilde{m}: X \to \CC$ that is constant on the connected subsets of $X \setminus \{x_0\}$, 
then $m$ satisfies Cotlar's identity and thus $T_m$ is bounded in $L_p$ for $1 < p < \infty$.
This result establishes a new connection between group actions on $\RR$-trees and Fourier multipliers. We show that $m$ is trivial when the action has global fixed points. This machinery allows us to simultaneously generalize the free group transforms of Mei and Ricard and the theory of Hilbert transforms in left-orderable groups, which follows from Arveson's subdiagonal algebras. Using Bass-Serre theory, we construct new examples of Fourier multipliers in groups. These include new families like Baumslag-Solitar groups. We also show that a natural Hilbert transform in $\PSL_2(\CC)$ satisfies Cotlar's identity when restricted to the Bianchi group $\PSL_2(\ZZ[\sqrt{-1}])$.
\end{abstract}

\vskip-20pt 

\null

\section*{Introduction}

% The Classical Hilbert transform.
% Multiplier expression
% Cotlar identity vs CZO theory. Concise explanation of the recursive method
%\textbf{The classical Hilbert transform.} see \cite{Duoan2001Book, Ste1970}
%\marginpar{{\color{red} I think the adjective classical is better omitted}}
The Hilbert transform was introduced by Hilbert in 1912 as part of his investigation of the Riemann problem in the realm of complex analysis \cite{Hilbert1989}. Indeed, it may be regarded as the operator describing the boundary behavior of the harmonic conjugate in the upper half plane. Explicitly, it is given by
\begin{equation}
  \label{eq:HT}
  \tag{HT}
  H f(x) 
  \, = \, 
  \frac1{\pi} \lim_{\varepsilon \to 0^+} \int\displaylimits_{|x - y| > \varepsilon} \frac{f(y)}{x - y} \, dy.
\end{equation}
Equivalently, it is the Fourier multiplier $(Hf)^\wedge(\xi) = i \, \sgn(\xi) \, \widehat{f}(\xi)$ \cite{Duoan2001Book}. %The importance of its boundedness over $L_p$ spaces was recognized early. 
In 1924, M. Riesz proved its $L_p$-boundedness for all $p \in 2 \ZZ_+$ using an \textit{ad hoc} complex analysis argument \cite{Riesz1924Conjugate, Riesz1928ConjugateMZ}. By duality and Marcinkiewicz's interpolation this yields $L_p$-boundedness for every $1 < p < \infty$. Afterwards, Kolmogorov and Calder\'on-Zygmund found proofs giving the weak type $(1,1)$ \cite{Kolmogorov1925Conjugate, SteinP1933, Calderon1950}. 
%Among other consequences, the mapping properties of the Hilbert transform are crucial for the $L_p$-convergence of partial Fourier series/integrals on Euclidean spaces. 
%In fact, by elementary manipulations, the frequency restriction of $f \in L_p(\RR^n)$ to $N \cdot \Pp$ converges to $f$ in the $L_p$-norm as $N \to \infty$, where $N \cdot \Pp \subseteq \RR^n$ is the dilation of any convex polyhedron $\Pp$ containing the origin.
%According to Feffermann's construction in his solution of the ball multiplier problem, this result is false if $\Pp$ is convex set such that $\partial \Pp$ has nonvanishing curvature. The intermediate case in which $\Pp$ is an infinitely-faceted polyhedron has been completely solved in dimension $2$ but remains open in higher dimensions \cite{Alfonseca2003, AlfonSoriaVargas2003, Bateman2009kakeya, ParcetRogers2015Directions}. 

In 1955, Cotlar proved the $L_p$-boundedness of the Hilbert transform through an extremely simple argument \cite{Cotlar1955Unified}. He showed that $H$ is $L_p$-bounded for $p = 2^k$ recursively ---from the trivial case $p=2$---using his elegant \emph{Cotlar identity}
\begin{equation}
  \label{eq:ClassicalCotlar}
  \tag{Classical Cotlar}
  | H f |^2 \, = \, 2 H \big( f \, H f \big) -  H \big( H |f|^2 \big).
\end{equation}
A key point in our work is to notice that Cotlar's identity and their generalizations have a nicer expression at the frequency side of the Fourier transform. As an illustration, notice that an Euclidean Fourier multiplier $(T_m f)^\wedge(\xi) = m( \xi) \, \widehat{f}(\xi)$ satisfies Cotlar's identity precisely when
\begin{equation}
  \tag{Classical ${\mathrm{C}\widehat{\mathrm{otla}}\mathrm{r}}$}
  \label{eq:DualCotlarRn}
  \big( m(\xi + \eta) - m( \xi) \big) \big( m(-\xi) - m(\eta) \big) = 0,
  \quad \mbox{ a.e } \xi, \eta \in \RR.
\end{equation}
%The Hilbert transform is just one example among the many multipliers that satisfy the classical {\hyperref[eq:DualCotlarRn]{Cotlar}} identity above, which in turn implies the $L_p$-boundedness of $T_m$. 
In fact, any bounded function satisfying the classical {\hyperref[eq:DualCotlarRn]{Cotlar}} identity above will be bounded on $L_p(\RR)$ for every $1<p<\infty$. 
%\marginnote{{\color{red}ADR: I changed the phrasing. So we weren't implying the existence of examples in $\RR$ other than $\{\sgn, \1\}$.}}
In this article, we will investigate similar identities on more general locally compact groups and their von Neumann algebras. A pioneering work in this direction was due to Mei and Ricard \cite{MeiRicard2017FreeHilb}, where the authors deployed a noncommutative analogue of \hyperref[eq:ClassicalCotlar]{Cotlar identity}, which holds in the context of amalgamated free product von Neumann algebras. The main goal of this article is to further Mei and Ricard's technique beyond free groups by illuminating the hidden connection between noncommutative Cotlar identities and group actions on trees and other tree-like structures ---like $\RR$-trees and uniquely arcwise connected spaces.

%This continues a pioneering idea of Mei and Ricard, who realized that an analogue of \eqref{eq:ClassicalCotlar} holds for free product von Neumann algebras \cite{MeiRicard2017FreeHilb}. 

% Multipliers on groups and NC Lp spaces
% Lack of a CZ theory on the purely NC setting. E
% xplain difficulties and very breifly "MAMS" + "Cadilhac-Conde-Parcet"
\textbf{Noncommutative Fourier multipliers.}
Here we will deal with operators analogous to the Hilbert transform over group algebras. Let $\sG$ be a locally compact group. %\marginnote{{\color{blue}ADR: Impersonal}} 
Its (left regular) group von Neumann algebra is defined as
\[
  \VN \sG := \wstspan\big\{ \lambda_g : g \in \sG \big\} \subseteq \B(L_2 \sG).
\] 
When $\sG$ is Abelian, $\VN \sG$ is isomorphic with $L_\infty(\widehat{\sG})$, the $L_\infty$-space over the Pontryagin dual of $\sG$. Thus, $\VN \sG$ behaves as a generalization of the Pontryagin dual for noncommutative groups. When $\sG$ is unimodular, the algebra $\VN \sG$ admits a normal, semifinite, and faithful tracial weight called the \emph{Plancherel trace} \cite[Chapter 8]{Ped1979}, with respect to which the noncommutative $L_p$-spaces $L_p(\VN \sG)$ are defined \cite{Terp1981lp, PiXu2003}. These $L_p$-spaces straightforwardly generalize the spaces of $L_p$-integrable elements over the dual of $\sG$. As such, many classical problems of Fourier analysis over $L_p$-spaces find an analogue in the noncommutative setting. A prominent example is the study of the $L_p$-boundedness of \emph{Fourier multipliers}. Indeed, given $m: \sG \to \CC$, the \emph{Fourier multiplier of symbol $m$} will be the ---potentially unbounded--- linear operator $T_m: D \subseteq \VN \sG \to \VN \sG$ given by linear extension of
\[
  T_m(\lambda_g) = m(g) \, \lambda_g.
\]
The boundedness of Fourier multipliers over noncommutative $L_p$-spaces presents challenges that are absent in the classical setting. A key difficulty is the extension of singular integral techniques to von Neumann algebras. Although steps toward a noncommutative Calder{\'o}n-Zygmund theory had been taken \cite{Par2009CZ, GonJunPar2017singular, JungeMeiParcetXia2019ACZ, CadiCondeParcet2021}, a fully noncommutative Calder\'on-Zygmund theory capable of yielding weak $(1,1)$ has not been found outside of semi-commutative or nilpotent settings. 

%\marginpar{{\color{red} Shall I remove this paragraph or merge it with the last one of the first page?}\color{blue} I made some small changes here, do you find them suitable?}
As hinted before, one possible way of overcoming this difficulty was found in  Mei and Ricard's article \cite{MeiRicard2017FreeHilb}. The noncommutative analogue of the Cotlar identity developed in their paper allowed them to prove that functions $m: \FF_2 \to \CC$ over the free group whose value $m(\omega)$ depends only on the starting letter $\{a, a^{-1}, b, b^{-1}\}$ of the reduced word $\omega \in \FF_2$ give rise to  bounded Fourier multipliers  on $L_p$, i.e.,  
\begin{equation}
  \label{eq:MR}
  \tag{MR}
  \big\| T_m: L_p(\VN \FF_2) \to L_p(\VN \FF_2) \big\|_\cb
  \, < \, \infty, \quad \mbox{ for every } 1 < p < \infty.
\end{equation}
In this paper, we will study a new geometric way to define $L_p$-bounded Fourier multipliers on groups admitting  actions on  tree-like structures, and we will see in Section \ref{S:BS} that this recovers \eqref{eq:MR} as a particular example.

\textbf{Noncommutative Cotlar identities.} Let $\sG_0 \subseteq \sG$ be an open subgroup of the locally compact and unimodular group $\sG$. 
It is trivial to see that $\sG_0$ is also unimodular and that $\VN \sG_0 \subseteq \VN \sG$ is a complemented inclusion, that is, an inclusion admitting a normal conditional expectation $\EE: \VN \sG \to \VN \sG_0$. %Furthermore, if $f = \lambda(\varphi) \in \VN \sG$ with $\varphi \in L_1(\sG)$, then $\EE[f] = \lambda(\1_{\sG_0} \varphi)$. 
We will say that a potentially unbounded multiplier $T_m$ satisfies a \emph{noncommutative Cotlar identity} with respect to the von Neumann subalgebra $\VN \sG_0$ iff
\begin{equation}
  \tag{Cotlar}
  \label{eq:NCCotlar}
  \EE^\perp \big[ T_m(f) \, T_m(f)^\ast \big] 
  \, = \, 
  \EE^\perp 
  \left[ 
    T_m \big( f \, T_m(f)^\ast \big) + T_m \big( f \, T_m(f)^\ast \big)^\ast - T_m \big( T_m(f f^\ast)^\ast \big)
  \right],
\end{equation}
where $\EE^\perp = (\Id - \EE)$.

We have distilled an easily verifiable closed formula \eqref{eq:CotlarFactor} for $m$ that is equivalent to \eqref{eq:NCCotlar} above, see Theorem \ref{thm:ClosedFormulaCotlar}. Since with an additional  assumption on the symbol, the Cotlar's formula implies $L_p$-boundedness, we obtain the following theorem.

\begin{ltheorem}
  \label{thm:TheoremA}
  Let $\sG$ be a locally compact unimodular group, $\sG_0 \subseteq \sG$ {a closed} subgroup and $m: \sG \to \CC$ a 
  left $\sG_0$-invariant bounded and measurable function. If $m$ satisfies that
  \begin{equation}
    \tag{${\mathrm{C}\widehat{\mathrm{otla}}\mathrm{r}}$}
    \label{eq:CotlarFactor}
    \big( m(g^{-1} ) - m(h) \big) \, \big( m(gh) - m(g) \big) \, = \, 0,
    \quad \text{ {for almost every } } g \in \sG \setminus \sG_0, \, h \in \sG,
  \end{equation}
  then $T_m$ is $L_p$-bounded for $1<p<\infty$ and furthermore
  \begin{nsecequation}
    \label{eq:TheoremABound}
    \big\| T_m: L_p(\VN \sG) \to L_p(\sG) \big\| \, \lesssim \, \left( \frac{p^2}{p-1} \right)^{\beta} \| m \|_\infty, 
    \quad \mbox{ with } \beta = \log_2(1+\sqrt{2}).
  \end{nsecequation}
\end{ltheorem}
%
%The result above holds true in the non-relative case in which the expectation $\EE$ is removed. This case can be thought of as a degenerate case in which $\sG_0$ is empty. This is specially useful when dealing with continuous groups and allows us to see the classical Cotlar identity in $\RR$ as a particular case of our theory, see Remark \ref{rmk:NonrelativeCotlar}.
The theorem above decouples into two different statements depending on whether $\sG_0 \subseteq \sG$ is open or has empty interior. In the case of an open subgroup, it holds that $0 < \mu(\sG_0)$ and thus condition \eqref{eq:CotlarFactor} has to be verified in a non-total set of pairs $g, \, h$. In this case, the associated Fourier multiplier $T_m$ satisfies \eqref{eq:NCCotlar} relative to the conditional expectation from $\VN \sG$ to $\VN \sG_0$. In contrast, in the case of a subgroup $\sG_0$ of empty interior we have that $\mu(\sG_0) = 0$ and thus \eqref{eq:CotlarFactor} has to be verified almost everywhere, in which case its associated Fourier multiplier satisfies the non-relative version of the Cotlar identity \eqref{eq:CotlarNoExp}. Furthermore, in the case of $\sG_0$ of empty interior, the condition of $m: \sG \to \CC$ being left $\sG_0$-invariant can be dropped, see Remark \ref{rmk:NonrelativeCotlar}. Although our most novel examples would be discrete groups, the case of $\mu(\sG_0) = 0$ is still useful as it recovers the classical case of the Hilbert transform on $\RR$.
% Explain that the natural discrete examples are:
%  1) Groups with a ZZ-quotient
%  2) Free groups, a la Mei-Ricard
% Question: Are there more? 
% Introduce Serre's property (FA)

% Answer: Yes, we have two families of examples
% Introduce R-trees and groups acting on them

The first advantage of the result above is that it makes the verification of the Cotlar identity for previously known cases almost trivial. Indeed, restricting ourselves to the discrete for clarity, we can easily show that it holds in the following situations.
\begin{enumerate}[label={\rm \textbf{(\arabic*)}}, ref={\rm {\arabic*}}, leftmargin=1.25cm]
  \item \label{itm:ExampleZZ}
  \textbf{Classical case.} In the classical case of $\sG = \ZZ$ and $\sG_0 =\{0\}$ with $m(x) = \sgn(x)$ we only have to verify that either $m(x+y) = m(x)$ or $m(-x) = m(y)$. But that is trivial since either $x$ and $y$ have different signs or $x + y$ and $x$ share the same sign.
  
  \item \label{itm:ExampleFree}
  \textbf{Free product case.} If $\sG = \sG_1 \ast \sG_2$ and $\sG_0 = \{e\}$ with both $\sG_2$ and $\sG_2$ discrete, then any function $m(\omega)$ such that its value depends on the first letter of the reduced word of $g$ satisfies \eqref{eq:CotlarFactor}. Indeed, we need to prove that either $m(g h) = m(g)$ or that $m(g^{-1}) = m(h)$. Assume the first equality fails, then the first letter of $g h$ and that of $g$ are different, but that can only happen if the reduced word of $h$ begins with the reduced word of $g^{-1}$. If that is the case $g^{-1}$ and $h$ have the same starting letter and thus $m(g^{-1}) = m(h)$. This family of examples was explored by Mei and Ricard \cite{MeiRicard2017FreeHilb}.
\end{enumerate}

A natural question that we answer affirmatively is whether there are examples of groups that go beyond those two categories. In order to explore that question it seems natural to search for bounded functions $m: \sG\to \CC$ satisfying \eqref{eq:CotlarFactor} with $\sG$ having Serre's property $(\mathrm{FA})$ \cite{Serre1980Trees}. A group $\sG$ is said to have Serre's property (FA) iff every action of $\sG$ on a tree has a global fixed point (a vertex in the tree which is fixed by the action of any $g\in \sG$). More deeply, Serre proved that a discrete group $\sG$ has property $(\mathrm{FA})$ iff it is finitely generated, it does not possess a quotient isomorphic to $\ZZ$ and it cannot be expressed as a nontrivial amalgamated free product $\sG = \sG_1 \ast_{A} \sG_2$. Therefore a group with property $(\mathrm{FA})$ is excluded from Examples \ref{itm:ExampleZZ} and \ref{itm:ExampleFree}. Although there are interesting examples with property $(\mathrm{FA})$, we also give applications to groups like Baumslag-Solitar groups $\BS(1,m)$ and the Bianchi group $\PSL_2(\O_{-1})$, which, despite failing property $(\mathrm{FA})$, admit bounded functions satisfying Cotlar's identity for reasons unrelated to them having $\ZZ$-quotients or being free products.

\textbf{Groups acting on uniquely arcwise connected spaces.} The closed formula in \eqref{eq:CotlarFactor} highlights a surprising connection between Cotlar's identity and geometric group theory. A topological space $X$ is arcwise connected iff given two points $x, y \in X$ there exists an injective continuous path $\gamma: [0,1] \to X$ joining $x$ and $y$. An arcwise connected space will be said to be a \emph{uniquely arcwise connected space} or \emph{UAC space} iff the path joining $x$ and $y$ is unique \cite{Bestvina2002}. Let $\sG \acts X$ be a topological action on a UAC space and fix a root $x_0 \in X$ with $\sG_0$ being the stabilizer $\st_{x_0}$ of $x_0$. Observe that $X \setminus \{x_0\}$ is given by 
\begin{nsecequation}
  \label{eq:DecompositionCC}
  X \setminus \{x_0\} \, = \, \bigsqcup_{\beta} X_\beta,
\end{nsecequation}
where each $X_\beta$ is arcwise connected. In some cases, the disjoint union above can be taken as a topological characterization of $X$ and the subsets $X_\beta$ as connected components. Nevertheless, there are UAC spaces for which $X \setminus \{x_0\}$ is not a disjoint union of connected components. Observe that the action of $\sG_0$ restricted to $X \setminus \{x_0\}$ permutes the $X_\beta$. The following theorem gives a machinery to get multipliers satisfying Cotlar's identity from actions on UAC spaces.

% Theorem B: Groups acting on UAC spaces
\begin{ltheorem}
  \label{thm:TheoremB}
  Let $\sG \acts X$ be a topological action on a UAC space. Fix $x_0 \in X$, $\sG_0$ being the stabilizer of $x_0$, $\sG_0= \st_{x_0}$, and 
  let $\widetilde{m}: X \to \CC$ be a bounded function satisfying that
  \begin{enumerate}[label={\rm \textbf{(\roman*)}}, ref={\rm {(\roman*)}}]
    \item \label{itm:TheoremB.1}
    $\widetilde{m}$ is constant over each $X_\beta$ of \eqref{eq:DecompositionCC}.
    \item \label{itm:TheoremB.2}
    $\widetilde{m}$ is constant over $\sG_0$ orbits, 
    ie $\, \widetilde{m}{|}_{X_\beta} = \widetilde{m}{|}_{X_\alpha}$ if there is an element 
    $h \in \sG_0$ such that $X_\beta = h \cdot X_\alpha$.
  \end{enumerate}
  Then, the function $m: \sG \to \CC$ given by $m(g) \, = \, \widetilde{m}(g \cdot x_0)$
  satisfies \eqref{eq:CotlarFactor} and therefore
  \begin{nsecequation}
    \big\| T_m: L_p(\VN \sG) \to L_p(\VN \sG) \big\| \, \lesssim \, \Big( \frac{p^2}{p - 1} \Big)^\beta \| m \|_\infty
    \quad \mbox{ with } \beta = \log_2(1+\sqrt{2}).
  \end{nsecequation}
\end{ltheorem}

By Remarks \ref{rmk:NonrelativeCotlar} and \ref{rmk:nonRelativeCaseModel}, condition \ref{itm:TheoremB.2} can be dropped when the interior of $\st_{x_0}$ is empty. The proof of Theorem \ref{thm:TheoremB} is so neat that can be tightly summarized here. First, notice that the condition \ref{itm:TheoremB.2} ensures that $m$ is left-$\sG_0$-invariant. Condition \ref{itm:TheoremB.1} on the other hand implies that \eqref{eq:CotlarFactor} holds. To see this, assume that $m(gh) \neq m(g)$. Since the two values are different, the arcwise connected subsets of $X \setminus \{x_0\}$ in which $g h \cdot x_0$ and $g \cdot x_0$ lay must be different, see Figure \ref{fig:UACaction}. Thus, there is a unique arc joining them that passes through $x_0$. Applying $g^{-1}$ to this arc gives an arc starting at $x_0$ and which passes by $g^{-1} \cdot x_0$ and $h \cdot x_0$ in that order. But, as such, $g^{-1} \cdot x_0$ and $h \cdot x_0$ must lay in the same arcwise connected subset and thus $m(h) = m(g^{-1})$.
\begin{center}
  \begin{figure}[ht]
    \centering
    \includegraphics[scale=1]{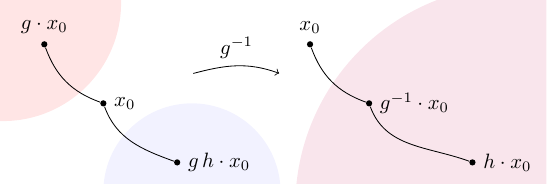}
    \caption{Action over paths.}
    \label{fig:UACaction}
  \end{figure}
\end{center}

% List of examples

% Explain what left-orderable group are
% Theorem B: left-orderable groups
\textbf{Left-orderable groups.} A family of examples that fits right into the model of Theorem \ref{thm:TheoremA} is that of \emph{left-orderable groups}. Those are groups admitting a \emph{total} or \emph{linear} order $(\sG, \preceq)$ that is invariant under left translations, ie $g \preceq h$ $\equivalent$ $a g \preceq a h$ for every $a, g, h \in \sG$. We will write $g \prec h$ when $g \preceq h$ but $g \neq h$. For left-orderable groups the following result holds for their sign function. 

\begin{ltheorem}
  \label{thm:TheoremC}
  Let $\sG$ be a left-orderable group and $\sgn: \sG \to \CC$ be the function 
  \[
    \sgn(g) \, = \,
    \begin{cases}
       1 & \mbox{ when } e \prec g\\
       0 & \mbox{ when } g = e\\
      -1 & \mbox{ when } g \prec e.
    \end{cases}
  \]
  Then $H = T_\sgn: L_p(\VN \sG) \to L_p(\VN \sG)$ satisfies that
  \[
    \big\| H: L_p(\VN \sG) \to L_p(\VN \sG) \big\| 
    \, \lesssim \,  \frac{p^2}{p - 1},
    \quad \mbox{ for every } 1 < p < \infty.
  \]
\end{ltheorem}

The boundedness of $H$ in Theorem \ref{thm:TheoremC} can be obtained by showing directly that \eqref{eq:CotlarFactor} holds. Alternatively, it is known that, in the discrete case, a group is left-orderable iff it acts on $\RR$ by order-preserving homeomorphisms. The observation that $\RR$ is a UAC space allows us to prove the result as a consequence of Theorem \ref{thm:TheoremB}. 
In principle, proving Cotlar's identity gives just that the $L_p$-norm grows like $\sO(p^\beta)$ with $\beta = \log_2(1 + \sqrt{2})$, as $p \to \infty$. Nevertheless, a more careful argument allows us to show that the constant can be lowered down to the optimal order $\sO(p)$, as long as $m(g) \overline{m(g^{-1})} = -1$ for $g \in \sG \setminus \{e\}$, see Corollary \ref{cor:OptConstantGroup}. It is also worth noticing that the above transforms for left-orderable group algebras can be seen as a particular example of the Hilbert transforms associated with Arveson's subdiagonal algebras \cite{Arveson1967Analytic}, for which the weak type $(1,1)$ was proved by Randrianantoanina \cite{Randrianantoanina1998Hilbert}, see also \cite[Theorem 8.4]{PiXu2003}. Therefore our geometric model in Theorem \ref{thm:TheoremB} generalizes simultaneously Mei and Ricard's free Hilbert transforms and subdiagonal Hilbert transforms, recovering the best known constants in both cases.
%Similarly, if in the UAC space model of Theorem \ref{thm:TheoremB} the space $X \setminus \{ x_0\}$ has just a finite number $r$ of connected components, then the exponent can also be lowered to $1$, see Corollary \ref{cor:FreeGroupOpt} and Corollary \ref{cor:OptimalConstant}.
%
Left-orderable groups include:
\begin{itemize}
  \item Torsion-free Abelian groups;
  \item Torsion-free nilpotent groups;
  \item Free groups $\FF_r$;
  \item Braid groups $B_n$;
  \item Right-angled Artin groups;
  \item Baumslag-Solitar groups $\BS(1, n)$ for $n \geq 2$;
  \item Surface groups;
  \item The Thompson group $F$.
\end{itemize} 
It is possible to obtain explicit examples of $L_p$-bounded Fourier multipliers satisfying Cotlar's identity in each of these families of groups.
Furthermore, there are known examples of left-orderable groups that have Serre's property $(\mathrm{FA})$. For instance, let $\widetilde{\PSL}_2(\RR) \onto \PSL_2(\RR)$ be the universal cover of $\PSL_2(\RR)$ and let $D(2,3,7) \subseteq \PSL_2(\RR)$ be the $(2,3,7)$-triangular group. Then, the lifting $\Gamma$ of $D(2,3,7)$ to $\widetilde{\PSL}_2(\RR)$ is an example of a group with Serre's property $(\mathrm{FA})$ that is also left-orderable \cite{Bergman1991Orderable,  CornulierKar2011, Khoi2003cut, Serre1980Trees}.
%{\marginpar{{\color{red}Are those references correct?}{\color{blue} Yes! I just added Serre's book as well.}}}
%\begin{equation}
%  \label{eq:GroupD}
%  \xymatrix{
%    \Gamma \ar@{->>}[rr] \ar@{^{(}->}[d] & & D(2,3,7) \ar@{^{(}->}[d] \\
%    \widetilde{\PSL}_2(\RR) \ar@{->>}[rr] & & \SL_2(\RR)
%  }
%\end{equation}

\textbf{Graphs of groups and Bass-Serre theory.} A wealth of examples of multipliers satisfying \eqref{eq:CotlarFactor} can be obtained from Bass-Serre theory ---which allows to classify groups acting on trees without edge inversions--- see \cite{Serre1980Trees}. Indeed, given a group acting on a tree $\sG \acts T$, it is possible to build a graph by taking the quotient with respect to the action $X = T/\sG$ and associating to each vertex and to each edge its corresponding stabilizer. Due to the lack of edge inversions, the stabilizer of an edge embeds into the stabilizers of its extremes. This structure ---a graph with groups on its edges and vertices and such that the groups at the edges embed into the extremes of said edge--- is called a \emph{graph of groups}. Let us denote it by $\XX$. Like in the case of graphs, it is possible to define a sort universal cover of $\XX$ such that its underlying graph is a tree $\widetilde{X}$ and its fundamental group $\pi_1(\XX)$ acts as Deck's transformations of the covering. The main point of the theory is that $\pi_1(\XX) \cong \sG$ and the action of $\pi_1(\XX) \acts \widetilde{X}$ recovers $\sG \acts T$. 

Our main theorem in Section \ref{S:BS} would be that the multiplier results for symbols depending on the starting letter, like inequality \eqref{eq:MR} from \cite{MeiRicard2017FreeHilb}, extends from free products to general graphs of groups. Let us fix some notation. First, when working with a graph of groups, we will consider that our graph is oriented, that both directions of the edge occur and that there can be both loops and multiple edges between two given vertices. As it is customary, given an edge $y$, we will denote by $\bar{y}$ the \emph{reverse edge} and by $\ori(y) \in \Vrt(X)$ and $\tar(y) \in \Vrt(X)$ the \emph{origin} and \emph{target} vertices of the edge. An orientation would be a subset $\Edg_+ \subseteq \Edg$ that contains either $y$ or $\bar{y}$ for each edge. We will also denote by $\sG_x$ the groups associated to vertices and by $\sH_y$ the groups associated to edges. By construction, we have that $\sH_y = \sH_{\bar{y}}$ and that there are injective homomorphisms $\alpha_y: \sH_y \to \sG_{\tar(y)}$ and $\alpha_{\bar{y}}: \sH_y \to \sG_{\ori(y)}$. We will denote the image group $\alpha_y[\sH_y] \subseteq \sG_{\tar(y)}$ by $\sH_y^y$. Similarly $\sH_{\bar{y}}^{\bar{y}}$ will denote the image of $\alpha_{\bar{y}}$. We can now recall the construction of $\pi_1(\XX)$. Define the auxiliary group $F(\XX)$ given by the free product of all the $\sH_y$ and $\sG_x$ ---where $y$ runs over the edge set $\Edg(X)$ and $x$ runs over the vertex set $\Vrt(X)$--- together with $\FF_{\Edg(X)}$, the free group generated by all the edges, with the following extra relations imposed 
\begin{nsecequation}
    \label{eq:ExtraRelations}
    y^{-1} = \bar{y}
    \;\;\; \text{ and } \;\;\;
    \alpha_y(h) \, y = y \, \alpha_{\bar{y}}(h), \; \forall y \in \Edg(X), h \in \sH_y
\end{nsecequation}
To construct the fundamental group, assume $X$ is connected and choose a base point $x_0 \in \Vrt(X)$. A closed path $c$ that starts and ends at $x_0$ will be a sequence of edges $c = y_1 \, y_2 \, \cdots y_m$, with $\tar(y_i) = \ori(y_{i+1})$ and $\ori(y_1) = x_0 = \tar(y_m)$. The group $\pi_1(\XX,x_0)$ is given by the subset of $F(\XX)$ of elements of the form:
  \begin{nsecequation}
    \label{eq:WordofPathType}
    g = | c, r | 
    \, = \,
    r_0 \, y_1 \, r_1 \, y_2 \, r_2 \cdots y_m \, r_m,
  \end{nsecequation}
where $r_j \in \sG_{\ori(y_{j+1})}$ for $j \leq m -1$ and $r_m \in \sG_{x_0}$. Here $r$ is just notation for the tuple $r = (r_0, \, r_1, \, \cdots r_m)$. We will call the pair $(c,r)$ a \emph{word of type $c$}. Similarly, a word $e \neq g = |c, r|\in \pi$ of type $c$ would be said to be in \emph{normal form} if it cannot be shortened applying the relations \eqref{eq:ExtraRelations}. 
%More formally, given a word $g = |c,r|$ of type $c$, it is in normal form if either $m = 0$ and $r_0 \neq e$ or, in the case in which $m \geq 1$, it holds that for every $1 \leq j \leq m - 1$, $r_j \not\in \sH_{y_j}^{y_j}$ when $y_{j+1} = {\overline{y}_j}$. Let $|c,r|$ and $|c,\mu|$ be two words of type $c$ with $r = (r_0, \, r_1, \, \dots, r_m)$ and $\mu = (\mu_0, \, \mu_1, \, \dots, \mu_m)$. They are said to be \emph{equivalent} if $\mu_0 = r_0 a_1^{\bar y_1}$ and $\mu_j=(a_j^{y_j})^{-1} r_j a_{j+1}^{\bar y_{j+1}}$. Here $a_j \in \sH_{y_j}$ and $a_j^{y_j}$ and $a_j^{\bar{y}_j}$ are the corresponding images in $\sH_{y_j}^{y_j}$ and $\sH_{\bar{y_j}}^{\bar{y_j}}$. We have that two words of type $c$ in normal form are equivalent if and only if they represent the same element in $\pi_1(\XX, x_0)$. Observe also that any $g \in \pi _1 (\XX, x_0)$ admits a unique normal form up to equivalence. 

We would like to define a multiplier $m(g)$ depending only on the starting segment $r_0 \, y_1$ of its associated word. But, given $h \in \sH_{y_1}^{y_1}$, we have, in $F(\XX)$, that
\[
    r_0 y_1 = r_0 \, \alpha_{y_1}(h^{-1}) \, \alpha_{y_1}(h) \, y_1 = r_0 \, \alpha_{y_1}(h^{-1}) \, y_1 \alpha_{\bar{y}_1}(h),
\]
as such, two words in normal form representing the same group element $g$ may have different starting segments. Nevertheless, $r_0$ can only change by an element in $\sH_{y_1}^{y_1}$ acting on the right. This motivates the definition of the \emph{space of starting segments} on $x_0 \in \Vrt(X)$ as
\begin{nsecequation}
    \label{eq:DefW}
    W_{x_0} = \bigsqcup_{y \, | \, \ori(y) = x_0} \,
    \sG_{x_0} / \sH_{\bar{y}}^{\bar{y}}
\end{nsecequation}
Given a symbol $m: \pi = \pi_1(\XX,x_0) \to \CC$ we will say that
\begin{itemize}
    \item $m$ \emph{depends on the starting segment} if it exists a function $\widetilde{m}:W_{x_0}\to \CC$ such that
    \begin{nsecequation}
        \label{eq:LiftFirstSegment}
        m( g ) 
        \, = \, 
        \begin{cases}
            \widetilde{m}\big( r_0 \cdot \sH_{\bar{y}_1}^{\bar{y}_1} \big), & \mbox{ when } g \not\in \sG_{x_0}\\
            0 & \mbox{ when } g \in \sG_{x_0},
        \end{cases}
    \end{nsecequation}
    where $g$ is as in \eqref{eq:WordofPathType}.
    \item $m$ \emph{depends on the starting edge} if it not only depends on the staring segment,
    but it is also constant in each of the components $\sG_{x_0} / \sH_{\bar{y}}^{\bar{y}}$, for each $y$ starting in $x_0$, in the disjoint union that forms $W_{x_0}$.
\end{itemize}

\begin{ltheorem}
  \label{prp:InitialSegment}
  Let $\pi = \pi_1(\XX,x_0)$ be the fundamental group of a graph of groups and let $m: \pi \to \CC$ be a function depending only on the initial segment. It holds that
    \begin{enumerate}[label={\rm \textbf{(\roman*)}}, ref={\rm {(\roman*)}}]
      \item \label{itm:InitialSegment1} $m$ satisfies Cotlar's identity \eqref{eq:CotlarNC} relative to $\sG_{x_0}$.
      
      \item \label{itm:InitialSegment2} Furthermore, if $m$ only depends on the starting edge, then it is left $\sG_{x_0}$-invariant. Thus, by Theorem \ref{thm:TheoremA} it induces a bounded Fourier multiplier on $L_p$ for every $1<p<\infty$ satisfying bound \eqref{eq:TheoremABound}.
    \end{enumerate}
\end{ltheorem}
The hypothesis that $m$ depends only on the starting edge is indeed quite restrictive. For instance, if there is only one edge, up to inversions, starting at $x_0$, then a symbol depending on the starting edge is constant for $g \not\in \sG_{x_0}$. Luckily, in many interesting examples the group $\sG_{x_0}$ will be Abelian. In that case, decomposing $\widetilde{m}|_{\sG_{x_0}/\sH_y^y}$ into characters allows us to work with functions depending on the starting segment, see Remark \ref{rmk:NonInvariantCharacters} and Theorem \ref{thm:InitialSegmentG0Abelian}. We will provide a direct proof of Theorem \ref{prp:InitialSegment}. Nevertheless, it also follows from noticing that, if $m$ depends on the initial segment, them it lifts to a function $\widetilde{m}: \widetilde{X} \to \CC$ on the Bass-Serre tree of $\pi_1(\XX,x_0)$ that is constant in connected components of $\widetilde{X} \setminus \{\widetilde{x}_0\}$, see Proposition \ref{prp:BassSerreModel1} for the details. Therefore, symbols depending on the starting segment lay within the template of Theorem \ref{thm:TheoremB}. We will illustrate Theorem \ref{prp:InitialSegment} in the case of free products and Higman-Neumann-Neumann (HNN) extensions. HNN extensions include the Baumslag-Solitar groups $\BS(n,m)$ for which a slightly different Model of multipliers satisfying Cotlar's identity would also be given, see Proposition \ref{cor:HNN2Models}.

\textbf{$\mathbf{PSL_2}(\mathbf{K})$, its lattices and open questions.}
% Models of Hilbert transforms in groups. 
% Three scenarios: 
%   1) Cocycles, 
%   2) Groups acting on graphs or trees, 
%   3) Groups acting on symmetric manifolds
% Difficulties with idempotent multipliers. The convex case: Fefferman and recent results of Batemann (brief)
% The work of Javier on "bad directions" for Hilbert trasforms coming from finite dimensional 1-cocyles
% The work of Mei-Ricard for free products and consequences (maybe even cite their recent work on "initial words")
%\textbf{Natural models: $\sG$ acting on a split space.}
Natural models of Hilbert transforms on a group $\sG$ often appear via the following straightforward idea. Let $\X$ be a geometric object on which $\sG$ acts and assume $\X$ contains a \textit{barrier} $\F \subseteq \X$ such that $\X \setminus \F$ is divided into two separated halves $\X \setminus \F = \X_+ \sqcup \X_-$. Then, given $x_0 \in \F$, a symbol $m$ can be defined as
\[
  m(g) = \1_{\X_+}(g \cdot x_0) - \1_{\X_{-}}(g \cdot x_0).
\]
When $\X = \RR$ and $\F = \{0\}$, the group multiplier induced by $\RR \acts \X$ coincides with the sign function. Thus, these symbols are natural generalizations of the Hilbert transform and we will refer to them as such.
Important instances of this include:
%
%The same idea can be generalized in a straight-forward manner to more than $2$ components
\begin{enumerate}[label={\rm \textbf{(\arabic*)}}, ref={\rm (\arabic*)}]
  \item \label{itm:Model.1} \textbf{Hilbert space model.}
  Let $\X = \H$ be a (real) Hilbert space in which $\sG$ acts by affine isometries $\pi(g)$. 
  These isometries are given by $\xi \mapsto \alpha(g) \xi + \beta(g)$, 
  where $\alpha(g)$ is an orthogonal transformation. Let  $\F = \langle v \rangle^\perp$ be the 
  codimension $1$ subspace of vectors perpendicular to $v \in \H \setminus \{0\}$. 
  Choosing $x_0 = 0$ gives the symbol $m(g) = \sgn \left( \langle \beta(g), v \rangle \right)$. 
  These symbols have been studied for finite dimensional $\H$ in \cite[Appendix A]{CasParPerrRic2014} 
  and \cite{ParRog2016Hilbert}.
  
  \item \label{itm:Model.2} \textbf{Manifold model.}
  Choose $\X = M$ as a $n$-dimensional Riemannian manifold in which $\sG$ acts by isometries 
  $\alpha: \sG \to \Iso(M)$ and let $\F \subseteq \X$ be a $(n-1)$-dimensional geodesic submanifold such that
  $\X \setminus \F$ has two connected components. 
  
  \item \label{itm:Model.3} \textbf{Tree model.}
  $\X = T$ being a tree on which $\sG$ acts. Choose $x_0 \in T$ to be a vertex, that we will henceforth call the \emph{root}. Then, $\X \setminus \{ x_0 \}$ is made up of $r$ connected components, with $r$ being the valence of $x_0$, that we can arrange into two families $\X_+$ and $\X_-$. This is an instance of the model described in Theorem \ref{thm:TheoremB} above.
\end{enumerate} 

%Observe also that in the last point we can trivially define $m$ to take more than two valued without changing the essence of the multipliers
It is very interesting to point out that, in many examples, the same idempotent Fourier multiplier on a group can be obtained from more than one of the three different models above. Here we will illustrate that phenomenon with the continuous groups $\PSL_2(\RR)$ and $\PSL_2(\CC)$, which will explain part of our original motivation.
% The case of SL_2_(RR) and SL_2(CC): cocycles vs Riemannian actions on hyperbolic spaces
% Open problem of $K$ -right invariant multipliers
%\textbf{Hilbert transforms from actions on hyperbolic spaces.}
%\marginpar{{\color{red} I've reduced the length of the hyperbolic space definition}}
Let $\SL_2(\KK)$ be the group of $2 \times 2$ matrices of determinant $1$ with entries over a field $\KK$ that in our examples will be $\RR$ or $\CC$. $\PSL_2(\KK)$ will denote the quotient of $\SL_2(\KK)$ by scalar matrices $\{\pm \Id\}$. Both groups, $\PSL_2(\RR)$ and $\PSL_2(\CC)$, act faithfully and transitively by isometries on the real hyperbolic spaces of dimension $2$ and $3$, $\PSL_2(\RR) \acts \HH^2$ and $\PSL_2(\CC) \acts \HH^3$ ---which we will identify with their upper half plane and upper half space models. 
%Observe also that, in the complex case $\PSL_2(\CC) \cong \PGL_2(\CC) = \GL_2(\CC) / \CC^\times$, the group of M\"obius transformations. 
%These groups have as maximal compact subgroups $\PSO(2) =  \SO(2) / \{ \pm \Id \}$ and $\PSU(2) \, = \, \SU(2) / \{ \pm \Id \}$ respectively.
%Their right quotients are homogeneous spaces that can be naturally identified with the hyperbolic upper half plane $\HH^2$ and the hyperbolic upper half space $\HH^3$
%\begin{eqnarray*}
%  \PSL_2(\RR) / \PSO(2) & \cong & \HH^2 \\
%  \PSL_2(\CC) / \PSU(2) & \cong & \HH^3
%\end{eqnarray*}
%with $\PSL_2(\RR)$, respectively $\PSL_2(\CC)$, acting by isometries. 
Let us denote the coordinates of $\HH^2$ by $(x,y)$ and the coordinates of $\HH^3$ by $(x,y,z)$. We can take the geodesic $\{x = 0\} \subseteq \HH^2$ as separating space in the first example. A calculation yields that the Hilbert transform in the sense of the \hyperref[itm:Model.2]{manifold model} is
\begin{nsecequation}
  %\tag{{HT$_\RR$}}
  \label{eq:HilbertContSLR}
  m
  \Big(
  \underbrace{
    \begin{matrix}
      a_{11} & a_{12} \\
      a_{21} & a_{22}
    \end{matrix}
  }_{g}
  \Big)
  \, = \, 
  \sgn \big( \Re \{ g \cdot i \} \big) 
  \, = \, 
  \sgn \big( a_{11} a_{21} + a_{12}a_{22} \big),
\end{nsecequation}
where $g$ is the class $\pm[a_{i, j}]_{i,j}$. This multiplier can be related to the other two models. Indeed, for the \hyperref[itm:Model.1]{Hilbert space model}, it is possible to construct a metrically proper $1$-cocycle $\beta: \PSL_2(\KK) \to \H$ into an infinite dimensional Hilbert space $\H$ and choose a unit vector $u \in \H$ such that $m(g) = \sgn \langle \beta(g), u \rangle$, see \cite{ErvenFal1981, CheCowJoJulgVal2001}. While the group $\PSL_2(\RR)$ is continuous, and thus it is unable to act on trees in an interesting way, the \hyperref[itm:Model.3]{tree model} interpretation is indeed available for the restriction of \eqref{eq:HilbertContSLR} to $\PSL_2(\ZZ)$. The key observation is that
\[
  \PSL_2(\ZZ) \cong \ZZ_2 \ast \ZZ_3.
\]
This free-product decomposition yields an action of $\PSL_2(\ZZ)$ on its Bass-Serre tree with respect to which the multiplier $m{|}_{\PSL_2(\ZZ)}$ can be recovered, see Figure \ref{fig:SL2Tessel}.

For the complex case, the separating subspace is given by the $2$-dimensional geodesic submanifold $\{ x = 0 \} \subseteq \HH^2$, which readily gives that
\begin{nsecequation}
  %\tag{{HT$_\CC$}}
  \label{eq:HilbertContSLC}
  m
  \Big(
  \underbrace{
    \begin{matrix}
      z_{11} & z_{12} \\
      z_{21} & z_{22}
    \end{matrix}
  }_{g}
  \Big)
  \, = \,
  \sgn \big( \Re \{ p( g \cdot x_0 ) \} \big)
  \, = \, 
  \sgn \big( \Re \{ z_{11} \overline{z}_{21} + z_{12} \overline{z}_{22} \} \big),
\end{nsecequation}
where $x_0 = (0,1) \in \CC \times \RR_+$, $p(z,r) = z$ and $g = \pm [z_{i, j}]_{i,j}$. In this case, the relationship with the other two models is more involved. Nevertheless, it is still possible to describe the multiplier \eqref{eq:HilbertContSLC} above in terms of proper infinite-dimensional $1$-cocycles with respect to a natural direction $u$. 
For the \hyperref[itm:Model.3]{tree model} the situation is a lot more contentious. Indeed, let $\O_{-d}$ be the ring of integers of the algebraic field $\QQ(\sqrt{-d})$, where $d$ is a square-free integer. The lattices $\PSL_2(\O_{-d}) \subseteq \PSL_2(\CC)$ are the \emph{Bianchi groups}. It is known that all of them except for $d = 3$ admit nontrivial actions on trees, see \cite{FrohmanFine1988Bianchi}. Indeed, for $d=1$ this yield the following, quite involved, isomorphism
\[
  \PSL_2(\O_{-1})  \, \cong \, 
  \big( S_3 \, {\ast}_{\, \ZZ_3} \, A_4 \big)
  \ast_{\PSL_2(\ZZ)}
  \big( S_3 \, \ast_{\, \ZZ_2} \, V \big),
\]
where $S_n$ are the permutation groups, $A_n$ are the alternating groups and $V$ is the Klein $4$ group, see \cite[Theorem 2.1.(i)]{FrohmanFine1988Bianchi}. It is possible that $m$, when restricted to $\PSL_2(\O_{-d})$, may have an expression in terms of a nontrivial action on a tree. Nevertheless, the complexity of the amalgamated free product decompositions obtained make it a difficult approach to work with. On the other hand, the strength of our characterization in Theorem \ref{thm:TheoremA} allows us to prove the boundedness of $m{|}_{\PSL_2(\O_{-d})}$ directly. We have also verified that $d = 1$ is the only Bianchi group for which the restriction of \eqref{eq:HilbertContSLC} satisfies \eqref{eq:CotlarFactor}.

\begin{ltheorem}
  \label{thm:BoundednessBianchi}
  Let $\sG = \PSL_2(\O_{-1}) \subseteq \PSL_2(\CC)$ and $m: \PSL_2(\O_{-1}) \to \CC$ be the function given by
  \[
    m 
    \begin{pmatrix}
      a_{11} + i b_{11} & a_{12} + i b_{12} \\
      a_{21} + i b_{21} & a_{22} + i b_{22} 
    \end{pmatrix}
    \, = \,
    \sgn \big( a_{11} a_{21} + b_{11} b_{21} + a_{12} a_{22} + b_{12} b_{22} \big).
  \]
  Then, $m$ satisfies that 
  \[
    \big\| T_m: L_p(\VN \, \PSL_2(\O_{-1})) \to L_p(\VN \, \PSL_2(\O_{-1})) \big\| 
    \, \lesssim \, \Big( \frac{p^2}{p - 1} \Big)^\beta
    \quad \mbox{ with } \beta = \log_2(1+\sqrt{2}).
  \]
\end{ltheorem}

This leaves open whether $m{|}_{\Gamma}$ is bounded for lattices other than $\PSL_2(\O_{-1})$. 
In the same spirit, the boundedness of \eqref{eq:HilbertContSLR} and \eqref{eq:HilbertContSLC} over the whole group is an natural problem that we leave open. 

\begin{lproblem}\ \normalfont
  \label{prb:LpBoundContSL2}
  % Problem for SL2(R)
  \begin{enumerate}[label={\rm \textbf{(A.\arabic*)}}, ref={\rm \textbf{(A.\arabic*)}}, leftmargin=1.5cm]
    \item \label{itm:LpBoundContSLR.1} Let $m$ be as \eqref{eq:HilbertContSLR}. Is 
    $T_m: L_p(\VN \, \PSL_2 (\RR)) \to L_p(\VN \, \PSL_2 (\RR))$ bounded?
    \item \label{itm:LpBoundContSLR.2} Let $m$ be as \eqref{eq:HilbertContSLC}. Is 
    $T_m: L_p(\VN \, \PSL_2 (\CC)) \to L_p(\VN \, \PSL_2 (\CC))$ bounded?
  \end{enumerate}
\end{lproblem}

In the classical case of $\sG = \RR$, the boundedness of the Hilbert transform can be obtained from smooth multiplier results, in particular it satisfies the hypothesis of both the H\"ormander-Mikhlin and Marcinkiewicz theorems \cite{Duoan2001Book}. Therefore, Problem \ref{prb:LpBoundContSL2} seems closely related to the question of whether smoothness conditions of a function $\widetilde{m}: \HH^2 \to \CC$ yield $L_p$-boundedness of the lifted multiplier $m(g) = \widetilde{m}(g \cdot i)$. 
Results that point in that direction have already appeared in the literature. For instance, a result for local smooth Fourier multipliers in $\SL_2(\RR)$ features in \cite{ParcetSLn2018}. In the case of $S_p$-bounded Schur multipliers, results for global Hörmander-Mikhlin Schur multipliers have been obtained in \cite{CondeGonParcetTablate2022I,CondeGonParcetTablate2022II}. This smooth multiplier approach to Problem \ref{prb:LpBoundContSL2} above presents two main obstacles. The first is that ---contrary to the results in \cite{CondeGonParcetTablate2022II}--- the singularity isn't located in a single point, instead it is a codimension $1$ subset containing the stabilizer of a point in $\HH^2$. The second is that currently available techniques for the passage from Schur to Fourier multipliers require to work with compactly supported multipliers or in amenable groups, see \cite{ParcetSLn2018, CondeGonParcetTablate2022II}. 

\textbf{Foreword.} Since the release of the first draft of this work, Problem \ref{prb:LpBoundContSL2} has been solved in the negative by a surprising generalization Feffermann's Kakeya construction for idempotent Schur multipliers \cite{Parcet2024idempotentschur}. 

\section{Cotlar identities and multipliers} \label{S1}

\textbf{Noncommutative integration.} Throughout this text we will use liberally \emph{noncommutative integration theory} and the theory of noncommutative $L_p$-spaces. Let $\M \subseteq \B(\H)$ be a \emph{von neumann algebra}  admitting a \emph{normal semifinite and faithful tracial weight} $\tau: \M_+ \to [0,\infty]$ that we will henceforth just refer to as a \emph{n.s.f trace}. It is possible to construct the noncommutative $L_p$-spaces associated to $(\M,\tau)$ as the subset of $\tau$-measurable operators $L_p(\M,\tau) \subseteq L_0(\M,\tau)$ satisfying that
\[
  f \in L_p(\M,\tau) \quad \equivalent \quad \| f \|_p := \tau \big( |f|^p \big)^\frac1{p} < \infty.
\]
This theory, which goes back all the way to Dixmier and Segal \cite{Dixmier1953Formes, Segal1953}, is already well understood and the interested reader can consult it in \cite{Terp1981lp, PiXu2003, GoldsteinLabouschange2020}.

Let $\sG$ be a locally compact group that we will throughout the text assume to be second countable, and let $L_2(\sG)$ be its $L_2$-space with respect to the \emph{left Haar measure} $\mu$ \cite{Foll1995}. As usual, we will denote by $\lambda: \sG \to \U(L_2 (\sG))$ the \emph{left regular representation}, which is the unitary representation $g \mapsto \lambda_g$ that acts by sending $\xi(h)$ to $\xi(g^{-1} h)$. The \emph{left regular von Neumann algebra} $\VN \sG \subseteq \B(L_2 (\sG))$ of $\sG$ is given by
\[
  \overline{ \lambda[L_1(\sG)]^{\wast}} 
  \, = \,
  \overline{\Big\{ \lambda(\varphi) := \int_\sG \varphi(g) \, \lambda_g \, d \mu(g) : \varphi \in L_1(\sG) \Big\}^\wast} 
  \, \subseteq \, 
  \B(L_2( \sG)).
\]
This von Neumann algebra admits a normal, semifinite and faithful weight $\tau: \VN \sG_+ \to [0,\infty]$ that satisfies the Plancherel identity, meaning that $\varphi \mapsto \lambda(\varphi)$ extends to a unitary isometry from $L_2(\sG)$ to Gelfand-Neumark-Segal space $L_2(\VN \sG;\tau)$. This weight is usually referred as the \emph{Plancherel weight} \cite[Chapter 7]{Ped1979}. The weight $\tau$ is a n.s.f trace precisely when $\sG$ is unimodular. Thus, we will work in the natural setting of unimodular groups and refer to $\tau$ as the \emph{Plancherel trace}. In this context, the Plancherel trace is given by
\[
  \tau \left( \int_\sG \varphi(g) \, \lambda_g \, d \mu(g) \right) = \varphi(e), 
  \quad \mbox{ for every } \, \varphi \in C_c(\sG) \ast C_c(\sG).
\]
We will denote the noncommutative $L_p$-spaces associated to $\tau$ simply by $L_p(\VN \sG)$. By analogy with the classical Fourier transform, we will denote by $\widehat{f}(g)$ the value $\tau(\lambda_g^\ast f)$, which is well defined whenever $f \in L_1(\VN \sG)$. It is also worth noticing that, by Plancherel's theorem, the map $f \mapsto \widehat{f}$ is well defined for any $f \in L_2(\VN \sG)$. In fact, we have that $f \in L_2(\VN \sG)$ can be expressed as
\begin{equation}
    \label{eq:Plancherel}
    f = \lambda(\widehat{f}) := \int_\sG \widehat{f}(g) \, \lambda_g \, d \mu(g),
\end{equation}
where the integral on the right hand side converges in the $L_2$ norm.

\textbf{Conditional expectations.} Let $\N \subseteq \M$ be a von Neumann subalgebra of $\M$, that is a $\ast$-subalgebra that is also ultraweakly closed. If $\tau$ is a n.s.f. trace over $\M$ and $\tau{|}_\N$ is still semifinite, then it is easy to see that the inclusion $\iota: L_1(\N) \into L_1(\M)$ is isometric and its dual map is a normal (i.e: ultraweakly continuous) \emph{conditional expectation} $\EE: \M \to \N \subseteq \M$. By conditional expectation we mean a \emph{unital and completely positive} map such that $\EE{|}_\N = \Id_\N$. Recall that conditional expectations satisfy that  $\EE \circ \EE = \EE$ and that, by Tomiyama's Theorem, see \cite[Theorem 1.5.10]{BroO2008}, $\EE$ is automatically $\N$-bimodular.

Let $\sG_0 \subseteq \sG$ be two groups such that $\sG_0$ is open inside $\sG$. Then $\sG_0$ is unimodular if $\sG$ is. Furthermore, the Plancherel trace $\tau_{\sG_0}$ of $\VN \sG_0$ coincides with the Plancherel trace of $\VN \sG$ restricted to $\VN \sG_0$. Therefore there is a normal and trace-preserving conditional expectation $\EE: \VN \sG \to \VN \sG_0 \subseteq \VN \sG$ that is given by the Fourier multiplier associated to $\1_{\sG_0}$ ie:
\[
  \EE(f)
  \, = \,
  \EE \left( \int_\sG \widehat{f}(g) \, \lambda_g \, d \mu(g) \right) 
  \, = \,
  \int_\sG \1_{\sG_0}(g) \, \widehat{f}(g) \, \lambda_g \, d \mu(g).
\]
The fact that $\EE$ is trace preserving allows us to extend $\EE$ as a contraction to all the $L_p$-spaces $1\leq p \leq \infty$, $\EE:L_p(\VN \sG) \to L_p(\VN \sG_0) \subseteq L_p(\VN \sG)$. 

\textbf{Noncommutative Cotlar identities.} Most of this section up until the closed-formula characterization of the Cotlar identity in the proof of Theorem \ref{thm:TheoremA} follows closely the results obtained by Mei and Ricard and it is included here for the sake of {completeness}.

%First, let us assume that $\A \subseteq \M$ is a unital $\ast$-subalgebra such that $\A \cap L_p(\M)$ is norm dense in $L_p(\M)$ for every $2 \leq p < \infty$. When dealing with a complemented von Neumann subalgebra $\N \subseteq \M$ we will assume that $\A \cap L_p(\N)$ is again norm dense in $L_p(\N)$ and that $\A \cap \N$ is ultraweakly dense in $\N$. Whenever we say that an operator $H: \A \to \A$ is $\N$-modular, we mean that it is modular with respect to $\A \cap \N$.
%{\color{red}In the following definition $\M$ will be a (semifinite) von Neumann algebra and $\N \subseteq \M$ a unital von Neumann subalgebra. In the second part we will need a dense class $\A \subseteq \M$ such that
%\begin{multicols}{2}
%    \begin{enumerate}[label={\rm {(\roman*)}}, ref={\rm {(\roman*)}}, leftmargin=1cm]
%        \item \label{itm:defA1} $\A \subseteq L_2(\M) \cap \M$,
%        \item \label{itm:defA2} $\A \cap L_p(\M) \subseteq L_p(\M)$ is dense $\forall \, 2 \leq p < \infty$,
%        \item \label{itm:defA3} $\N \, \A \, \N \subseteq \A$,
%        \item \label{itm:defA4} $\A \cap L_p(\N) \subseteq L_p(\N)$ is dense $\forall \, 2 \leq p < \infty$.
%    \end{enumerate}
%\end{multicols}
%}

\begin{definition}[{\cite[from Proposition 3.2(iv)]{MeiRicard2017FreeHilb}}] \normalfont
  \label{def:CotlarNC}
  Let $\M$, $\N$ be as above, $\EE: \M \to \N$ be the conditional expectation and 
  let $H:L_2(\M) \to L_2(\M)$ be a bounded operator
  \begin{enumerate}[label={\rm \textbf{(\roman*)}}, ref={\rm {(\roman*)}}, leftmargin=1.5cm]
    \item $H$ satisfies the (non-relative) Cotlar identity if
    \begin{equation}
      \tag{Cotlar$_\mathrm{nr}$}
      \label{eq:CotlarNoExp}
      H(f) \, H(f)^\ast
      \, = \, 
      H \big( f \, H(f)^\ast \big) + H \big( f \, H(f)^\ast \big)^\ast - H \big( H(f f^\ast) \big)^\ast, 
    \end{equation}
    for every $f \in \M \cap L_2(\M)$.
    \item $H$ is said to satisfy Cotlar identity (relative to $\N$) if
      \begin{equation}
        \tag{Cotlar$_{\EE^\perp}$}
        \label{eq:CotlarNC}
        \EE^\perp \big[ H(f) \, H(f)^\ast \big] 
        \, = \, 
        \EE^\perp \left[ H \big( f \, H(f)^\ast \big) + H \big( f \, H(f)^\ast \big)^\ast  - H \big( H(f f^\ast)^\ast \big) \right],
      \end{equation}
      for every $f \in L_2(\M) \cap \M$, where $\EE^\perp = (\Id - \EE)$.
  \end{enumerate}
\end{definition}

We are assuming that $H$ is bounded in $L_2(\M)$ \textit{a priori}, this is not a restrictive imposition since it holds trivially in the case of Fourier multipliers. Observe as well that that, although both Cotlar identities require applying $H$ to products of functions that therefore may not be in $L_2(\M)$, the fact that $f \in L_2(\M) \cap \M$ and that $L_2(\M)$ is stable by multiplications by $\M$ makes both equations meaningful. 
%In the relative case, we need to impose conditions \ref{itm:defA1}---\ref{itm:defA4} for reasons that will become apparent in the prove of Lemma \ref{lem:ExpectationBdd}. 

We will need the following lemma. Notice that if $H: L_2(\M) \to L_2(\M)$ is left $\N$-modular, then its restriction to $L_2(\N)$ composed with the expectation $\EE:L_2(\M) \to L_2(\N)$ gives a map
\[
  \EE \, H \, {|}_{L_2(\N)}:L_2(\N) \to L_2(\N)
\]
that is both bounded with norm equal to that of $H$ and and left $\N$-modular. But all left $\N$-modular maps in $L_2(\N)$ are right multiplicators by an element in $\N$ of norm equal to the operator norm. As such, $\EE \, H \, {|}_{L_2(\N) \cap L_p(N)}$ extends to a bounded operator in $L_p(\N)$ and
\begin{equation}
    \label{eq:ModularExtenstion}
    \big\| \EE \, H \, {|}_{L_p(\N)}:L_p(\N) \to L_p(\N) \big\| \leq \big\| H: L_2(\M) \to L_2(\M) \big\|.
\end{equation}
With this observation at hand we can proceed to prove the following lemma.

\begin{lemma}[{\cite[Proposition 3.4]{MeiRicard2017FreeHilb}}]
  \label{lem:ExpectationBdd}
  Let $\M$, $\N$ and $\EE$ be as above and let $H:L_2(\M) \to L_2(\M)$ be a bounded and left $\N$-modular map. For every $f \in L_2(\M) \cap \M$, it holds that
  \begin{equation}
    \label{eq:ExpectationBdd.1}
    \EE \big[ H(f) H(f)^\ast \big]
      \leq \left\| H: L_2(\M) \to L_2(\M) \right\|^2 \EE \big[ f f^\ast \big].
  \end{equation}
  Furthermore, if $\EE H = H \EE$, we have that for every $1 \leq p \leq \infty$
  \begin{eqnarray}
    \big\| \EE \big[ H \big(f \, H(f)^\ast \big) \big] \big\|_p
      & \leq & \big\| H: L_2(\M) \to L_2(\M) \big\|^2
               \, \left\| \EE \big[ f f^\ast \big] \right\|_p \label{eq:ExpectationBdd.2} \\
    \big\| \EE \left[ H \big( H(f f^\ast)^\ast \big) \right] \big\|_p
      & \leq & \big\| H: L_2(\M) \to L_2(\M) \big\|^2 \, \left\| \EE \big[ f f^\ast \big] \right\|_p \label{eq:ExpectationBdd.4}
  \end{eqnarray}
\end{lemma} 
\begin{proof}
  All of the points are elementary. For \eqref{eq:ExpectationBdd.1} first notice that every 
  state of $\N$ is of the form $f \mapsto \tau( \delta \, f)$, where $\delta$ 
  is a positive element of norm $1$ in the space $L_1(\N)$. Decomposing it as $\delta = \delta^\frac12 \delta^\frac12$ gives 
  \begin{eqnarray*}
    \tau \left\{ \delta \, \EE \big[ H(f) H(f)^\ast \big] \right\}
      & = & \tau \left\{ \delta^\frac12 \EE \big[ H(f) H(f)^\ast \big] \delta^\frac12 \right\} \\
      & = & \tau \left\{ \EE \big[ \delta^\frac12 H(f) H(f)^\ast \delta^\frac12 \big] \right\} \\
      & = & \tau \left\{ \EE \big[ H(\delta^\frac12 f) H(\delta^\frac12 f)^\ast  \big] \right\} \\
      & = & \tau \left\{ H(\delta^\frac12 f) H(\delta^\frac12 f)^\ast \right\} \\
      & \leq & \big\| H:L_2(\M) \to L_2(\M) \big\|^2 
               \, \tau \left\{ (\delta^\frac12 f)(\delta^\frac12 f)^\ast \right\} \\
      & = & \big\| H:L_2(\M) \to L_2(\M) \big\|^2 \, 
            (\tau \circ \EE) \left\{ (\delta^\frac12 f)(\delta^\frac12 f)^\ast \right\} \\
      & = & \big\| H:L_2(\M) \to L_2(\M) \big\|^2 
            \, \tau \left\{ \delta \,  \EE \big[ f f^\ast \big] \right\}.
  \end{eqnarray*}
  Since this is true for every state, the operator inequality \eqref{eq:ExpectationBdd.1} holds.
  
  For \eqref{eq:ExpectationBdd.2} we use that $\EE H = H \EE$ to rewrite $\EE \left[ H \big(f \, H(f)^\ast \big) \right]$ as $(\EE H  \EE) \circ \EE \left[ \big(f \, H(f)^\ast \big) \right]$. 
  The operator norm on $\EE H \EE: L_p(\M) \to L_p(\M)$ is bounded by that of $\EE H {|}_{L_p(\N)}: L_p(\N) \to L_p(\N)$, which is bounded by the norm of $H$ by \eqref{eq:ModularExtenstion}.
  To estimate the term $\EE \left[ \big(f \, H(f)^\ast \big) \right]$ we will use the following version of Hölder's inequality \cite[Inequality (2.1)]{Jun2002Doob}
  \[
    \big\| \EE \big[ f \, g^\ast \big] \big\|_p 
    \leq 
    \big\| \EE \big[ f \, f^\ast \big]^\frac12 \big\|_r \, \big\| \EE \big[ g \, g^\ast \big]^\frac12 \big\|_s 
    \quad \mbox{ when } \quad \frac1{p} = \frac1{r} + \frac1{s},
  \]
  with $r = s = 2p$ and $g = H(f)$ to obtain that
  \[
    \big\| \EE \left[ f \, H(f)^\ast \right] \big\|_p
    \leq 
    \big\| \EE \left[ f \, f^\ast \right] \big\|_p^\frac12 \, \big\| \EE \left[ H(f) \, H(f)^\ast \right] \big\|_p^\frac12.
  \]
  Applying the inequality in \eqref{eq:ExpectationBdd.1} gives the result. Identity \eqref{eq:ExpectationBdd.4} follows immediately after using two times the fact that $H$ and $\EE$ commutes and that $\EE H \EE$ has a norm in $L_p$ bounded by the norm in $L_2$ of $H$ by \eqref{eq:ModularExtenstion}.
\end{proof}

%\begin{remark} \normalfont
%  \label{rmk:whyUnit}
%  The reason why we need to include the unit in the algebra $\A$ is in order to make sense of \eqref{eq:ExpectationBdd.Mod}. In many natural examples, like in the classical Hilbert transform \eqref{eq:HT}, the use of a principal value in the integral automatically sends the constant functions to $0$, this trivially including them in the domain of definition. Nevertheless, definitions based on functions of the Schwartz class may be undefined over constants. In order to apply the framework of this Section to such operators it is necessary to extend $H$ to $\CC \1$ in a way that preserves the left $\N$-modularity. This will be trivial in cases in which $\N = \CC \1$ or ---when dealing with multipliers--- if $H(\1)$ is chosen to be $m(e) \1$ where $m(e)$ is the essentially unique value of $m$ over $\sG_0$. 
%\end{remark}

We can now prove the following extrapolation result.

\begin{proposition}[{\cite[Theorem 3.5]{MeiRicard2017FreeHilb}}]
  \label{prp:LpExtrapolation}
  Let $\N \subseteq \M$ be as before and let $H:L_2(\M) \to L_2(\M)$ be a left $\N$-modular operator commuting with $\EE: \M \to \N$. If $H$ satisfies \eqref{eq:CotlarNC} then $\forall \, 2 \leq p < \infty$
  \begin{equation}
    \label{eq:Extrapolation}
    \big\| H: L_p(\M) \to L_p(\M) \big\|
    \, \lesssim \, 
    p^\beta \, \big\| H: L_2(\M) \to L_2(\M) \big\|, \; \text{ where } \beta = \log_2(1 + \sqrt{2}).
  \end{equation}
  Similarly, if $H:L_2(\M) \to L_2(\M)$ is a general bounded linear map that satisfies the non-relative Cotlar identity \eqref{eq:CotlarNoExp}, then the same extrapolation inequality \eqref{eq:Extrapolation}.
\end{proposition}

\begin{proof}
  First, let us denote the operator norm on $L_p$ of $H$ by $c_p := \big\| H: L_p(\M) \to L_p(\M)\big\|$. 
  We are going to proceed by induction, assuming that $c_p < \infty$ to prove that $c_{2p} < \infty$.
  Choose $f \in \M \cap L_2(\M)$ with $\| f \|_{2p} \leq 1$ and notice that
  \begin{eqnarray}
    \| H(f) \|_{2p}^2 
      & =    & \| H(f) H(f)^\ast \|_{p} \nonumber\\
      & \leq & \big\| \EE \big[ H(f) H(f)^\ast \big] \big\|_{p} 
               + \big\| \EE^\perp \big[ H(f) H(f)^\ast \big] \big\|_{p} \nonumber \\
      & \leq & c_2^2 \, \big\| \EE \big[ f f^\ast \big] \big\|_{p} 
               + \big\| \EE^\perp \big[ H \big(f \, H(f)^\ast \big) 
               + H \big(f \, H(f)^\ast \big)^\ast 
               - H \big( H(f f^\ast)^\ast \big) \big] \big\|_{p} \label{eq:UseofLemma1}\\
      & \leq & c_2^2 \, \| f \|_{2p}^2 
               + \big\| H \big(f \, H(f)^\ast \big) \big\|_p
               + \big\| H \big(f \, H(f)^\ast \big)^\ast \big\|_p
               + \big\| H \big( H(f f^\ast)^\ast \big) \big\|_{p} \nonumber \\
          &  & + \big\| \EE \big[ H \big(f \, H(f)^\ast \big) \big] \big\|_{p} 
               + \big\| \EE \big[ H \big(f \, H(f)^\ast \big)^\ast  \big] \big\|_{p} 
               + \big\| \EE \big[ H \big( H(f f^\ast)^\ast \big) \big] \big\|_{p}  \label{eq:UseofLemma2}\\
      & \leq & 2 \, c_p \, \| H(f) \|_{2p} \|f\|_{2p} + c_p^2 \| f \|_{2p} + 4 \, c_2^2 \| f \|_{2p}^2 \nonumber \\
      & \leq & 2 \, c_p \, \| H(f) \|_{2p} + c_p^2 + 4 \, c_2^2  \label{eq:UseofLemma3}
  \end{eqnarray}
We have used \eqref{eq:CotlarNC} in the second term of the sum of \eqref{eq:UseofLemma1} and estimate \eqref{eq:ExpectationBdd.1} of Lemma \ref{lem:ExpectationBdd} in the first.
  To pass from \eqref{eq:UseofLemma2} to \eqref{eq:UseofLemma3} we have used the other two identities of Lemma \ref{lem:ExpectationBdd}. Notice that \eqref{eq:UseofLemma3} is a quadratic inequality of the form $0 \leq -t^2 + 2 c_p \, t + c_p^2 + 4 \, c_2^2$, where $t = \| H(f)\|_{2p}$. Since the leading term in negative, this implies a bound of $\|H(f)\|_{2p}$ in terms of $c_p$ and $c_2$ only. Taking supremum over $f$ and using the norm density of $\M \cap L_{2}(\M)$ in $L_{2p}(\M)$ implies that $c_{2p}$ is finite. More explicitly, setting $a_p = c_p / c_2$ gives the recursive inequality
  \begin{equation}
    \label{eq:Recurrence1}
    a_{2p}^2 \leq 2 \, a_p \, a_{2p} + a_p^2 + 4.
  \end{equation}
  Adding $a_{2p}^2$ to both sides in order to complete squares gives
  \[
    2 a_{2p}^2 
    \, \leq \, a_{2p}^2 + 2 \, a_p \, a_{2p} + a_p^2 + 4
    \, = \, (a_{2p} + a_p)^2 + 4 
    \, \leq \, (a_{2p} + a_p + 2)^2.
  \]
  After taking square roots and recursively applying the inequality above, the following is obtained
  \[
    a_{2^k} \, \leq \, C  \, (1 + \sqrt{2})^{k-1}, 
    \quad \mbox{ with } \quad C \leq 3 + \sqrt{2}.
  \]
  This, together with Marcinkiewicz interpolation for intermediate values of $p$,
  gives the desired inequality. The non-relative case works similarly, but the $4$ extra terms that give $4 \, c_2^2$ in the right hand side of inequality \eqref{eq:UseofLemma3} do not appear, which only affects the absolute constant $C$ in the final inequality.
\end{proof}

%% Explain the group setting
%% Explain that by default a Cotlar identity is relative to $\{e\} if the group is dicrete and non-relative if the group is not discrete
Our source of examples for the extrapolation theorem above is taken when $\M = \VN \sG$ and $H=T_m$ is a Fourier multiplier. We will work with Cotlar identities relative to subalgebras induced by subgroups $\sG_0 \subseteq \sG$. In that setting it holds that there is a normal and trace-preserving conditional expectation $\EE: \VN \sG \to \VN \sG_0$ if and only if $\sG_0 \subseteq \sG$ is open. Thus, our relative Cotlar identities would be taken only with respect to open subgroups. In general, given a group $\sG$, when speaking about the Cotlar identity for a multiplier $T_m$ ---without specifying any subgroup--- we will mean, that it satisfies the non-relative Cotlar identity \eqref{eq:CotlarNoExp} when the group $\sG$ is continuous  i.e. $\mu(\{e\}) = 0$, while in the discrete case $\mu(\{e\}) \neq 0$, we will instead mean that it satisfies \eqref{eq:CotlarNC} with respect to the subgroup $\{e\}$.

We are now going to prove the equivalence between Cotlar's identity \eqref{eq:CotlarNC} and the closed formula in Theorem \ref{thm:TheoremA}.

\begin{theorem}
  \label{thm:ClosedFormulaCotlar}
  Let $\sG_0 \subseteq \sG$ be an open subgroup of $\sG$ and $m: \sG \to \CC$ be a bounded function. 
  The following properties are equivalent
  \begin{enumerate}[label={\rm \textbf{(\roman*)}}, ref={\rm {(\roman*)}}]
    \item \label{itm:ClosedFormulaCotlar.1}
    $T_m$ satisfies \eqref{eq:CotlarNC}.
    
    \item \label{itm:ClosedFormulaCotlar.2}
    The function $m$ satisfies that 
    $\; \big( m(g^{-1} ) - m(h) \big) \, \big( m(gh) - m(g) \big) \, = \, 0$,
    for almost every $g \in \sG \setminus \sG_0, h \in \sG$.
  \end{enumerate}
\end{theorem}

\begin{proof}
  Expanding \eqref{eq:CotlarNC} for $T_m$ gives
  \begin{equation*}
    0 \, = \, \underbrace{\EE^\perp \Big[ T_m(f) \, T_m(f)^\ast \Big]}_{\mathrm{(I)}}
              - \underbrace{\EE^\perp \Big[ T_m \big( f \, T_m(f)^\ast \big) \Big]}_{\mathrm{(II)}}
              - \underbrace{\EE^\perp \Big[ T_m \big( f \, T_m(f)^\ast \big)^\ast \Big]}_{\mathrm{(III)}}
              + \underbrace{\EE^\perp \Big[ T_m \big( T_m(f f^\ast)^\ast \big) \Big]}_{\mathrm{(IV)}}.
  \end{equation*}
  Now, elementary computations yield that
  \begin{eqnarray*}
    \mathrm{(I)}
      & = & \int_{\sG \setminus \sG_0} \int_\sG \widehat{f}(g h) \, 
            \overline{\widehat{f}(h)} \, m(gh) \, \overline{m(h)} \, \lambda_g \, d\mu(h) \, d\mu(g)\\
    \mathrm{(II)}   
      & = & \int_{\sG \setminus \sG_0} \int_\sG \widehat{f}(g h) \, 
            \overline{\widehat{f}(h)} \, m(g) \, \overline{m(h)} \, \lambda_g \, d\mu(h) \, d\mu(g)\\
    \mathrm{(III)}
      & = & \int_{\sG \setminus \sG_0} \int_\sG \widehat{f}(g h) \, 
            \overline{\widehat{f}(h)}  \, m(gh) \, \overline{m(g^{-1})} \, \lambda_g \, d\mu(h) \, d\mu(g)\\
    \mathrm{(IV)}
      & = & \int_{\sG \setminus \sG_0} \int_\sG \widehat{f}(g h) \, 
            \overline{\widehat{f}(h)} \, m(g) \, \overline{m(g^{-1})} \, \lambda_g \, d\mu(h) \, d\mu(g),
  \end{eqnarray*}
  where $\widehat{f}$ is given as in \eqref{eq:Plancherel}. This in turn implies, using the Plancherel theorem, that
  \[
    0 = \int_\sG \widehat{f}(g h) \, \overline{\widehat{f}(h)} \, 
        \big( m(gh) - m(g) \big) \, \big( \overline{m(h) - m(g^{-1})} \big) \, d\mu(h)
    \quad \mbox{ for almost every } g \in \sG \setminus \sG_0.
  \]
  Obviously, if the factor $G_g(h) = ( m(gh) - m(g) ) \, ( \overline{m(h) - m(g^{-1})} )$ is equal to $0$ so is the above integral and therefore \eqref{eq:CotlarNC} holds. The reciprocal is immediate in the case of discrete groups. Indeed, choose any $h_0 \in \sG$ and assume that $g \in \sG \setminus \sG_0$ is fixed. Pick $\widehat{f} = \delta_{h_0} + \delta_{g h_0}$. In order to evaluate the integral, notice that
  \[
    \widehat{f}(g h) \, \overline{\widehat{f}(h)}
    = \delta_{\{h = h_0\}} + \delta_{\{g^2=e\}} \cdot \delta_{\{h = g h_0\}}.
  \]
The term $\delta_{h=h_0}$ in the above sum gives ${G_g(h_0)}$ in the integral. The term in which $g^2=e$ and $h = g h_0$ gives $\overline{G_g(h_0)}$. Therefore, $\Re\{ G_g(h)\}=0$ for any $h \in \sG$. The imaginary part is similarly shown to be $0$.
In the case of a continuous group $\sG$ it is necessary to change $\delta_{h_0}$ with a modification of the unit.
\end{proof}

%\marginpar{{\color{blue}See the previous comment  for the non-relative case, maybe remark 1.6 can be removed?} } 

\begin{remark} \normalfont
  \label{rmk:NonrelativeCotlar}
  % You can have the same identity for the non-relative case
  Notice that Theorem \ref{thm:ClosedFormulaCotlar} works similarly in the non-relative case. In that case it holds that $T_m$ has the \eqref{eq:CotlarNoExp} if and only if $m$ satisfies the identity $( m(g^{-1} ) - m(h) ) \, ( m(gh) - m(g) ) = 0$ for almost every $g, h \in \sG$. 
\end{remark}

With all that at hand we are ready to prove Theorem \ref{thm:TheoremA}.

\begin{proof}[Proof (of Theorem \ref{thm:TheoremA})]
  Since $\sG_0$ is closed we have two situations, either it is open or of empty interior. In the first case, we have that $\mu(\sG_0) > 0$ and we have that the Formula in \ref{itm:ClosedFormulaCotlar.2} is equivalent to \eqref{eq:CotlarNC} by Theorem \ref{thm:ClosedFormulaCotlar}. Observe that, if $m$ is left-$\sG_0$ invariant, then $T_m$ is left $\VN \sG_0$-modular. Thus, we can apply Proposition \ref{prp:LpExtrapolation} to obtain the bound \eqref{eq:TheoremABound} for $p \geq 2$. Using that, by Plancherel Theorem, $\| T_m:L_2(\VN \sG) \to L_2(\VN \sG) \| = \| m \|_\infty$ we get the result for $2 \geq p$. For $1 < p < 2$ the result follows by standard duality arguments. In the case of $\sG_0$ of empty interior the proof follows similarly by using Remark \ref{rmk:NonrelativeCotlar} and the last assertion on Proposition \ref{prp:LpExtrapolation}. 
\end{proof}

%\marginpar{{\color{blue}Shall we move remark 1.7 after proposition 1.4?} } 
\begin{remark} \normalfont
  \label{rmk:NonInvariantCharacters}
  Let $\alpha: \N \to \N$ be a normal and trace preserving $\ast$-homomorphism. It is immediate that both Proposition \ref{prp:LpExtrapolation} as well as Lemma \ref{lem:ExpectationBdd} hold if we change the condition of $H$ being left $\N$-modular by that of being left $\N$-modular relative to $\alpha$, i.e.,
  \[
    H(f \, g) = \alpha(f) \, H(g), \quad \mbox{ for } f \in \N, \, g \in L_2(\M) \cap \M.
  \]
  In the case of multipliers this easy observation has deep consequences. For instance, let $\chi: \sG_0 \to \TT$ be a (multiplicative) character. It is a straightforward consequence of Fell's absorption principle that the map $\lambda_g \mapsto \chi(g) \, \lambda_g$ induces a normal and trace-preserving $\ast$-homomorphism $\alpha_\chi: \VN \sG_0 \to \VN \sG_0$. Let $H = T_m$ be a Fourier multiplier on $\VN \sG$. We have that it is left $\VN \sG_0$-modular with respect to $\alpha_\chi$, ie $H(f \, g) = \alpha_\chi(f) \, H(g)$, for every $f \in \N$ and $g \in L_2(\M) \cap \M$ iff
  \begin{equation}
    \label{eq:LeftInvariant}
    m(k \, g) =  \chi(k) \, m(g), \quad \mbox{ for every } k \in \sG_0, \, g \in \sG.
  \end{equation}
  This is specially useful when $\sG_0$ is Abelian since, in that case, every function on $\sG_0$ can be expressed as a convex combination of characters by the Fourier transform. This will be exploited in Theorem \ref{thm:InitialSegmentG0Abelian}. It would also be used in a forthcoming paper of the third-named author \cite{Runlian2023Chinese}.
\end{remark}

%\marginpar{{\color{blue}This doesn't seem to be a good place to put remark 1.8. Maybe we can remove this remark or put it somewhere else?} } 
%\begin{remark} \normalfont
%  \label{rmk:RowvsColumCotlar}
  % Row vs Column identities.
%  Observe that the von Neumann algebra $\M$ can be turned into a $W^\ast$ Hilbert module in two canonical ways, depending on the operator valued sesquilinear form $\langle \cdot, \cdot \rangle: \M \times \M \to \M$ chosen:
%  \begin{eqnarray}
%    \langle f, g \rangle_r & = & f^\ast g, \\
%    \langle f, g \rangle_c & = & g \, f^\ast,
%  \end{eqnarray}
%  which are called the \emph{row} and \emph{column} inner products respectively. We have chosen a Cotlar identity that has a hidden row character. Indeed, \eqref{eq:CotlarNC} can be reformulated as
%  \[
%    \EE^\perp \Big[ \big\langle H(f), H(f) \big\rangle_r \Big]
%    \, = \,
%    \EE^\perp \Big[ H \langle f, H(f) \rangle_r 
%    + H \langle f, H(f) \rangle_r^\ast
%    - H \big( H\langle f, f \rangle_r^\ast \big) \Big]
%  \]
%  But we could have symmetrically defined it with respect to the row product. In that case it is necessary to require $H$ to be right $\N$-modular instead of left $\N$-modular in Proposition \ref{prp:LpExtrapolation}, which, in the case of a multiplier $T_m$, will restrict $m$ to be right $\sG_0$-invariant.
%\end{remark}

\textbf{Tightening the constant.} It is known that, in the real line $\sG = \RR$ the operator $L_p$-norm of the classical Hilbert transform \eqref{eq:HT} is given by 
\[
  \big\| H: L_p(\RR) \to L_p(\RR) \big\| 
  \, = \, \max \Big\{ \tan \Big( \frac{\pi}{2 \, p} \Big), \,\, \cot\Big( \frac{\pi}{2 \, p} \Big) \Big\},
  \quad \mbox{ for } 1 < p < \infty
\] 
see \cite{Pichorides1972} or \cite{Grafakos1997Constant} for a simplified proof. These constants grow asymptotically like $p$ as $p \to \infty$ and like $1/(p-1)$ as $p \to 1^{+}$, and those are the growth orders that we conjecture optimal in the noncommutative case as well. 
An interesting observation, originally made by Gokhberg and Krupnik in the classical case \cite{GokhbergKrupnik1968} is that Cotlar's identity in the real line gives the optimal order of growth for the constant in terms of $p$. Indeed, in the classical case, the fact that $H^2 = - \Id$ yields a recurrence relation of the form
\begin{equation}
  \label{eq:RecurrenceClassical}
  c_{2p}^2 
  \, \leq \,
  2 c_p \, c_{2 p} + 1
\end{equation}
instead of \eqref{eq:Recurrence1}. The lack of a term depending on $c_p^2$ gives a decisively smaller bound. Solving the quadratic inequality in \eqref{eq:RecurrenceClassical}, gives
\[
  c_{2p} \leq c_p + \sqrt{c_p^2 + 1}
\]
and that results, after applying duality and interpolation, in the optimal growth order for the constant. 

In the noncommutative case, the same type of argument holds for operators $H: L_2(\M) \to L_2(\M)$ satisfying that $H \, H_\op = -\Id$, where $H_\op(f) = H(f^\ast)^\ast$. We have the following improvement over Proposition \ref{prp:LpExtrapolation}.

\begin{proposition}
  \label{prp:ExtrapolationIdempotent}
  Let $\N \subseteq \M$ and $H:L_2(\M) \to L_2(\M)$ be as before and assume $H$ is left $\N$-modular, 
  commutes with $\EE: \M \to \N$ and satisfies that $\EE^\perp H \, H_\op = -\EE^\perp$. 
  If $H$ satisfies \eqref{eq:CotlarNC}, then
  \[
    \big\| H: L_p(\M) \to L_p(\M) \big\|
    \, \lesssim \,
    p \, \big\| H:L_2(\M) \to L_2(\M) \big\|
    \quad \mbox{ for every } p \geq 2.
  \]
  The same inequality holds in the non-relative case if $H \, H_\op = -\Id$.
\end{proposition}

\begin{proof}
  The proof is immediate once it is noticed that the property $\EE^\perp H \, H_\op = -\EE^\perp$ implies that \eqref{eq:CotlarNC} can be rewritten as
  \[
    \EE^\perp \big[ H(f) \, H(f)^\ast \big] 
    \, = \, 
    \EE^\perp \left[ H \big( f \, H(f)^\ast \big) + H \big( f \, H(f)^\ast \big)^\ast 
      + f f^\ast \right].
  \]
  Applying the same proof of Proposition \ref{prp:LpExtrapolation} gives the recurrence $c_{2p}^2 \leq 2 \, c_{2 p} \, c_p + 1 + 3 c_2^2$. After solving the quadratic inequality, we obtain
  $c_{2p} \leq c_p + \sqrt{c_p^2 + \kappa}$, where $\kappa = 1 + 3 c_2^2$. Iterating and applying Marcinkiewicz interpolation gives the bound. The same proof works in the non-relative case.
\end{proof}

Observe that if $H = T_m$ is a Fourier multiplier, then $(T_m)_\op = T_{\tilde{m}}$, for $\widetilde{m}(g) = \overline{m(g^{-1})}$. Thus, we are asking that $m(g) \, \overline{m(g^{-1})} = -1$ for every $g \in \sG \setminus \sG_0$. Similarly, since in the case of multipliers $\| \EE H(\1) \|_\infty = \| m \, \1_{\sG_0} \|_\infty \leq \| m \|_\infty = c_2$ we can simplify the recurrence above assuming $c_2 = 1$. In particular, we obtain the following corollary.

\begin{corollary}
  \label{cor:OptConstantGroup}
  Let $\sG$ be a group and let $m: \sG \to \CC$ be a function satisfying \eqref{eq:CotlarFactor} 
  relative to a subgroup $\sG_0 \subseteq \sG$, and such that $m$ is left $\sG_0$-invariant and $m(g)\overline{m(g^{-1})} = -1$, 
  for every $g \in \sG \setminus \sG_0$. Then
  \[
    \big\| T_m: L_p(\VN \sG) \to L_p(\VN \sG) \big\|_\cb
    \,\, \lesssim \,\,
    \Big( \frac{p^2}{p-1} \Big) \, \| m \|_\infty.
  \]
\end{corollary}

Both Proposition \ref{prp:ExtrapolationIdempotent} and Corollary \ref{cor:OptConstantGroup} are still true if one changes the value of $m(g) \, \overline{m(g^{-1})}$ from $-1$ to any other constant independent of $g$.

\textbf{Non Fourier multiplier examples.}
All of the machinery developed or review in this Section is formulated in a way that works beyond the case of Fourier multipliers. As an illustration, another family of examples comes from Schur multipliers. Indeed, let $\D \subset \B(\ell_2 \ZZ)$ be the finite span of the matrix units $e_{j \, k}$. A Schur multiplier is a linear map $S_m: \D \subseteq \B(\ell_2 \ZZ) \to \B(\ell_2 \ZZ)$ defined by sending the matrix unit $e_{j \, k}$ to $m(j,k) \, e_{j \, k}$ for some function $m:\ZZ \times \ZZ \to \CC$ ---called the symbol of $S_m$. Schur multipliers are $\ell_\infty(\ZZ)$-bimodular, with $\ell_\infty(\ZZ)$ sitting inside $\B(\ell_2 \ZZ)$ as the diagonal subalgebra. As such, if we take in Theorem \ref{prp:LpExtrapolation} $\M$ to be $\B(\ell_2 \ZZ)$ and $\N$ to be $\ell_\infty(\ZZ)$, we have that any bounded symbol $m$ satisfying  \eqref{eq:CotlarNC} extends boundedly to all Schatten classes $S_p$ with $1 < p < \infty$. The same computations of Theorem \ref{thm:ClosedFormulaCotlar} can be performed in this setting, yielding that
\begin{align}
    S_m \text{ satisfies } \eqref{eq:CotlarNC} & \text{ relative to } \ell_\infty \nonumber\\
      & \equivalent
        \big( m(j,k) - m(j,\ell) \big) \, \big( m(\ell,j) - m(\ell,k) \big) \, \1_{\{j \neq \ell\}} = 0, \; \forall j, k, \ell \in \ZZ. \label{eq:CotlarSchur}
\end{align}    
This makes verifying the Cotlar identity for Schur multipliers straightforward. For instance, the triangular truncation $H = S_m$ is the Schur multiplier with symbol $m(j,k) = \sgn(j - k)$. It is trivial to verify that it satisfies identity \eqref{eq:CotlarSchur}. Indeed, if $m(j,k) \neq m(j,\ell)$ it is because $j < \ell$ and $j \geq k$ or vice--versa. If the first case holds, then $k \leq j < \ell$, which implies that $m(\ell,j) = m(\ell,k)$. The other case is checked similarly. Since $H \, H^\op = - \Id$, we recover the sharp bound of the operator-norm on $S_p$ as $p \to 1^{+}, \infty$. This is not a novel result, since the triangular truncation is a very well understood object whose $S_p$-boundedness can be obtained from an array of techniques ranging from subdiagonal algebras \cite{Randrianantoanina1998Hilbert} to Fourier-Schur transference for amenable groups \cite{NeuRic2011, CasSall2015} and smooth Schur multiplier techniques \cite{CondeGonParcetTablate2022I}. Nevertheless, it signals the possibility of exploiting Cotlar identities in settings beyond Fourier multipliers.

\textbf{The convex hull of Cotlar-type multipliers.} A natural question is which class of multipliers $m$ can be shown to be bounded in $L_p(\VN \sG)$ by being represented as a convex combination of multipliers satisfying \eqref{eq:CotlarNC} or natural modifications of them.
To that end notice that if $m$ is an $L_p$-bounded multiplier, then so is $g \mapsto m(h^{-1} \theta(g) r)$, where $h, r \in \sG$, $\theta \in \Aut(\sG)$ and their norms coincide. Let us denote the group of transformations of $\sG$ given by $g \mapsto h^{-1} \theta(g) r$ as the \emph{affine transformations $\Aff(\sG)$ of $\sG$}. Observe that, if we define $\sG^\Delta = (\sG \oplus \sG) / \Delta$, where $\Delta$ is the subgroup of diagonal central elements, $\Delta = \{ (a,a) : a \in \Zent(\sG) \} \subseteq \sG \oplus \sG$, then there is a faithful representation that sends $(h, r)$ to $g \mapsto h^{-1} \, g \, r$. A trivial computation gives that
\[
  \Aff(\sG) \cong \sG^\Delta \rtimes \Aut(\sG),
\] 
with the natural action. 
It is clear that, if $\mu \in M(\Aut(\sG))$ is a finite signed measure and $m: \Aff(\sG) \times \sG \to \CC$ is a bounded map such that $g \mapsto m(\alpha, g)$ satisfies \eqref{eq:CotlarNC} for every $\alpha$, then 
\[
  m(g) \, = \, \int_{\Aff(\sG)} m(\alpha, \alpha(g)) \, d \mu(\alpha),
\]
is clearly bounded in $L_p(\VN \sG)$ for every $1 < p < \infty$. We could add more flexibility to this technique by allowing the map $g \mapsto m(\alpha, g)$ to be the product of $k$ terms satisfying \eqref{eq:CotlarNC}. Let us call this class $\mathrm{coCot}^k(\sG)$. We leave mostly unexplored the following natural problem

\begin{problem}
  \label{prb:ConvexCotlar}
  Let $\sG_0 \subseteq \sG$ and $m:\sG \to \CC$ be as above. Are there sufficient conditions, 
  for example in terms of smoothness, implying that $m \in \mathrm{coCot}^k(\sG)$?
\end{problem}

This remains as a underexplored approach to prove the boundedness of Fourier multipliers over groups without recourse to noncommutative analogues of singular integral theory. In fact, Observe that in the classical case of $\RR$ any function of bounded variation lays in the convex hull of (translations of) the classical Hilbert transform and therefore $m \in \mathrm{BV}(\RR) \implies m \in \mathrm{coCot}^1(\RR)$, see \cite[Corollary 3.8]{Duoan2001Book}. In higher dimensions the behavior is even richer. For instance, let $m:\RR^2 \to \CC$ be an homogeneous function, ie a function satisfying that $m(\lambda \, \xi) = m(\xi)$ for every $\lambda > 0$. Clearly, $m$ depends only on its angular component $m{|}_\TT$. We have that
\[
  m{|}_\TT \in \mathrm{BV}(\TT) \quad \implies \quad m \in \mathrm{coCot}^2(\RR^2),
\]
where $\mathrm{BV}(\TT)$ is the space of functions of bounded variation on the torus. To see that, just notice that the indicator function $\1_{\Sigma_\theta}$ of the sector $\Sigma_\theta \subseteq \RR^2$ of all vectors whose polar angle $\omega$ lays in $[0,\theta)$ can expressed as the product of the characteristic functions of two half-planes, each of which satisfies a Cotlar identity. Homogeneous functions of angular bounded variation can be expressed as convex combinations of different sector indicators.
%
%We can write its characteristic function as
%\[
%  \1_{\Sigma_\theta}(\underbrace{\xi_1, \xi_2}_\xi)
%  \, = \,
%  \left( \frac{\sgn \langle \xi, v_{\frac{\pi}{2}} \rangle + 1}{2} \right) \,
%  \left( \frac{\sgn \langle \xi, v_{\theta - \frac{\pi}{2}} \rangle + 1}{2} \right),
%\]
%when $\theta < \pi$ and a similar expression otherwise. But, any function $\widetilde{m}: \TT \to \CC$ of bounded variation can be identified with a function on $[0,2\pi]$ such that $\widetilde{m}(2\pi) - \widetilde{m}(0)$ equals the jump discontinuity of the original function around $0$. Elementary manipulations show that $\widetilde{m}$ lays in the convex hull of the functions $\1_{[0,\theta]}$. Radially extending the argument gives that $m$ is in the convex hull of $\1_{\Sigma_\theta}$.

\textbf{Generalizations.}
It worth noticing that the identity \eqref{eq:CotlarNC} can be generalized in a natural way by changing the equality by an operator inequality
\begin{equation}
  \tag{Cotlar$_{\EE^\perp}^{\leq}$}
  \label{eq:CotlarNCleq}
  \EE^\perp \big[ H(f) \, H(f)^\ast \big] 
  \, \leq \, 
  \EE^\perp \left[ H \big( f \, H(f)^\ast \big) + H \big( f \, H(f)^\ast \big)^\ast 
    - H \big( H(f f^\ast)^\ast \big) \right] + \Lambda \, \EE [f \, f^\ast],
\end{equation}
for every $f \in \M \cap L_2(\M)$ and some constant $\Lambda \geq 0$. It is clear that this inequality implies the same bound \eqref{eq:Extrapolation} of Proposition \ref{prp:LpExtrapolation} for left $\N$-modular operators, just with a constant growing like $\sO(1 + \sqrt{\Lambda})$. Similarly, there exists a closed-formula characterization of Fourier multipliers $T_m$ with left $\sG_0$-invariant symbol $m$ satisfying \eqref{eq:CotlarNCleq}. To formulate such characterization let us define
\[
  \Omega_m(g,h) = \big( m(g) - m(g h^{-1}) \big) \, \overline{\big( m(h) - m(h g^{-1}) \big)} \, \1_{\sG \setminus \sG_0}(g h^{-1}) + \Lambda \, \1_{\sG_0}(g h^{-1})
\]
and notice that \eqref{eq:CotlarNCleq} is actually equivalent to
\[
  0 \, \leq \, 
  \int_{\sG \setminus \sG_0} \int_\sG 
    \widehat{f}(g h) \, \overline{\widehat{f}(h)} \, 
    \Omega_m(g h, h) \, \lambda_g \, d\mu(h) \, d\mu(g).
\]
This is the key to the following
\begin{theorem}
  \label{thm:CotlarLeqFactor}
  Let $\sG$ be a unimodular group and let $\sG_0 \subseteq \sG$ be an open subgroup and $m$ 
  be a left $\sG_0$-invariant function. The following are equivalent.
  \begin{enumerate}[label={\rm \textbf{(\roman*)}}, ref={\rm {(\roman*)}}]
    \item \label{itm:CotlarLeqFactor.1} 
    The operator $H = T_m$ satisfies \eqref{eq:CotlarNCleq}.
    \item \label{itm:CotlarLeqFactor.2} 
    There is a Hilbert space $\H$ and a bounded measurable function $\xi: \sG \to \H$ such that $\Omega_m(g,h) = \big\langle \xi(h), \xi(g) \big\rangle$.
  \end{enumerate}
  If any of the two conditions hold, then bound \eqref{eq:TheoremABound} is satisfied.
\end{theorem}

\begin{proof}
It is clear that if there exists a map $\xi: \sG \to \H$ as above, then
\[
  \begin{split}
    \int_{\sG \setminus \sG_0} \int_\sG 
      \widehat{f}(g h) \, & \overline{\widehat{f}(h)} \, 
      \Omega_m(g h, h) \, \lambda_g \, d\mu(h) \, d\mu(g) \\
    & = \left\langle \int_{\sG} \widehat{f}(g) \, \xi(g) \otimes \lambda_g \, d\mu(g), \,
                     \int_{\sG} \widehat{f}(h) \, \xi(h) \otimes \lambda_h \, d\mu(h) \right\rangle_\X,
  \end{split}
\]
where $\X$ is the left Hilbert $\VN \sG$-module given by completing $\H \otimes \VN \sG$ with the inner product 
$\langle \xi \otimes f, \eta \otimes g \rangle_\X = \langle \eta, \xi \rangle f g^\ast$. The order of the entries $\xi$ and $\eta$ is switched to maintain the convention that every scalar Hilbert product is antilinear in the first component. The same applies to the statement in point \ref{itm:CotlarLeqFactor.2}.

For the reciprocal, first notice that the identity \eqref{eq:CotlarNCleq} can be understood as a positivity condition for a quadratic form. Thus, applying polarization gives the sesquilinear form $B_H: \D \times \D \to \D$ given by
\[
  B_H( f, g )
  \, = \,
  \EE^{\perp} \Big[ - H(g) \, H(f)^\ast + H\big(g \, H(f)^\ast \big) + H\big(f \, H(g)^\ast \big)^\ast 
                    - H\big( H(f \, g^\ast)^\ast \big) \Big] + \Lambda \, \EE[g \, f^\ast] ,
\]
where $\D:= \lambda[C_c(\sG)]$. Observe that the unbounded map $\Ee: \D \subseteq \VN \sG \to \CC$ given by
\[
  \Ee \left( \int_{\sG} \widehat{f}(g) \, \lambda_g \, d\mu(g) \right)
  \, = \,
  \int_{\sG} \widehat{f}(g) \, d\mu(g)
\]
is multiplicative, and thus positive, therefore $\langle f, g \rangle_H = \Ee(B_H(f,g))$ is an inner product. Now, we can construct a Hilbert space $\H_0$ by quotienting out the nulspace and taking closures of $\D$ as usual. Notice that, in the case of discrete $\sG$, it holds that $\lambda_g \in \D$ and that $\langle \lambda_g, \lambda_h \rangle_H = \Omega_m(g,h)$. Therefore, defining $\xi(g)$ as the class in $\H_0$ of $\lambda_g$ gives the desired result. In the continuous case we can substitute $\lambda_g$ for $\lambda(\delta_g \ast \psi_{\alpha})$, where $(\psi_\alpha)_\alpha  \subseteq \D$ is an approximation of the unit and apply standard ultraproduct arguments.
\end{proof}

The reason why we need to add the term $\Lambda \, \EE[f \, f^\ast]$ to the Cotlar identity in the case in which $\sG_0$ is open is that, otherwise, we will be forcing the kernel $\Omega_m(g,h)$ to be positive definite while also $0$ at the diagonal $g = h$, which implies that it is $0$ everywhere. In the case in which $\sG_0$ is of empty interior, this extra term can be dropped.

This more general Cotlar identity \eqref{eq:CotlarNCleq} leads to the following problem, which we l eave unexplored.

\begin{problem} \normalfont
  \label{prb:ModelForleqCotlar}
  Let $\sG_0 \subseteq \sG$ and $m: \sG \to \CC$ be as above. Is there a geometric model of $\sG \acts \X$, possibly generalizing that of actions on UAC spaces in Theorem \ref{thm:TheoremB}, such that if $m(g)$ lifts to $\X$ via a function $m(g) = \widetilde{m}(g \cdot x_0)$, then it satisfies the condition in Theorem \ref{thm:CotlarLeqFactor}.\ref{itm:CotlarLeqFactor.2}?
\end{problem}

\section{Groups acting on \texorpdfstring{$\RR$}-trees \label{S:UAC}}
Let $X$ be a Hausdorff topological space. An arc $\gamma$ on $X$ is a subset of $X$ that is the image of an injective continuous function mapping $[0, 1]$ onto $\gamma$. The space $X$ is said to be \emph{uniquely arcwise connected}, or UAC, if any two points in $X$ are joined by a unique arc. We say that a group $\sG$ acts on an UAC space $X$ if $\sG$ acts by homeomorphisms on $X$.

If, in addition, the UAC space $X$ is metrisable and there is metric $d: X \times X \to \CC$ such that the unique arc joining two points is isometric to a closed interval of the real line, then $(X, d)$ is called an $\RR$-tree. We will say that a group $\sG$ acts on an $\RR$-tree $X$ if it acts on it by isometries. 

Observe that this definition is topological in nature since the underlying space is required to be arcwise connected. An alternative route to $\RR$-trees can be taken by defining them as hyperbolic spaces with $\delta = 0$, i.e., every triangle is a \emph{tripod}. These two definitions, although equivalent in spirit, are slightly different. Namely, a tree seen as a discrete set with the edge metric is a $0$-hyperbolic metric space but not an $\RR$-tree in our definition. This is not a problem since trees can still be seen as a subclass of $\RR$-trees by treating them as \emph{simplicial trees}, i.e. the one-dimensional simplicial complexes obtained from the incidence information of the tree. We also have that, given an $\RR$-tree $X$, if the set of points whose complement has three or more connected components is discrete in $X$, then $X$ is a simplicial tree.
The first definition of $\RR$-trees was given by Tits \cite{Tits1977LieKolchin}, then Morgan and Shalen \cite{MorganShalen1984}, following earlier results of Alperin and Moss, drew attention to the theory of $\RR$-trees by showing how to compactify a generalization of Teichmuller space for a finitely generated group using $\RR$-trees. We refer the reader to \cite{Bestvina2002} for more on $\RR$-trees.  
%
%{\marginpar{I think it is not true that $X \setminus \{x_0\}$ decomposes as a union of connected components. Think of the cone whose base is a Cantor space.}}

We define the following two models for a group acting on a UAC space. Let $\sG \acts X$ be a topological action and $x_0 \in X$ a selected point. We will say that a bounded measurable function $\varphi: X\setminus \{x_0\} \to \CC$ is \emph{constant along arcs} iff $\varphi(x) = \varphi(y)$ if there is an arc connecting $x$ and $y$ inside $X \setminus \{x_0\}$. This is equivalent to decomposing $X \setminus \{x_0\}$ as a union of arcwise connected subsets and imposing the function to be constant over those subsets.

\begin{model} \normalfont
  \label{mod:model1}
  Let $\sG \acts X$ be a topological action on a UAC space, $x_0 \in X$ a point
  and $\widetilde{m}: X \to \CC$ a bounded measurable function such that
  \begin{enumerate}[label={\rm \textbf{(\roman*)}}, ref={\rm {(\roman*)}}]
    \item $\widetilde{m}$, restricted to $X \setminus \{x_0\}$, is constant along arcs.
    \item $\widetilde{m}$, is invariant under the action $\st_{x_0} \acts X \setminus \{x_0\}$.
  \end{enumerate}
  Then, we define the multiplier $m: \sG \to \CC$ as
  \[
    m(g) = \widetilde{m}(g \cdot x_0)
  \]
  and fix $(\sG_0, \sG)$ as $(\st_{x_0}, \sG)$.
\end{model}

This definition has the drawback that the invariance under $\st_{x_0}$ of $m$ can make the symbol constant outside $\st_{x_0}$ in some cases. We introduce the following, more involved, model.

\begin{model} \normalfont
  \label{mod:model2}
  Let us fix two distinct constants $C_1$, $C_2 \in \CC$. Let similarly $\sG \acts X$ be a topological action on a UAC space and $x_0 \in X$ a point. Choose $X_0 \subseteq X \setminus \{x_0\}$ an arcwise connected subset. We define $m: \sG \to \CC$ to be
  \begin{equation*}
    m(g)
    \, = \,
    \begin{cases}
      0   & \mbox{when } \, g \in \st_{x_0} \, \mbox{ and } \, g \cdot X_0 = X_0\\
      C_1 & \mbox{when } \, g \in \st_{x_0} \, \mbox{ and } \, g \cdot X_0 \neq X_0 \\
      C_1 & \mbox{when } \, g \not\in \st_{x_0} \, \mbox{ and } \, g \cdot x_0 \not\in X_0 \\
      C_2 & \mbox{when } \, g \cdot x_0 \in X_0.
    \end{cases}
  \end{equation*}
  We will also fix $\sG_0$ to be $\st_{x_0} \cap \{ g \in \sG : g \cdot X_0 = X_0 \}$.
\end{model}

Observe that Model \ref{mod:model2} is a natural modification of Model \ref{mod:model1} for the function $\widetilde{m}: X \to \CC$ given by $\widetilde{m} = C_1 \1_{X\setminus (\{ x_0\}\cup X_0)} + C_2 \1_{X_0}$. The main difference is that we extend the value of $C_1$ to a portion of the stabilizer. 

\begin{proposition}
  \label{prp:modelCotlar}
  Let $\sG \acts X$ be an action as above.
  \begin{enumerate}[label={\rm \textbf{(\roman*)}}, ref={\rm {(\roman*)}}]
    \item \label{itm:modelCotlar1} Let $m: \sG \to \CC$ and $\sG_0$ be like in Model \ref{mod:model1}. Then, $T_m$ satisfies \eqref{eq:CotlarNC} relative to $\sG_0$ and is left $\VN \sG_0$-modular.
    \item \label{itm:modelCotlar2}Let $m: \sG \to \CC$ and $\sG_0$ be like in Model \ref{mod:model2}. Then, $T_m$ satisfies \eqref{eq:CotlarNC} relative to $\sG_0$ and is left $\VN \sG_0$-modular.
  \end{enumerate}
  The points above imply that $T_m$ in both model \ref{mod:model1} and  \ref{mod:model2} are bounded in $L_p(\VN \sG)$ for $1 < p < \infty$.
\end{proposition} 

\begin{proof}
  The statement in point \ref{itm:modelCotlar1} has already been proved in the introduction. Thus, we concentrate on point \ref{itm:modelCotlar2}. The fact that $m$ is left invariant under the action of $\sG_0 = \{ g \in \sG : g \cdot x_0 = x_0 \} \cap \{ g \in \sG : g \cdot X_0 = X_0 \}$ is immediate. Now, all we have to do is to verify \eqref{eq:CotlarFactor}. To that end, let us divide the group as a disjoint union $\sG = \sG_0 \cup \sG_1 \cup \sG_2 \cup \sG_3$, where
  \begin{eqnarray*}
    \sG_0 & = & \{ g \in \sG : g \cdot x_0 = x_0 \} \cap \{ g \in \sG : g \cdot X_0 = X_0 \} \\
    \sG_1 & = & \{ g \in \sG : g \cdot x_0 = x_0 \} \cap \{ g \in \sG : g \cdot X_0 \neq X_0 \} \\
    \sG_2 & = & \{ g \in \sG : g \cdot x_0 \neq x_0 \} \cap \{ g \in \sG : g \cdot x_0 \not\in X_0 \} \\
    \sG_3 & = & \{ g \in \sG : g \cdot x_0 \neq x_0 \} \cap \{ g \in \sG : g \cdot x_0 \in X_0 \}.
  \end{eqnarray*}
  Observe also that, since $m$ is both left and right $\sG_0$-invariant,  it is enough to verify \eqref{eq:CotlarFactor} for $g, h \in \sG \setminus \sG_0$. Assume that $m(g^{-1}) \neq m(h)$, otherwise we are done, our aim is to show that $m(g h) = m(g)$. We will proceed by cases. First, assume that $g \in \sG_1$. This is equivalent to $g^{-1} \in \sG_1$ and therefore $h \in \sG_3$. But then, $g \, h \cdot x_0 \not\in X_0$ and $g \, h \cdot x_0 \neq x_0$. Therefore $g \, h \in \sG_2$ and we get $m(g h) = m(g)$. In the case of $g \in \sG_2$ we have the whole range of possibilities and $h$ can belong to $\sG_1$, $\sG_2$ or $\sG_3$. In the case of $g \in \sG_2$ and $h \in \sG_1$ we have that $g h \in \sG_2$. Indeed, $g h \cdot x_0 = g \cdot x_0 \not\in X_0$ and it is immediate that $g h$ is not in the stabilizer of $x_0$. For the second case of $g \in \sG_2$ and $h \in \sG_2$ the condition $m(g^{-1}) \neq m(h)$ implies that $g^{-1} \cdot x_0$ and $h \cdot x_0$ live in distinct {arcwise} connected subsets of $X \setminus \{x_0\}$. But then there is a unique path connecting both points that passes through the root. Applying $g$ to the whole arc gives that $g \cdot x_0$ and $g h \cdot x_0$ lay in the same {arcwise} connected subset and do not stabilize $x_0$, see Figure \ref{fig:UACaction}. Therefore $m(g h) = m(g)$. The third case is given by $g \in \sG_2$ and $h \in \sG_3$. Observe that if $g \in \sG_2$, then $g^{-1}$ can only be inside $\sG_2$ or $\sG_3$. But the case $g^{-1} \in \sG_3$ can be easily discarded since it will contradict the assumption $m(g^{-1}) \neq m(h)$. But if $g^{-1} \in \sG_2$ and $h \in \sG_3$, then $g^{-1} \cdot x_0$ and $h \cdot x_0$ live in distinct {arcwise} connected subsets and we can proceed like in the previous case.
  It remains to check the case of $g \in \sG_3$. We have that $h$ can be in either $\sG_1$, $\sG_2$ or $\sG_3$. In the first case we deduce that $g h \in \sG_3$. For the second one we have that if $g \in \sG_3$ and $h \in \sG_2$, then we can assume that $g^{-1} \in \sG_3$, the only other choice being $g^{-1} \in \sG_2$ which will contradict the assumption $m(g^{-1}) \neq m(h)$. But this implies that $g^{-1} \cdot x_0$ and $h \cdot x_0$ live in different arcwise connected subsets of $X \setminus \{x_0\}$ and we can apply the argument in Figure \ref{fig:UACaction}. Lastly, if $g \in \sG_3$ and $h \in \sG_3$ we obtain similarly that $g^{-1}$ can only lay in $\sG_2$. The same path argument applies. Applying Theorem \ref{thm:TheoremB}, we get the $L_p$-boundedness of $T_m$.
\end{proof}

\begin{remark} \normalfont
    \label{rmk:nonRelativeCaseModel}
    Recall also that, by Remark \ref{rmk:NonrelativeCotlar}, when the subgroup $\st_{x_0}$ has empty interior, condition \ref{itm:Model.2} in Model \ref{mod:model1} can be dropped. %Nevertheless, most of our examples would be in the context of discrete groups and 
\end{remark}

Observe that in the case of $\sG = \RR$,  the Hilbert transform is also of weak type $(1,1)$, ie $H:L_1(\RR) \to L_{1,\infty}(\RR)$ and bounded between $L_\infty(\RR)$ and the space of \emph{bounded mean oscillation functions} $\BMO(\RR)$. Both endpoint spaces give ---by either complex or real interpolation with $L_2$--- the optimal order for the operator $L_p$ norm of $H$. The following problem remains open

\begin{problem} \normalfont
  \label{prb:endpoint}
  Let $\sG$ be a group and $m$ a multiplier like in Model \ref{mod:model1} or \ref{mod:model2}. 
  Is it possible to construct spaces $\sX_1$ and $\sX_\infty$, in place of $L_{1,\infty}$ and $\BMO$,
  such that
  \begin{enumerate}[label={\rm \textbf{(\roman*)}}, ref={\rm {(\roman*)}}]
    \item \label{itm:endpoint1}
    $\| T_m: L_1(\VN \sG) \to \sX_1 \| < \infty$ and $\| T_m: L_\infty(\VN \sG) \to \sX_\infty \| < \infty$,
    \item \label{itm:endpoint2}
    interpolation of $\sX_1$ or $\sX_\infty$ 
    with $L_2$ yields growth of $\max\{p, p'\}$ for the operator $L_p$ norm of $T_m$?
  \end{enumerate}
\end{problem}

This problem presents at least three challenges. The first difficulty comes from the fact that weak type $(1,1)$ bounds are difficult to obtain for noncommutative singular integral type operators. There are known in some semicommutative examples \cite{Par2009CZ, CadiCondeParcet2021} but open in the case of Quantum Euclidean spaces \cite{GonJunPar2017singular} and in most group settings beside left-orderable groups \cite{Randrianantoanina1998Hilbert}. The second challenge is that the specific endpoint space $\sX_\infty$ to be used for $m$ in Model \ref{mod:model1} or Model \ref{mod:model2} has to be defined in terms of the geometry of $\sG \acts X$ and it cannot just be the usual noncommutative $\BMO$ space, see \cite{JunMei2012BMO, Mei2008Tent, Mei2012H1BMO}. Indeed, there is a natural \emph{unital and completely positive} semigroup $S_t: \VN \FF_2 \to \VN \FF_2$ in the free group algebra given by $S_t(\lambda_g) = e^{-t|g|}$, see \cite{Haa1978}. This semigroup allows to construct a natural and interpolating semigroup $\BMO$ space $\BMO(\VN \FF_2)$, see \cite{JunMeiPar2014Riesz}. But it is known that the multipliers \eqref{eq:MR} are unbounded from $L_\infty(\VN \FF_2)$ to $\BMO(\VN \FF_2)$, see \cite[Appendix A]{MeiRicardXu2022Free}. The third difficulty comes from the fact that the classical technique, employed by Kolmogorov \cite{Kolmogorov1925Conjugate}, of comparing a singular integral operator with a maximal one is delicate in this context since some of the operators obtained are not positivity preserving, which makes interpolating maximal functions an open problem \cite{JunXu2003}. The technique of Kolmogorov has been used in the noncommutative case in \cite{HongLaiXu2022MaximalSIO}. We will also mention that some tentative progress in the direction of Problem \ref{prb:endpoint} has been made. For instance, Ga{\l}{\c a}zka and Os{\c e}kowski \cite{GalkazkaOsekowski2022} have lowered the operator $L_p$-norm of the Fourier multiplier $m: \FF_2 \to \CC$ that depend on the starting letter from $\sO(p^\beta)$ to $\sO(p \log p)$ as $p \to \infty$.

%% Fixed points 
Now, we are going to study Models \ref{mod:model1} and \ref{mod:model2} in the context of $\RR$-trees. Observe that, since any $\RR$-tree action is an action of the underlying UAC topological space, Proposition \ref{prp:modelCotlar} above works for actions on $\RR$-trees. Furthermore, in the case of group actions on $\RR$-trees we have the following result connecting the existence of global fixed points with the form of the Fourier multiplier $T_m$. We are going to say that the multipliers $m$ coming from Models \ref{mod:model1} and \ref{mod:model2} are \emph{trivial} iff they are constant for any $g \in \sG \setminus \sG_0$.

\begin{proposition}
  \label{prp:FixedPoint}
  Let $X$ be an $\RR$-tree and $\sG \acts X$ an (isometric) action of discrete group. The following holds
  \begin{enumerate}[label={\rm \textbf{(\roman*)}}, ref={\rm {(\roman*)}}]
    \item \label{itm:FixedPoint.1} 
    If the action $\sG \acts X$ has a global fixed point, then for any choice of a root $x_0 \in X$
    the multipliers in Model \ref{mod:model1} and Model \ref{mod:model2} are trivial.
    
    \item \label{itm:FixedPoint.2} 
    If $\sG$ is finitely generated, for any action on an $\RR$-tree $X$, there is either a global fixed point 
    or there exists $x_0 \in X$ such that the corresponding symbol given in Model \ref{mod:model2} is nontrivial. 
  \end{enumerate}
\end{proposition}

\begin{proof}
  We will prove first \ref{itm:FixedPoint.1} for $m$ as in Model \ref{mod:model1}. Assume that the actions has a global fixed point $x_1 \in X$. If $x_1$ coincides with the root $x_0$ then $m = 0$ and there is nothing to prove. Therefore we assume $x_0 \neq x_1$. Similarly, we can assume without loss of generality that $\st_{x_0}\neq \sG$, since otherwise $m = 0$. Pick $g \in \sG$ and assume that $g \cdot x_0$ lays in an arcwise connected subset of $X \setminus \{x_0\}$ different from that of $x_1$. Then, there is a unique path joining $x_1$ and $g \cdot x_0$ that passes through the root $x_0$, see Figure \ref{fig:FixedPoint1}. But since $x_1$ is fixed by any $g\in \sG$, after applying $g$ to the path we obtain a larger path, which contradict the fact that the action is isometric. 
  \begin{figure}[ht]
    \centering
    \includegraphics[scale=1]{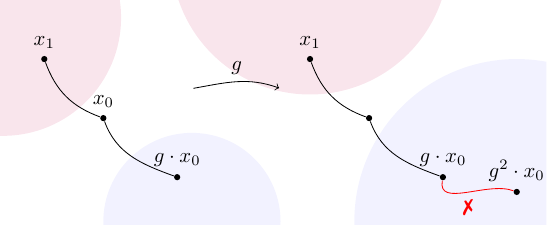}
    \caption{The action of $g$ over the path connecting the global fixed point and $g \cdot x_0$.}
    \label{fig:FixedPoint1}
  \end{figure}
This implies that $g^2 \cdot x_0 = g \cdot x_0$ and so $g \cdot x_0 = x_0$, which contradicts the assumptions. Therefore, $g \cdot x_0$ belongs to the same connected subset of $x_1$ for every $g$ that do not stabilize the root $x_0$, that is, $m(g)$ is constant for any $g\in \sG\setminus \st_{x_0}$.

For the case of $m$ as in Model \ref{mod:model2}, let $x_1 \in X$ be a global fixed point. We can again consider without loss of generality that $x_0 \neq x_1$ and that $\st_{x_0}\neq \sG$. By those assumptions, there exists  $g_0 \in \sG$ with $g_0 \cdot x_0 \in X \setminus \{x_0\}$. We have two possibilities for $g_0$. If $g_0 \cdot x_0 \neq x_0$ and $g_0 \cdot x_0 \not\in X_0$, then $m(g_0) = C_1$, while if $g_0 \cdot x_0 \in X_0$ then $m(g_0) = C_2$. If we are in the first case $g_0 \cdot x_0 \not\in X_0$, then the multiplier $m$ will be trivial unless there exists a $g_1 \in \sG$ such that $g_1 \cdot x_0 \in X_0$. Let us obtain a contradiction. First, we claim that the global fixed point $x_1 \not\in X_0$. Assume $x_1 \in X_0$. Then, there is a path joining $x_1$ and $g_0 \cdot x_0$ that passes through the root $x_0$, see Figure \ref{fig:FixedPoint2}
  \begin{figure}[ht]
    \centering
    \includegraphics[scale=1]{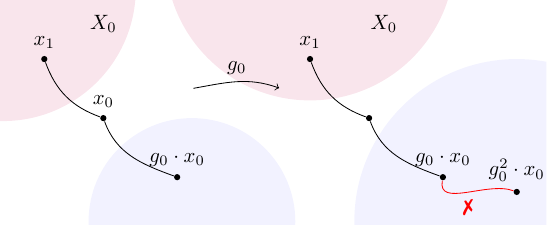}
    \caption{The action of $g_0$ over the path connecting $x_1$ and $g_0 \cdot x_0 \not\in X_0$.}
    \label{fig:FixedPoint2}
  \end{figure}
But applying $g_0$ to the path gives that $g_0 \cdot x_0 = x_0$ which is a contradiction. 
Therefore $x_1 \not\in X_0$, But since $g_1 \cdot x_0 \in X_0$, we can build a path from $x_1$ to $g_1 \cdot x_0$ that passes through the root $x_0$ and, repeating the same argument as before, obtain that $g_1 \cdot x_0 = x_0$, which is a contradiction. 
For the second case $g_0 \cdot x_0 \in X_0$. First, we notice that $x_1$ must live in $X_0$, if not, proceeding as before we will get that $g_0 \cdot x_0 = x_0$, which is a contradiction. In order for $m$ to be nontrivial there should be a $g_1 \in \sG$ such that either $g_1 \cdot x_0 = x_0$ and $g_1 \cdot X_0 \neq X_0$ or $g_1 \cdot x_0 \neq x_0$ and $g_1 \cdot x_0 \not\in X_0$. For the first case, let us choose a point $x_2 \in X_0$. Then, $g_1 \cdot x_2 \not\in X_0$ and similarly $g_1 \cdot x_2 \neq x_0$. Now, construct a path joining $x_1$ with $g_1 \cdot x_2$. Since $x_1 \in X_0$ but $g_1 \cdot x_2$ belongs to a different connected subset, the arc joining them passes through the root $x_0$. Let us apply $g_1^{-1}$ to the whole arc. Since we have that $g_1^{-1} \cdot x_0 = x_0$, we obtain an arc that joins $x_1$ and $x_2$ and passes through the root $x_0$. But this is a contradiction with the fact that $x_1$ and $x_2$ both belong to $X_0$, see Figure \ref{fig:FixedPoint3}, since $X$ is a UAC space the arc joining two elements in the same connected subset of $X \setminus \{x_0\}$ cannot pass trough the root $x_0$.
  \begin{figure}[ht]
    \centering
    \includegraphics[scale=1]{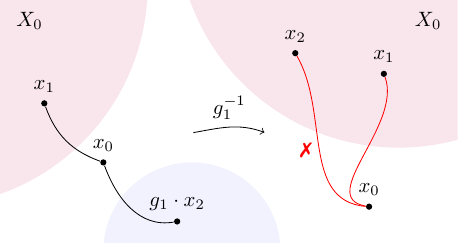}
    \caption{The action of $g_1^{-1}$ over the path connecting $x_1$ and $g_1 \cdot x_2$.}
    \label{fig:FixedPoint3}
  \end{figure}
  The remaining case is when $g_1 \cdot x_0 \neq x_0$ and $g_1 \cdot x_0 \not\in X_0$. Then, the arc joining $x_1 \in X_0$ and $g_1 \cdot x_1$ passes through $x_0$. Applying $g_1$ gives a contradiction with the fact that the action is isometric and that $g_1 \cdot x_0 \neq x_0$. For the other point, if $\sG$ is finitely generated, the converse statement is also true and it is a corollary of \cite[Lemma 5.2]{Bowditch1998}. This lemma tells us that when a finitely generated group $\sG$ acts non-trivially on an $\RR$-tree $X$, there exist a point $y_0$ and elements $g, h \in \sG$ such that $y_0$ lies in the arc between $g \cdot y_0$ and $h \cdot y_0$ and $y_0$, $g \cdot y_0$ and $h \cdot y_0$ are all different. Therefore, if $\sG$ does not have property (F$\RR$), we can choose $x_0$ as $y_0$ and $X_0$ as the maximally connected subset of $g \cdot y_0$. Then, there is an $h \in \sG$ such that $h \cdot x_0 \neq x_0$ and $h \cdot x_0 \not\in X_0$. Thus, sending $g$ and $h$ to different values $C_1$ and $C_2$ gives a nontrivial multiplier within Model \ref{mod:model2}.
\end{proof}

Many groups we are familiar with admit actions on $\RR$-trees. For example, every finitely generated hyperbolic group. Indeed, a finitely generated group is hyperbolic if and only if every asymptotic cone of the group is an $\RR$-tree \cite{Gromov1993}. Then, the action of a hyperbolic group on its Cayley graph induces an action on the asymptotic cone of the Cayley graph. Moreover, any surface group having Euler characteristic less than $-1$ acts freely on an $\RR$-tree \cite{MorganShalen1991}.

On the other hand, there are many examples of groups for which any action on an $\RR$-tree has a global fixed point. When this happens the group $\sG$ is said to have property {(F$\RR$)}, see \cite{Bestvina2002, Shalen1987dendrology}. The following corollary of Proposition \ref{prp:FixedPoint} characterizes groups with property {(F$\RR$)}.

\begin{corollary}
  $\sG$ has property {\rm (F$\RR$)} if and only if the symbol $m$ in Model \ref{mod:model2} is trivial for any $x_0 \in X$ in any $\RR$-tree $X$ on which $\sG$ acts (isometrically).  
\end{corollary}

%{\marginpar{\color{blue} It's not so clear to me what you mean in the last two sentences.}
One natural question, that we leave open, is whether there are functions $m: \sG \to \CC$ satisfying \eqref{eq:NCCotlar} for $\sG_0 = \{e\}$ and such that they do not lift to a function $\widetilde{m}$ on a UAC space like in Model \ref{mod:model1} or \ref{mod:model2} on which $\sG$ acts. One way to prove the non-existence of such lift $\widetilde{m}$ would be a procedure to assemble from the group $\sG$ and the function $m$ a $\sG$-space $X$ on which $\sG$ acts naturally and a lift $\widetilde{m}: X \to \CC$ satisfying hypothesis like those of Models \ref{mod:model1} and \ref{mod:model2}. So far, this reverse construction has escaped us. 

\section{Left orderable groups \label{S:LOG}}

% Left orderable groups
Recall that a \emph{total} or \emph{linear} order $\preceq$ is an order relation such that, given any two points $x$, $y$ then it holds that $x \preceq y$ or $y \preceq x$.
%\marginnote{{\color{red} ADR removed ", with both of them happening simultaneously precisely when $x = y$"}}
A \emph{left-orderable group} is a group admitting a left-invariant total order, ie $(\sG, \preceq)$ satisfies that $g \preceq h$ if and only if $k g \preceq k h$. Recall that, as we defined in the introduction,  every left-orderable group has a \emph{sign function} $\sgn: \sG \to \CC$ that assigns $+1$ or $-1$ depending on whether $e \prec g$ or $g \prec e$. We will prove that $H = T_\sgn$ is bounded.

\begin{proof}[Proof (of Theorem \ref{thm:TheoremC})]
  First, assume that $\sG$ is discrete. All that is required to do is to prove 
  the identity \eqref{eq:CotlarFactor} relative to $\sG_0 = \{e\}$. Assume that $\sgn(g) \neq \sgn(g h)$, where both $g$ and $h$ are different from $e$. If both $g$ and $h$ had the same sign, so would $g h$, therefore the sign of $h$ has to be different from that of $g$. But that implies that the signs of $g$ and $h^{-1}$ coincide. To see that the norm grows as $\sO(p)$ as $p\to \infty$ just notice that $\sgn(g^{-1}) = - \sgn(g)$ and therefore $m(g) \overline{m(g^{-1})} = -1$, which allows to apply the same technique in Proposition \ref{prp:ExtrapolationIdempotent} and Corollary \ref{cor:OptConstantGroup}. 
  The case of continuous groups follows similarly using the nonrelative Cotlar identity in \eqref{eq:CotlarNoExp}.
\end{proof}

Apart from the direct proof above, it is interesting to notice that these Hilbert transform type multipliers on left-orderable groups can be put into the framework of Model \ref{mod:model1} and Theorem \ref{thm:TheoremB}. To do that we need to use the well-known characterization of countable left-orderable groups as order preserving groups of homeomorphisms of the real line $\RR$, which is a UAC space. The first reference we found for this is \cite{Holland1963}. %, we include a proof below for the reader's convenience. Before proceeding to the proof recall that a total ordering is \emph{dense} iff for any $x$, $y$ such that $x\prec y$, there exists $z$ such that $x \prec z \prec y$.

\begin{proposition}
  \label{prp:LOGactonR}
  Every countable left-orderable group acts on the real line $\RR$ 
  by orientation preserving homeomorphisms and without global fixed point. 
\end{proposition}

%\marginnote{{\color{red} ADR removed the proof. It is now in trimmings. Is it possible to request }}

Observe as well that, given a homeomorphism of the real line $f: \RR \to \RR$ it is either orientation preserving or orientation reversing and that the composition of two orientation reversing maps is orientation preserving. Therefore, there is a group homomorphism $\Homeo(\RR) \to \ZZ_2 \cong \{1,-1\}$ that associate each homeomorphism of $\RR$ with $+1$ or $-1$ depending on whether the homeomorphism is orientation preserving or reversing. As a consequence each discrete group $\sG$ acting on $\RR$ has an index $2$ subgroup that is left-orderable. So, when the UAC space $X$ in Model \ref{mod:model1} is equal to $\RR$, the group $\sG$ is left-orderable and have an index $2$ subgroup in which the multiplier $m$ essentially coincides with the sign function.

The left invariant order of a left-orderable group is by no means unique. Thus, the sign functions associated to different orders may give different $L_p$-bounded Fourier multiplier. Similarly, left-orderable groups can have actions on UAC spaces other than $\RR$ which will yield different multipliers still within our Model \ref{mod:model1}. An example of this will be that of Baumslag-Solitar groups $\BS(1,n)$, which both admit nontrivial actions on their Bass-Serre trees and are left-orderable.

% Free groups (with Duchamp Thibon approach)
%\textbf{Examples: Free groups and the Magnus embedding.}
%\marginnote{{\color{red} ADR I've removed this examples. I have realized it is trivial. Extending the Magnus embedding gives that $m(s_1^{n_1} s_2^{n_2} \cdots s_r^{n_r}) = \sgn(n_1+n_2+\cdots +n_r)$. But this is the composition of $\FF_2 \to \ZZ$ with a sign. Since the only example remaining was the Thomson's group $F$, I've removed it as well.}}

% Thomsom's group

% The case of subdiagonal algebras
\textbf{Left orderable groups and subdiagonal algebras.}
Finally, let us mention that the Hilbert transform $H=T_\sgn: L_p(\VN \sG) \to L_p(\VN \sG)$ of a left-orderable group can be considered as a particular case of the Hilbert transforms associated with \emph{subdiagonal algebras} which have been studied in \cite{Randrianantoanina1998Hilbert}, see also \cite[Section 8]{PiXu2003}. Such algebras were introduced by Arveson \cite{Arveson1967Analytic}. Indeed, let $(\M, \tau)$ be a von Neumann algebra with a n.s.f trace and let $\EE: \M \to \D \subseteq \M$ be a $\tau$-preserving conditional expectation onto a von Neumann subalgebra $\D \subseteq \M$. A \emph{finite subdiagonal algebra} $\H^\infty(\M) \subseteq \M$ with respect to $\EE$ is a weak-$\ast$ closed, non self-adjoint algebra such that
\begin{enumerate}[leftmargin=1.25cm,label={\rm \textbf{(\roman*)}}]
  \item $\EE(f \, g) = \EE(f) \, \EE(g)$,
  \item $\displaystyle \big\{f + g^\ast : f, g \in \H^\infty(\M)  \big\} = \M$ and 
  \item $\H^\infty(\M)  \cap (\H^\infty(\M) )^\ast = \D$.
\end{enumerate}
The notation comes from the fact that, for the algebra $L_\infty(\TT)$ and the integral as expectation, the Hardy space $\H^\infty(\TT)$ gives a finite subdiagonal algebra. Other examples are the $n \times n$ upper triangular matrices, which are a finite subdiagonal algebra with respect to the diagonal subalgebra $\ell_\infty^n \subseteq M_n(\CC)$, and more generally the Nest algebras \cite{Davidson1988Nest} of totally ordered families of projections. In this setting any $f \in \M$ admits a unique decomposition as
\[
  f = g + \delta + h^\ast, \quad \mbox{ with } g, h \in {\H_0^\infty(\M)}, \,\, \delta \in \D,
\] 
where $\H_0^\infty (\M) = \{f\in \H^\infty(\M): \EE(f)=0\}$.
The Hilbert transform associated to the finite subdiagonal algebra $\H^\infty (\M)$ is thus $H(f) = -i g + i h^\ast$. For these family of operators Randrianantoanina \cite{Randrianantoanina1998Hilbert} proved their weak type $(1,1)$, which after interpolation gives the optimal constant in terms of $p$. In the case of a discrete left-orderable group $\VN \sG$ with the subalgebra $\CC \1 \subseteq \VN \sG$ and the conditional expectation given by $f \mapsto \tau(f) \1$, we have that
\[
  \H^\infty_0 (\VN \sG)
  \, = \,
  \bigg\{ f = \sum_{e \prec g} \widehat{f}(g) \, \lambda_g  : f \in \VN \sG \bigg\}
\]
is a finite subdiagonal algebra whose Hilbert transform coincides with the multiplier $-i T_\sgn$ in Theorem \ref{thm:TheoremC}. Thus, by Corollary \ref{cor:OptConstantGroup}, Cotlar identities can be used to recover previously known results with optimal constant.

\section{Multipliers from Bass-Serre theory \label{S:BS}}
The theory of discrete groups acting on trees (without edge inversions) is very well understood due to the work of Serre. The interested reader can find more on this topic in \cite{Serre1980Trees}. Here we are going to briefly explain how it is possible to use this theory to build examples of multipliers satisfying Cotlar's identity and to understand  previously known examples like Free group multipliers from \cite{MeiRicard2017FreeHilb, MeiRicardXu2022Free} under a geometric lens.

% Graph of groups
%Recall that a \emph{graph of groups} is an object $\XX$ composed of a connected graph $X$ together with a group $\sG_x$ for each vertex $x \in \Vrt(X)$ and another group $\sH_y$ for every edge $y \in \Edg(X)$ satisfying that the group in the edge embeds into the groups of both of its extremes. More formally, we will consider that our graph is oriented, that both directions of the edge occur and that there can be multiple edges between two given vertices. As it is customary in this context, we will denote by $\bar{y}$ the \emph{reverse edge} and by $\ori(y) \in \Vrt(X)$ and $\tar(y) \in \Vrt(X)$ the \emph{origin} and \emph{target} vertices of the edge. We will say that $\Edg_+ \subseteq \Edg$ is an orientation of the graph if $\Edg_+$ contains either $y$ or $\bar{y}$. By definition, we have that $\sH_y = \sH_{\bar{y}}$ and that there are injective homomorphisms $\alpha_y: \sH_y \to \sG_{\tar(y)}$ and $\alpha_{\bar{y}}: \sH_y \to \sG_{\ori(y)}$. We will denote the image group $\alpha_y[\sH_y] \subseteq \sG_{\tar(y)}$ by $\sH_y^y$. Similarly $\sH_{\bar{y}}^{\bar{y}}$ will denote the image of $\alpha_{\bar{y}}$.

% Fundamental group in two ways. Sketch of their equivalence
Given a graph of groups $\XX$, that we will assume connected, its \emph{fundamental group} $\pi_1(\XX)$ can be defined in two ways, with the second one being more handy when constructing the Bass-Serre tree of $\XX$.
\begin{enumerate}[leftmargin=1.25cm, label={\rm \textbf{(D.\roman*)}}, ref={\rm (D.\roman*)}]
  \item \label{itm:DefFundamentalG1}
  The first construction, denoted by $\pi_1(\XX, x_0)$ and briefly explained in the introduction, consists of the subgroup of $F(\XX)$ given by words of type $c$, where $c$ is a closed path in $X$ that starts and ends at the base point $x_0 \in \Vrt(X)$. 
  
  \item \label{itm:DefFundamentalG2}
  For the second construction fix a \emph{spanning tree} $T \subseteq X$, that is a tree containing all vertices of $X$.
  The fundamental group $\pi_1(\XX;T)$ can be constructed as
  \[
    \pi_1(\XX ; T) 
    \, = \,
    F(\XX) / \llangle[\big] y : y \in \Edg(T) \rrangle[\big],
  \]
  where $\llangle y : y \in \Edg(T) \rrangle$ is the normal subgroup generated 
  by all the edges $y$ in the spanning tree.
\end{enumerate}
The first definition is independent on the choice of $x_0$, while the second is independent of the choice of spanning tree $T$, see \cite{Serre1980Trees}. Both are isomorphic and we will usually denote the resulting group just by $\pi_1(\XX)$. We are going to sketch why they are equal. All we have to see is that the canonical projection
\begin{equation}
  \label{eq:quotientPi}
  p: F(\XX) \longto \pi_1(\XX;T)
\end{equation}
restricts to an isomorphism of $\pi_1(\XX,x_0) \subseteq F(\XX)$ onto $\pi _1(\XX;T)$. The injectivity of the restricted map is clear, since no nontrivial closed loop can be fully contained in the spanning tree. To see the surjectivity let us construct a partial inverse $\phi: \pi_1(\XX;T) \to \pi_1(\XX,x_0)$.
Recall that for every point $x \in \Vrt(X)$ there is a unique path starting from $x_0$ and ending in $x$ contained within the spanning tree $T$. Let us denote by $\gamma_x \in F(\XX)$ the element associated to the path $\gamma_x = y_1 \, y_2 \cdots y_r$  with $o(y_1)=x_0$ and $t(y_r)=x$. If $y \in \Edg(X)$ we define $\phi$ as
\[
  \phi(y) = 
  \begin{cases}
    e & \mbox{ when } y \in \Edg(T)\\
    \gamma_{\ori(y)} \, y \, \gamma_{\tar(y)}^{-1} & \mbox{  when } y \not\in \Edg(T).    
  \end{cases}
\]
Observe that in both cases we obtain words associated to a closed path starting with $x_0$ like in \eqref{eq:WordofPathType}. Similarly, if $g \in \sG_x$ for a vertex $x \in \Vrt(X)$, then we define $\phi$ as
$\phi(g)  = \gamma_x \, g \, \gamma_x^{-1}$.
Straightforward computations give the desired properties of the partial inverse. Recall from the introduction that a word $g$ of type $c$ is in \emph{normal form} if, informally, it can not be shortened applying \eqref{eq:ExtraRelations}. More formally
\begin{itemize}
    \item Given a word $g = |c,r|$ of type $c$, it is in \emph{normal form} if either $m = 0$ and $r_0 \neq e$ or, in the case in which $m \geq 1$, it holds that for every $1 \leq j \leq m - 1$, $r_j \not\in \sH_{y_j}^{y_j}$ when $y_{j+1} = {\overline{y}_j}$. 
    \item Let $|c,r|$ and $|c,\mu|$ be two words of type $c$ with $r = (r_0, \, r_1, \, \dots, r_m)$ and $\mu = (\mu_0, \, \mu_1, \, \dots, \mu_m)$. They are said to be \emph{equivalent} if $\mu_0 = r_0 a_1^{\bar y_1}$ and $\mu_j=(a_j^{y_j})^{-1} r_j a_{j+1}^{\bar y_{j+1}}$, where $a_j \in \sH_{y_j}$ and $a_j^{y_j}$ and $a_j^{\bar{y}_j}$ are the corresponding images in $\sH_{y_j}^{y_j}$ and $\sH_{\bar{y_j}}^{\bar{y_j}}$.
\end{itemize} 
We have that two words of type $c$ in normal form are equivalent if and only if they represent the same element in $\pi_1(\XX, x_0)$. Observe also that any $g \in \pi _1 (\XX, x_0)$ admits a unique normal form up to equivalence. 
% reduced words
%Using the definition of $\pi_1(\XX)$ in Definition \ref{itm:DefFundamentalG1}, a notion of \emph{normal form} generalizing the intuitive notion of reduced word can be easily defined. It is said that a word $g = |c,r|$ of type $c$ is in normal form if either $m = 0$ and $r_0 \neq e$ or, in the case in which $m \geq 1$, it holds that for every $1 \leq j \leq m - 1$, $r_j \not\in \sH_{y_j}^{y_j}$ when $y_{j+1} = {\overline{y}_j}$. Let $|c,r|$ and $|c,\mu|$ be two words of type $c$ with $r = (r_0, \, r_1, \, \dots, r_m)$ and $\mu = (\mu_0, \, \mu_1, \, \dots, \mu_m)$. They are said to be \emph{equivalent} if $\mu_0 = r_0 a_1^{\bar y_1}$ and $\mu_j=(a_j^{y_j})^{-1} r_j a_{j+1}^{\bar y_{j+1}}$. Here $a_j \in \sH_{y_j}$ and $a_j^{y_j}$ and $a_j^{\bar{y}_j}$ are the corresponding images in $\sH_{y_j}^{y_j}$ and $\sH_{\bar{y_j}}^{\bar{y_j}}$. We have that $(c,r)$ and $(c,\mu)$ are equivalent if and only if $|c,r| = |c, \mu |$. This implies that any $g \in \pi _1 (\XX)$ admits a unique normal form up to equivalence. 

% Examples: Amalgamated free products, HNN extensions and fundamental groups of 1-dim simplicial complexes
The fundamental group of a graph of groups unifies several constructions. 
%\marginnote{{\color{red}ADR: Reduce the extension of this. Add motivation: these three points would be our main examples}}
%\begin{enumerate}[leftmargin=1.25cm, label={\rm \textbf{(\roman*)}}]
\begin{itemize}
  \item \textbf{Homotopy groups of $\mathbf{1}$-dimensional simplicial complexes.}
  Any graph $X$ is naturally a graph of groups by putting the trivial group in each vertex and edge. In this case $\pi_1(\XX)$ is naturally isomorphic to the (usual) fundamental group $\pi_1(X)$ of the associated $1$-dimensional simplicial complex, see \cite[Chapter 11]{Rotman1999groupsBook}.  
  
  \item \textbf{Amalgamated free products.} 
  Let $X$ be a connected graph with two vertices $x_1$ and $x_2$ associated to groups $\sG_1$ and $\sG_2$ and connected by a single edge $y$ to which a common subgroup $\sH$ of $\sG_1$ and $\sG_2$ is attached. Then, the fundamental group of this graph of groups is the amalgamated free product $\pi_1(\XX) = \sG_1 \ast_\sH \sG_2$, see Figure \ref{fig:FreeGroupGoG}, and when $\sH = \{e\}$ the usual free product is obtained.
    
  \item \textbf{HNN extensions.} Let $\sG = \langle S \, | \, R \rangle$ be a finitely presented group and $\sH_1$ and $\sH_2$ two isomorphic subgroups $\sH_1, \, \sH_2 \subseteq \sG$. Fix an isomorphism $\theta: \sH_1 \to \sH_2$. The \emph{HNN} extension of $\sG$ relative to $\alpha$, see \cite{Rotman1999groupsBook}, is the smallest extension of $\sG$ such that $\theta$ is implemented by an inner automorphism, i.e. $\theta(h) = t \, h \, t^{-1}$, for some $t$ and every $h \in \sH_1$. In the case of a finitely presented group, its HNN extension is given by
  \begin{equation}
    \label{eq:HNNDef}
    \big\langle S, \, t \, {\big|} \, R, \; t \, h \, t^{-1} = \theta(h) \, \big\rangle,
  \end{equation}
  where $h$ runs trough $\sH_1$. This construction can be obtained as the fundamental group of a graph of groups with a single vertex with group $\sG$ and an edge $y$ whose associated group is $\sH_1$ and the two inclusions are given by $\alpha_{y}(h) = h$ and $\alpha_{\bar{y}}(h) = \theta(h)$, see Figure \ref{fig:HNNGoG}.
\end{itemize}
\begin{figure}[t]
  \centering
  \begin{minipage}{.5\textwidth}
    \centering
    \includegraphics[width=.75\linewidth]{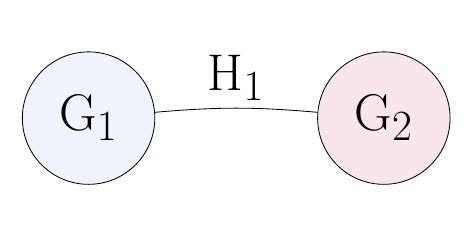}
    \captionof{figure}{Amalgamated free product.}
    \label{fig:FreeGroupGoG}
  \end{minipage}%
  \begin{minipage}{.5\textwidth}
    \centering
    \includegraphics[width=.55\linewidth]{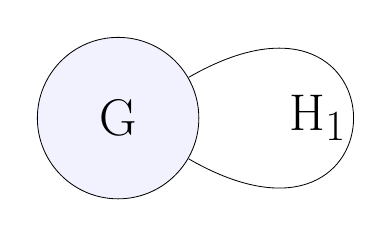}
    \captionof{figure}{HNN extension.}
    \label{fig:HNNGoG}
  \end{minipage}
\end{figure}

% Bass-Serre trees and the examples above 
Let $\XX$ be a graph of groups whose underlying graph we will denote by $X$ and let $\pi = \pi_1(\XX)$ be its fundamental group. We will recall the construction of the Bass-Serre tree $\widetilde{X}$ of $\XX$ whose vertices are
\[
  \Vrt(\widetilde{X}) = \bigsqcup_{x \in \Vrt(X)} \pi \slash \sG_{x},
\]
while its edges are given by
\[
  \Edg_+(\widetilde{X})
     = \bigsqcup_{y \in \Edg_+(X) } \pi \slash \sH_{\bar y}^{\bar y} 
  \;\;\;\;\; \text{and} \;\;\;\;\;
  \Edg_-(\widetilde{X})
     =  \bigsqcup_{y \in \Edg_-(X) } \pi \slash \sH_{\bar y}^{\bar y}, 
\]
where $\Edg_+(X)$ is a fixed orientation of the edges of $X$. In the definition above we are including only the edges of a fixed orientation of $\widetilde{X}$ induced by the orientation of $X$. 
%As in the rest of the section, we assume that the opposite edges are also included. 
Observe that there is a canonical projection $\widetilde{X} \onto X$ that maps vertices onto vertices and edges onto edges. We can define the endpoint of each edge as follows
\begin{equation}
  \begin{array}{lll}
    \ori(g \, \sH_{\bar y}^{\bar y}) = g \sG_{\ori(y)}, 
      & \tar(g \sH_{\bar y}^{\bar y}) = g g_y \sG_{\tar(y)} 
      & \mbox{ for }  \, y \in \Edg_+(X), \,\, g \in \pi, \\
    \ori( g \sH_{\bar y}^{\bar y}) = g g_y^{-1}\sG_{\ori(y)},
      & \tar(g \sH_{\bar y}^{\bar y})= g \sG_{\tar(y)}
      & \mbox{ for } y \in \Edg_-(X), \,\, g \in \pi,
  \end{array}
  \label{eq:DefBassSerreTree}
\end{equation}
where $g_{y}=p(y)$ is the image of $y$ under the canonical projection $p$ in \eqref{eq:quotientPi}. The group $\pi$ acts on $\widetilde{X}$ by left multiplication.
%
% Fundalmental theorem of Bass-Serre theory
%Let $\sG \acts T$ be a group acting on a tree without edge inversions. Then, it is possible to build a graph of groups by taking $X$ as the space of orbits $X = \sG \backslash T$. Furthermore, to each of the vertices of $X$, we can associate it with stabilizer group and do the same for edges. The fact that there are no edge inversions gives that the groups on the edges embed into the groups on the vertices and thus we have a graph of groups $\XX$ constructed from the action $\sG \acts T$. Serre's fundamental theorem attests that the group $\sG$, as well as the tree action $\sG \acts T$ can be recovered from $\XX$. Indeed, it holds that $\pi_1(\XX) \cong \sG$ and that the action of $\pi_1(\XX)$ on its Bass-Serre tree $\widetilde{X}$ is isomorphic to $\sG \acts T$. 
%
%It is interesting to point out that in the case in which $\XX$, as a graph of groups, have only trivial group on its vertices and edges, its Bass-Serre tree its just the \emph{universal covering space} $\widetilde{X}$ of the underlying simplicial graph $X$ and the action of $\pi_1(X)$ on $\widetilde{X}$ is the usual action by Deck transformations. We will further describe the Bass-Serre trees of the examples in Figures \ref{fig:FreeGroupGoG} and \ref{fig:HNNGoG} in the subsections below. 
%
The Bass-Serre tree construction also allows us to produce actions on trees for the fundamental group of any graph of groups. We will exploit this to obtain examples of Fourier multipliers satisfying Cotlar's identity. 

% Multipliers in this context
%\textbf{Multipliers from normal forms.}
We are ready to proof the main Theorem of this section.

\begin{proof}[Proof. (of Theorem \ref{prp:InitialSegment})]
For \ref{itm:InitialSegment1}, take $h \in \sG_{x_0}$, then it is immediate that $m(g h)=m(g)$. Therefore, it is enough to prove the identity for $g, h \not\in \sG_{x_0}$. Assume that $m(gh) \neq m(g)$, since otherwise the identity holds. Let us write $g$ in its normal form $g=r_0 \, y_1 \, r_1 \, y_2 \, r_2 \cdots y_{n} \, r_n$ with $r_0 \in g_0 \cdot \sH_{\bar y_1}^{\bar y_1}$. Since $m(gh) \neq m(g)$, $h$ must admit a normal form $b = r_n^{-1} y_n^{-1} \cdots y_1^{-1}r$ with $r=r_0^{-1}\cdot r'$ for some $r' \in g_0' \cdot  \sH_{\bar y_1}^{\bar y_1}$. Here $g_0, \, g_0' \in \sG_{x_0}$ and  $g_0^{-1}g_0' \notin \sH_{\bar y_1}^{\bar y_1}$. Moreover, since $ r_n^{-1} y_n^{-1} \cdots y_1^{-1}r_0^{-1}$ is a normal form of $g^{-1}$, by the definition of $m$, it is clear that $m(g^{-1})=m(h)$.
For \ref{itm:InitialSegment2}, notice that the left action $\sG_{x_0} \acts \sG_{x_0} / \sH_{\bar{y}}^{\bar{y}}$ is transitive for each $y$, therefore a left $\sG_{x_0}$-invariant function has to be constant in each of the elements of the disjoint union $W$.
\end{proof} 

Observe that, by the definition of the Bass-Serre tree of $\XX$ in equation \eqref{eq:DefBassSerreTree}, the set of edges in $\widetilde{X}$ connected to the vertex $\widetilde{x}_0 = e \cdot \sG_{x_0} \in \Vrt(\widetilde{X})$ is given, up to orientation, by the disjoint union of $\sG_{x_0} / \sH_{\bar{y}}^{\bar{y}}$, where, again, $y$ runs over every edge in $X$ starting with $x_0$. This already established that the functions depending only on the initial segment and the functions that are constant on the connected components of $\widetilde{X} \setminus \{\widetilde x_0\}$ are in natural and bijective correspondence. The following proposition, whose proof is immediate, asserts that  $m$ falls into Model \ref{mod:model1}. This provides an alternative proof of Theorem \ref{prp:InitialSegment} via Proposition \ref{prp:modelCotlar}.\ref{itm:modelCotlar1}.

\begin{proposition}
  \label{prp:BassSerreModel1}
  Let $\pi = \pi _1 (\XX,x_0)$ be the fundamental group of a connected graph of groups $\XX$ 
  and let $\widetilde{X}$ be its Bass-Serre tree.
  \begin{enumerate}[label={\rm \textbf{(\roman*)}}]
    \item Every symbol $m: \pi \to \CC$ depending on the first segment satisfies that
    \[
      m( g ) = \widetilde{m}(g \cdot \widetilde{x}_0), 
    \]
    where $\widetilde{m}: \widetilde{X} \to \CC$ is the function constant 
    over the connected components of $\widetilde{X} \setminus \{\widetilde{x}_0\}$ 
    that is naturally associated to \eqref{eq:LiftFirstSegment} and $\widetilde{x}_0$ 
    is a vertex of $\widetilde{X}$ labeled by $\sG_{x_0}$.
    
    \item If furthermore $m$ depends only on the edge of the starting segment, then
    $\widetilde{m}:\widetilde{X} \setminus \{\widetilde{x}_0\} \to \CC$ has the same 
    value over each two connected components $\widetilde X_\alpha$ and $\widetilde X_\beta$ such that
    $r \cdot \widetilde X_\alpha = \widetilde X_\beta$, for some $r \in \sG_{x_0}$.
  \end{enumerate}
\end{proposition}

In the case in which $\sG_{x_0}$ is Abelian, the more restrictive condition that $m:\pi \to \CC$ depends on the initial edge can be removed and a multiplier theorem still holds for symbols depending on the starting segment. Let $m$ be a symbol depending on the starting segment and $\widetilde{m}: W_{x_0} \to \CC$ be its lift to the space of initial segments defined in \eqref{eq:DefW}. We will denote by $\widetilde{m}_y$ the restriction of $\widetilde{m}$ to $\sG_{x_0}/\sH_{\bar{y}}^{\bar{y}}$. We have the following result.

\begin{theorem}
    \label{thm:InitialSegmentG0Abelian}
    Let $\pi=\pi_1(\XX,x_0)$ be fundamental group of a graph of groups with $\sG_{x_0}$ Abelian. In order to lighten the notation, let us denote $\sG_{x_0}/\sH_{\bar{y}}^{\bar{y}}$ by $\sG_y$. For each $y \in \Edg(X)$ with $\ori(y) = x_0$ and each character $\chi: \sG_y \to \TT$ we define the multiplier $\sigma_{y, \, \chi}$ given by $\sigma_{y, \, \chi}(g) = \chi(r_0 \, \sH_{\bar{y}}^{\bar{y}})$, whenever $g$ is represented by a word starting by $r_0 \, y$, and $0$ otherwise. The following results hold
    \begin{enumerate}[label={\rm \textbf{(\roman*)}}, label={\rm {(\roman*)}}]
        \item \label{itm:InitialSegmentG0Abelian.1} The map $\rho_{y}: L_p(\VN \pi) \to L_p(\widehat{\sG}_y;L_p(\VN \pi))$, defined by $f \mapsto (T_{\sigma_{y,\chi}}(f))_{\chi \in \widehat{\sG}_y}$ satisfies
        \[
            \| f \|_p \leq \big\| \rho_{y}(f) \big\|_{L_p(\widehat{\sG}_y;L_p(\VN \pi))} \lesssim \Big( \frac{p^2}{p-1} \Big)^\beta \| f \|_{p}, \;\; \text{ where } \beta = \log_2(1+\sqrt{2}),
        \]
        \item \label{itm:InitialSegmentG0Abelian.2} For a general symbol $m: \pi \to \CC$ depending on the starting segment, it holds that
        \[
            \big\| T_m: L_p(\VN \pi) \to L_p(\VN \pi) \big\|
            \lesssim \Big( \frac{p^2}{p-1} \Big)^\beta \,  \max_{\ori(y) = x_0 } \Big\{ \big\| T_{\tilde{m}_y}: L_p(\widehat{\sG}_y) \to L_p(\widehat{\sG}_y) \big\|_\cb \Big\}.
        \]
    \end{enumerate} 
\end{theorem}

\begin{proof}
    The proof is almost immediate. For \ref{itm:InitialSegmentG0Abelian.1} first notice that $m_{y,\chi}$ depends on the starting segment and that, given $a \in \sG_{x_0}$, $\sigma_{y,\chi}$ satisfies that $\sigma_{y,\chi}(a \, g ) = \chi(a) \, \sigma_{y,\chi}(g)$. Thus, since it satisfies \eqref{eq:CotlarNC} relative to $\sG_{x_0}$, by Remark \ref{rmk:NonInvariantCharacters} we have that it is a bounded multiplier in $L_p$ for $1<p<\infty$ satisfying bound \eqref{eq:TheoremABound} in Theorem \ref{thm:TheoremA}. The fact that $\widehat{\sG}_y$ is compact implies that $\rho_y$ is bounded with the same bound as $T_{\sigma_{y,\chi}}$. For the lower bound, notice that, since the Haar measure of $\widehat{\sG}_y$ is a probability measure, $\1_{\widehat{\sG}_y}$ is of norm one on $L_{p'}(\widehat{\sG}_y)$. But evaluating against it gives
    \[
        \big\| \rho_y(f) \big\|_{L_p(\widehat{\sG}_y;L_p(\VN \pi))} \geq \bigg\| \int_{\widehat{\sG}_y} T_{m_{y,\chi}}(f) \, d \widehat{\mu}(\chi) \bigg\|_{L_p(\VN \pi)} = \| f \|_p.
    \]
    This shows that $\rho_y$ is a quasi-isometry which gives \ref{itm:InitialSegmentG0Abelian.1}. For \ref{itm:InitialSegmentG0Abelian.2}, just notice that if $m$ is a symbol depending on the starting segment and that is $0$ outside $\sG_y$, then the following diagram commutes
    \[
        \begin{tikzcd}[column sep=2cm, row sep=0.8cm, every label/.append style = {font = \small}]
        L_p(\VN \pi) \arrow{d}{T_{m}} \arrow{r}{\rho_y} & L_p(\widehat{\sG}_y;L_p(\VN \pi)) \arrow{d}{T_{\tilde{m}_y}\otimes \Id} \\
        L_p(\VN \pi) \arrow{r}{\rho_y} & L_p(\widehat{\sG}_y;L_p(\VN \pi))
        \end{tikzcd}
    \]
    and, together with point \ref{itm:InitialSegmentG0Abelian.2} this yields the result. In the general case of $m$ depending on the starting segment, we can decompose it as a sum of symbols $\widetilde{m}{|}_{\sG_y}$ and build a direct sum map
    \[
      \rho: L_p(\VN \pi) \longrightarrow \bigoplus_{\ori(y) = x_0}{} L_p(\widehat{\sG}_y;L_p(\VN \pi))
    \]
    by taking $\rho_y$ in each of the components. 
\end{proof}
%\begin{remark}\normalfont
%  \label{rem: bdness}
%  The above proposition together with Proposition \ref{prp:InitialSegment} and Theorem \ref{thm:TheoremB} imply that every symbol $m$ on $\pi=\pi _1 (\XX,x_0)$ depending only on the edge of the starting segment gives a Fourier multiplier $T_m$ bounded on $L_p(\VN\pi)$ for any $1<p<\infty$. 
%\end{remark}

%\begin{remark} \normalfont
  % Not necessarily left G0 invariant.
%  Recall that the multiplier $m: \pi \to \CC$ above depending on the first segment is not necessarily left $\sG_{x_0}$-invariant. In order to guarantee the $L_p$-boundedness, we will need to check this condition separately. In the case of $\sG_{x_0}$ being Abelian and $m{|}_{\sG_{x_0}}$ is a character a different argument can be employed, see \cite{}.
%\end{remark}

% Example: The free group
Now, we have all the tools required to illustrate our theory with the examples. 

\textbf{Examples: Amalgamated free products.}
This first example is already known, but our proof gives a concise geometric interpretation about it. Let $\{\sG_i: i \geq 1 \}$ be a family of groups and $A$ a common subgroup. Let us denote the injective homomorphisms by $\alpha_i : A \rightarrow \sG_i$, for $i \geq 1$. Now, consider the graph of groups $\XX$ in Figure  \ref{fig:AmFreeProdGoG}, which is based on a tree $X$ connecting the points $x_0$, whose group is $A$, with the point $x_i$, $i \geq 1$, whose groups are $\sG_i$. Meanwhile, the groups on the edges $y_i$ are all isomorphic to $A$ with the embedding into $\sG_{x_0} = A$ being the identity map and the embedding into $\sG_i$ being $\alpha_i$. 

\begin{center}
  \begin{figure}[ht]
    \centering
    \includegraphics[scale=1.3]{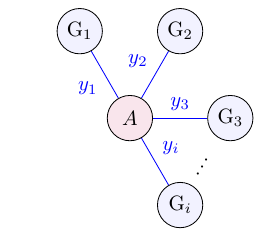}
    \caption{Graph of groups for the amalgamated free product.}
    \label{fig:AmFreeProdGoG}
  \end{figure}
\end{center}

In this case, we have that $\pi = \pi_1(\XX,x_0)$ is the free product of $\{\sG_i : i \geq 1 \}$ with $A$ as amalgam  $\pi= {\freePr_{i, \, A}} \sG_i$, see \cite[Theorem 9]{Serre1980Trees}. Functions $m: \pi \to \CC$ that depend on the starting edge in this context are given by the span of $\{L_i\}_{i \geq 1}$, where each $L_i$ is the Fourier multiplier associated to the characteristic function of the words starting by $a \, s_{i}$, where $s_{i} \in \sG_i$. These operators are an example of the \emph{free Hilbert transforms} that have been studied in the pioneering work of Mei and Ricard \cite{MeiRicard2017FreeHilb}. %In particular, given a collection of signs $\epsilon = (\epsilon_i)_i$, they studied the operator $H_\epsilon: \CC[\pi] \to \CC[\pi]$
%\begin{equation*}
%  H_\epsilon (f)
 % \, = \, 
 % \sum_{i \geq 1} \epsilon_i \, L_i(f),
%\end{equation*}
%where $\epsilon_i =\pm 1$ and the operator $L_i$ is the Fourier multiplier whose symbol $m_i$ is the characteristic function of the words starting by $a \, s_{i_0}$, where $s_{i_0} \in \sG_i$. Note that this operator falls into Model \ref{mod:model1} for the action of $\pi$ on its Bass-Serre tree associated to $\XX$ in Figure \ref{fig:AmFreeProdGoG}. Moreover, the symbol of $H_\epsilon$ is a function depending only on the outgoing edges $y_i$ for $i\geq 1$. By remark \ref{rem: bdness}, $H_\epsilon$ is bounded on $L_p({\VN (\freePr_{i, \, A}} \sG_i))$ for any $1<p<\infty$. 
Note that free groups can be obtained from other graphs of groups. Let us pick the free group of two generators $\FF_2$ for simplicity. We can represent $\FF_2$ as the fundamental group of the graph of groups $\XX$ given by two vertices $\{x_0, x_1\}$ whose associated groups are isomorphic to $\ZZ$ and joined by a single edge with trivial group in it. 

Let us denote the generator associated to $x_0$ by $a$ and the generator associated to $x_1$ by $b$. Since $\ZZ =\langle a \rangle$ is Abelian, we can apply Theorem \ref{thm:InitialSegmentG0Abelian}. Indeed, functions depending on the starting segment are of the form $m( s_1^{n_1} \, s_2^{n_2} \, \cdots \, s_r^{n_r}) = \widetilde{m}(n_1)$ if $s_1 =a$ and $0$ otherwise. Changing the roles of $x_0$ and $x_1$ and combining the two multipliers obtained we get symbols of the form
\begin{equation}
  \label{eq:SymbolStartingLetter}
  m \big( s_1^{n_1} \, s_2^{n_2} \, \cdots \, s_r^{n_r} \big)
  \, = \,  \begin{cases}
  \widetilde{m}_1(n_1),\; \text{ if }\; s_1=a\\
  \widetilde{m}_2(n_1),\; \text{ if }\; s_1=b.
  \end{cases} 
\end{equation}
Then, Theorem \ref{thm:InitialSegmentG0Abelian} implies that
\[
    \big\| T_m: L_p(\VN \FF_2) \to L_p(\VN \FF_2) \big\| \lesssim \big(p \cdot p'\big)^\beta \, \max_{i=1,2} \Big\{ \big\| T_{\tilde{m}_1}:L_p(\TT) \to L_p(\TT) \big\|_\cb \Big\}.
\]
With this, we partially recover the multiplier theorems for initial segments of the free group in \cite{MeiRicardXu2022Free}. Our approach has the advantage of giving a geometric interpretation as follows. Consider the Bass-Serre tree $\widetilde{X}$ of this graph of groups, see Figure \ref{fig:BassSerreTreeF2}, it is easy to see that group elements in $\FF_2$ can only map pink vertices to pink vertices and blue ones to blue ones. Denote the set containing the blue vertices in the first layer and the pink vertex labeled by $\ZZ$ in the middle by $B_1$. It is straightforward  to realize that any function on $\widetilde{X}$ that is constant along the connected components of $\widetilde{X} \setminus B_1$ induces a symbol of the form \eqref{eq:SymbolStartingLetter}. 

The fact that the starting letter induces bounded multipliers suggest the enticing possibility of simultaneously generalizing our results for the starting segment in Theorem \ref{prp:InitialSegment} and \ref{thm:InitialSegmentG0Abelian} and the main result in \cite{MeiRicardXu2022Free} to multipliers depending on the starting $k$-word of the normal form in the context of graph of groups, see \cite{Runlian2023Chinese}.

%{\marginpar{\color{blue} In this Bass-Serre tree, the blue vertex labelled by Z in the first layer is missing.}}
\begin{center}
  \begin{figure}[ht]
    \centering
    \includegraphics[width=0.85\textwidth]{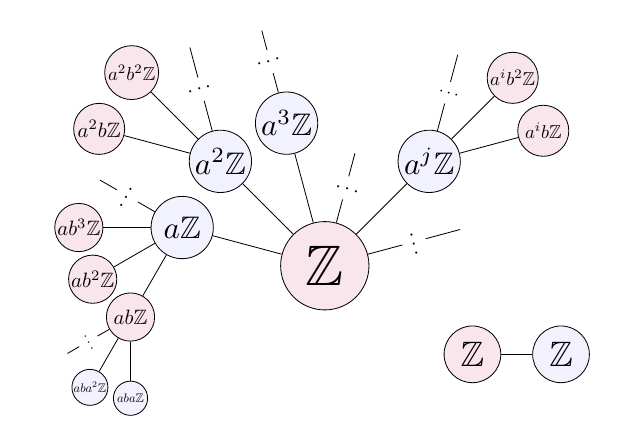}
    \caption{Graph of groups whose fundamental group is $\FF_2$ (in the corner) and its Bass-Serre tree.
    For readability, we have omitted branches going from the root towards vertices labeled 
    $a^{k} \ZZ$ with $k \leq 0$.}
    \label{fig:BassSerreTreeF2}
  \end{figure}
\end{center}

%Now we will show their definition of the symbol is a special case of our topological  definition \eqref{eq: def of symbol}, or more precisely the algebraic  definition \eqref{eq: def of m} when $\pi= \ast_{A}\G_i$ and we consider its action on its Bass-Serre tree. Notice that  definition \eqref{eq: def of symbol} goes beyond  \cite{MeiRicard2017FreeHilb} because we can deal with groups which are not amalgamated free products, for instance the group $\Gamma$ given in subsection \ref{counter example}.

%Let $c_{P_{i_j}}$  be the geodesic path in $Y$ joining $P_0$ to $P_{i_j}$ and $y_1, y_2, \ldots y_{i_j}$ be the edges in $c_{P_{i_j}}$.  Set $\gamma_{P_{i_j}}=y_1\cdots y_{i_j}$. By the identification  between $\pi_1(\G, Y,Y)$ and $\pi_1(\G,Y,Q)$, we know that $g$ also admits a normal form in $\pi_1(\G,Y,Q)$:

% Remark on Runlian's result generalizing Mei-Ricard-Xu and cite

% Example: The Baumslag solitar 
\textbf{Examples: HNN extensions.}
Suppose $\sH_1$ and $\sH_2 \subseteq \sG$ are two subgroups isomorphic under $\theta: \sH_1 \to \sH_2$ and let $\pi$ be the HNN-extension of $\sG$ relative to $\theta$, given in \eqref{eq:HNNDef}. Observe that the HNN-extension can be written as $\pi = H \rtimes T$ is the semi-direct product of the infinite cyclic group $T$ generated by $t$ and the normal subgroup $H$ generated by $t^n \sG t^{-n}$, $n \in \ZZ$. In particular, $\pi$ has a quotient is isomorphic to $\ZZ$. 

We have already seen, see Figure \ref{fig:HNNGoG}, that $\pi = \pi_1(\XX,x_0)$ is the fundamental group of a graph of groups based on a single loop. By the definition of fundamental group, every element of $\pi$ is represented by word $r_0, t^{e_1} \, r_1 \, \cdots \, t^{e_k} \,, r_k$ with $k\geq 0$, $e_i=\pm 1$ and $r_i \in \sG$. A word in this form will be reduced if it contains neither a substring of the form $t a t^{-1}$ with $ a \in \sH_1$ nor one of the form $t^{-1} b t$ with $b \in \sH_2$. Fix coset representatives of $\sG/\sH_1$ and $\sG/\sH_2$, for any $g\in \pi$, there exists a unique reduced word such that 
\begin{equation}
  \label{eq:HNNnormalFont}
  g \, = \, r_0\,  t^{e_1} \, r_1 \, \cdots \, t^{e_k} \, r_k
\end{equation}
with $r_0 \in \sG$, $r_i \in \sG/\sH_1$ if $e_i=1$, and $r_i \in \sG/\sH_2$ if $e_i=-1$. We are going to give two algebraic forms for Fourier multipliers satisfying \eqref{eq:CotlarNC}, one falling within Model \ref{mod:model1} and another within Model \ref{mod:model2} by considering the action of the HNN-extension $\pi$ on its Bass-Serre tree $\widetilde{X}$ by left multiplication.
For the first one, we choose the root  $\widetilde{x}_0$ to be the vertex labeled by $\sG$. Note that if we set the orientation $\Edg_+(X)=\{t\}$, then in the induced orientation of $\widetilde{X}$ there are $[\sG: \sH_2]$ many edges starting with $ \widetilde{x}_0$ and $[\sG: \sH_1]$ many edges ending in $\widetilde{x}_0$. 
%Denote by $q_1: \sG \to \sG/\sH_1$ and $q_2:\sG \to \sG / \sH_2$ the respective quotient maps. 
It is immediate that a function $m: \pi \to \CC$ depends on the starting segment iff there is a function $\widetilde{m}: ( \sG/\sH_1 ) \sqcup ( \sG/\sH_2) \to \CC$ such that
\begin{equation}
  \label{eq:CotlarMultiHNN0}
  m(g) 
  \, = \,
  \begin{cases}
    \widetilde{m}_{1} \big( r_0 \cdot \sH_1 ) & \mbox{ when } \;e_1 = +1\\
    0 & \mbox{ when } g \in \sG\\
    \widetilde{m}_{2} \big( r_0 \cdot \sH_2 ) & \mbox{ when } \;e_1 = -1,
  \end{cases}
\end{equation}
where $\widetilde{m}_i$ is the restriction of $\widetilde{m}$ to $\sG/\sH_i$ for $i \in \{1,2\}$ and $g$ is equal to its normal form like in \eqref{eq:HNNnormalFont}. By Proposition \ref{prp:BassSerreModel1}, we have that $m$ is left $\sG$-invariant if and only if it depends only on the first edge in the normal form \eqref{eq:HNNnormalFont}, which can only be $t$ or $t^{-1}$ in our setting, ie: 
\begin{equation}
  \label{eq:CotlarMultiHNN1}
  m(g) \, = \, C_1 \1_{\{e_1 = +1\}} + C_2 \, \1_{\{e_1 = +1\}}
\end{equation} 
Observe that, if $\sG$ is Abelian, we can deal with the boundedness of \eqref{eq:CotlarMultiHNN0}.

Now, we can use the definition in Model \ref{mod:model2} to the action $\pi \acts \widetilde{X}$. Choose $\widetilde{x}_0$ as the root and let $\widetilde{X}_0 \subseteq \widetilde{X}$ be the connected component separated from the root by the edge $\sH_1$. It is easy to check the vertices in the connected component of $\sH_2$ always take the form $g\cdot \sG$ with the expression \eqref{eq:HNNnormalFont} of $g$ starting with $t$. Hence we get the following symbol $\varphi: \pi \to \CC$:
\begin{equation}
  \label{eq:CotlarMultiHNN2}
  \varphi(g) 
  \, = \,
  \begin{cases}
    0   & \text{ if } g \in \sH_2 \subseteq \sG \\
    D_1 & \text{ if } r_0 = e \mbox{ and } e_1 = 1\\
    D_2 & \mbox{ otherwise, }
  \end{cases} 
\end{equation} 
where $D_1$ and $D_2$ are two different constants. 

\begin{corollary}
  \label{cor:HNN2Models}
  Let $\pi$ be the HNN extension of $\sG$ with respect to $\theta$ as before.
  \begin{enumerate}[label={\rm \textbf{(\roman*)}}]
    \item Let $m: \pi \to \CC$ be like in \eqref{eq:CotlarMultiHNN1}. Since $m$ depends on the starting edge, $T_m$ is $L_p$-bounded for $1 < p < \infty$ and satisfies bound \eqref{eq:TheoremABound} by Theorem \ref{prp:InitialSegment}.
    \item If $\sG$ is Abelian, and the symbol $m$ in \eqref{eq:CotlarMultiHNN0} lift to $\widetilde{m}: (\sG/\sH_1) \sqcup (\sG/\sH_2) \to \CC$, then, by Theorem \ref{thm:InitialSegmentG0Abelian}, it holds that
    \[
        \big\| T_m: L_p(\VN \pi) \to L_p(\VN \pi) \big\|
            \lesssim \Big( \frac{p^2}{p-1} \Big)^\beta \, \max_{i=1,2} \Big\{ \big\| T_{m_{i}}: L_p(\widehat{\sG/\sH_i}) \to L_p(\widehat{\sG/\sH_i}) \big\|_\cb \Big\}.
    \]
    \item Let $\varphi: \pi \to \CC$ be like in \eqref{eq:CotlarMultiHNN2}, then $T_\varphi$ satisfies \eqref{eq:CotlarNC} relative to $\sH_2 \subseteq \sG$ and is left $\VN \sH_2$-modular. Thus $T_\varphi$ is bounded in $L_p$ with bound \eqref{eq:TheoremABound}.
  \end{enumerate}
\end{corollary}

The result above can be illustrated in the particular case of the Baumslag-Solitar group
\[
  \BS(n,m)
  \, = \,
  \big\langle r, t \, \big{|} \, t \, r^m \, t^{-1} = r^n \big\rangle,
\]
which can be seen as a HNN extension of $\ZZ = \langle r \rangle$ with respect to the map $\theta$ that sends $r^{m k} \mapsto r^{n k}$ that establishes an isomorphism between the subgroups $m\ZZ$ and $n\ZZ \subseteq \ZZ$. It holds that $\ZZ / n\ZZ \cong \ZZ_n$ and $\ZZ / m \ZZ \cong \ZZ_m$. We can take representatives $\{0, \, 1, \, \dots \, n - 1\}$ and $\{0, \, 1, \, \dots \, m - 1\}$ of the quotients. The unique normal form of $g$ is given by
\begin{equation}
  \label{eq:normalformBS} 
  g= r^{k_0} t^{e_1} r^{k_1} \cdots t^{e_m} r^{k_m}
  \quad \text{ with }  k_0 \in \mathbb{Z}, e_1, \ldots e_j \in \{\pm 1\}
\end{equation}

if $e_{j}=1$, then $k_j\in \ZZ \slash {m\ZZ}$ for $j \geq 1$, if  $e_{j}=-1$, then  $k_j \in \ZZ \slash {n\ZZ}$ for $j \geq 1$. Multipliers that depend on the starting segment $\varphi: \pi \to \CC$ lift to a function $\widetilde{\varphi}: (\ZZ/n\ZZ) \sqcup (\ZZ/m\ZZ) \to \CC$ as in \eqref{eq:CotlarMultiHNN0}. Using that $\ZZ$ is Abelian and Corollary \ref{cor:HNN2Models} gives that
\[
    \big\| T_\varphi: L_p(\VN \BS(n,m)) \to L_p(\VN \BS(n,m)) \big\|
    \lesssim \max\{ n, m \}^{\big| \frac1{2} - \frac1{p}\big|} \, \big( p \cdot p' )^\beta \| \varphi \|_\infty,
\]
where we have also use that the multiplier norm of $\varphi: \ZZ/\ell \ZZ \to \CC$ on $L_p(\ZZ/\ell\ZZ)$ is bounded by $\ell^{|\frac1{2}-\frac1{p}|} \| \varphi \|_\infty $. Now, let explore multipliers on $\BS(n,m)$ coming from Model \ref{mod:model2} as in \eqref{eq:CotlarMultiHNN2}. Let $\XX$ be the single loop graph of groups associated to $\BS(n,m)$ and $\widetilde{X}$ be its Bass-Serre tree depicted in Figure \ref{fig:BSTreeeofBaumslagSolitar}. The two directions of the loop give us the two subgroups $m\ZZ, \, n\ZZ$, and by definition of the Bass-Serre tree you get that $\widetilde{X} \setminus \{\widetilde{x}_0\}$ has $n + m$ connected components, see Figure \ref{fig:BSTreeeofBaumslagSolitar}. Now, we select the connected component of the vertex $t \sG$. The symbol obtained has a value that depends on whether the normal form \eqref{eq:normalformBS} starts with $t$ or not.
\begin{center}
  \begin{figure}[t]
    \centering 
    \includegraphics[width=0.75\textwidth]{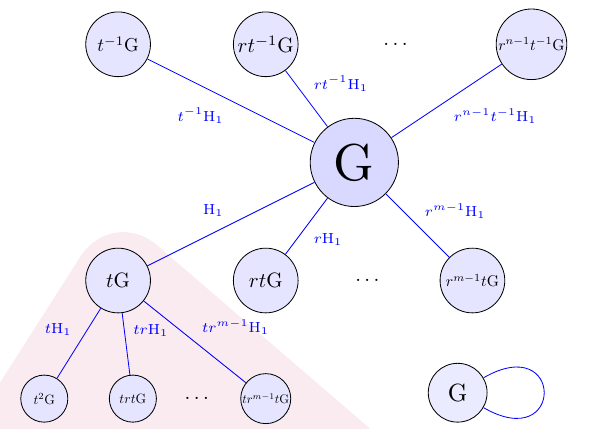}
    \caption{Graph of groups whose fundamental group is the Baumslag-Solitar group 
    (in the corner) and its Bass-Serre tree. Here $\sH_1 = n\ZZ$.}
    \label{fig:BSTreeeofBaumslagSolitar}
  \end{figure}
\end{center}
Observe that $\BS(n,m)$ has an Abelian quotient $q: \BS(n,m) \to \ZZ$ given by moding out the generator $t$. Composing $q$ with a sign would give an example of a multiplier satisfying Cotlar's identity. Nevertheless, the multipliers studied above, \eqref{eq:CotlarMultiHNN0} and \eqref{eq:CotlarMultiHNN2} do not come from the signs of Abelian quotient. Similarly, even though $\BS(n,m)$ clearly fails Serre's property (FA) there is no easy way to write the Baumslag-Solitar group $\BS(m,n)$ as an amalgamated free product in a way that will allow us to interpret the multiplier formulas in \eqref{eq:CotlarMultiHNN1} and \eqref{eq:CotlarMultiHNN2} as examples of Mei and Ricard's free Hilbert transforms as in \cite{MeiRicard2017FreeHilb}. 

% A left-orderable group with Serre's property (FA)
\textbf{Examples with Serre's Property (FA).}  We will present here the examples of groups having multipliers satisfying Cotlar's identity and fitting into our Model \ref{mod:model1} while having Serre's property $\mathrm{(FA)}$. Recall that a group has Serre's property $\mathrm{(FA)}$ ---the $\mathrm{A}$ stemming from the French word for tree, \textit{arbre}--- if and only if every isometric action on a simplicial tree has global fixed point. This property admits a closed characterization. Indeed, a countable group $\sG$ has property $\mathrm{(FA)}$ if and only if the following conditions are satisfied: $\sG$ is not an amalgamated free product, $\sG$ has no quotient isomorphic to $\ZZ$ and $\sG$ is finitely generated. 

First, we will give an example of a left-orderable group with property $\mathrm{(FA)}$. This gives an example for which Cotlar's identity was previously unknown. On the other hand, its associated Hilbert transform can be proven to be of weak type $(1,1)$ due to the theory of Hilbert transform on finite subdiagonal algebras \cite{Randrianantoanina1998Hilbert}.

Denote by $D(2,3,7)$ the $(2,3,7)$-triangle group, which is of particular interest in hyperbolic geometry. It is the group of orientation-preserving isometries of the tiling by the $(2,3,7)$ Schwarz triangle. $D(2,3,7)$ admits a presentation 
\[
  D(2,3,7) 
  \, = \, 
  \big\langle x,y \, {\big|} \, x^2=y^3=(xy)^7=1 \big\rangle.
\]
From the above presentation, it is easy to see $D(2,3,7)$ is a quotient of the modular group $\PSL(2,\ZZ) \cong \ZZ_2 \ast \ZZ_3$. It is also isomorphic to a discrete subgroup of $\PSL(2,\RR)$. 
Now let us consider the lifting of $(2,3,7)$-triangle group to the universal cover of $\PSL(2,\RR)$ and denote it by $\Gamma$. We know from \cite{Bergman1991Orderable} that $\Gamma$ has a presentation
\[
  \Gamma \, = \, \big\langle x, y, z \, {\big|} \, x^2=y^3=z^7=xyz \big\rangle.
\]
The fact that $\Gamma$ has Serre's property (FA) and simultaneously acts on the real line by homeomorphisms is already known. We gather the different pieces in the proposition below.

\begin{proposition}
  The group $\Gamma$ has property ${\rm (FA)}$ and is left-orderable. Therefore, 
  its sign Hilbert transform $H = T_\sgn$ satisfies \eqref{eq:CotlarNC} and thus $\| H: L_p(\VN \Gamma) \to L_p(\VN \Gamma) \| \, < \, \infty$ for $1 < p < \infty$. 
\end{proposition}

\begin{proof}
It follows from the presentations of $D(2,3,7)$ and its covering group $\Gamma$ that the kernel  of the covering homomorphism is isomorphic to $\ZZ$. That is, we have a short exact sequence
\[
  1 \to \ZZ \to \Gamma \to D(2,3,7) \to 1.
\]
Note that $\Gamma$ is perfect \cite{Bergman1991Orderable}, i.e. it does not contain any nontrivial Abelian quotient, then  \cite[Proposition 3.2]{CornulierKar2011} tells us that $\Gamma$ has property (FA) if and only if $D(2,3,7)$ has property (FA). Since $D(2,3,7)$  has property (FA), see \cite[p. 61]{Serre1980Trees}, we deduce $\Gamma$ also has property (FA). 
Moreover, the action of $\PSL(2,\RR)$ denoted by $\widetilde{\rm PSL} (2,\RR)$ on the circle lifts to an action of $\widetilde{\PSL}(2,\RR)$  on $\RR$ by orientation-preserving homeomorphism, so in this way, the $\widetilde{\PSL} (2,\RR)$ embeds into $\Homeo_+(\RR)$, see \cite{Khoi2003cut} for more details, it naturally admits a left-invariant total order. 
Therefore, $\Gamma$ is also left-orderable. Applying Theorem \ref{thm:TheoremC}, we conclude that $\Gamma$ admits a nontrivial Hilbert transform.
\end{proof}

Recall that this gives a new example of group multiplier satisfying Cotlar's identity but not a new example of a group multiplier being $L_p$ bounded since it falls within the theory described in \cite{Randrianantoanina1998Hilbert}. In order to obtain new examples outside both the classical theory, the theory of free Hilbert transforms \cite{MeiRicard2017FreeHilb} and the theory of subdiagonal algebras we will need an example of a group with property (FA) acting without global fixed points on an $\RR$-tree other than $\RR$ itself. There are a  few examples on that direction in the literature, see \cite{Minasyan2016Rtrees}, but the complexity of the constructions makes finding explicit formulas for the multipliers more involved.
%We recall as well that there are known examples, see \cite{DunwoodMinasyan2013FA},of groups with property $\mathrm{(FA)}$ that act without global fixed points on a complicated $\RR$-tree not homeomorphic to $\RR$. Applying Model \ref{mod:model1} to such example will wive multipliers outside both Mei and Ricard results as well as Randrianantoanina's.

%The group $\Gamma$ cannot be written as an amalgamated free product since Serre \cite{Serre80} shows that for a countable group $\G$, $\G$ has property (FA) if and only if the following three conditions are satisfied:
%\begin{itemize}
%  \item $\G$ is not an amalgamated free product;
%  \item $\G$ has no quotient isomorphic to $\Z$;
%  \item $\G$ is finitely generated. 
%\end{itemize}
%In this sense, our Hilbert transform on $\Gamma$ is completely out of the scope of Hilbert transforms on amalgamated free products  and HNN-extensions which we will discuss in subsections \ref{section: amalgam} and \ref{section: HNN}.

\section{Hilbert transforms over lattices of \texorpdfstring{$\SL_2(\RR)$} and \texorpdfstring{$\SL_2(\CC)$} \label{S:SL2}}

The modular group $\PSL_2(\ZZ)$ is a discrete subgroup of $\PSL_2(\RR)$. Over $\PSL_2(\RR)$ there is a very natural multiplier symbol, playing a role analogous to that of the Hilbert transform in $\RR$. This is given by equation  \eqref{eq:HilbertContSLR}. In the statement of following theorem we will abuse our notation and denote by the elements in $\PSL_2(\RR)$, which are classes of matrices up to sign $\pm \Id$ by simply matrices. We will also denote by $S$ and $T$ the matrices
\[
  S =
  \begin{pmatrix}
    0 & -1\\
    1 & 0
   \end{pmatrix}
   \quad
   \mbox{ and }
   \quad
   T = 
   \begin{pmatrix}
     1 & 1\\
     0 & 1
   \end{pmatrix}.
\]
It holds that $S^2 = -\Id$ and $(ST)^3=-\Id$, which in the quotient group gives that $S^2=\Id$ and $(ST)^3=\Id$. We have the following result.

\begin{proposition}
  \label{prp:LpBddSL2Z}
  Let $m{|}_{\PSL_2(\ZZ)}: \PSL_2(\ZZ) \to \CC$ be the restriction of \eqref{eq:HilbertContSLR} to the modular group. Then $m$ satisfies \eqref{eq:CotlarNC} with respect to the subgroup $\sG_0 = \{ \Id, S \}$ and as a consequence
  \[
    \big\| T_{m}:L_p(\VN \PSL_2(\ZZ)) \to L_p(\VN \PSL_2(\ZZ)) \big\|
    \, \lesssim \,
    \Big( \frac{p^2}{p - 1} \Big)^\beta
    \quad
    \mbox{ where }
    \beta = \log_2(1+\sqrt{2}).
  \]
\end{proposition}

\begin{proof}
  % elementary proof
  First notice that $m: \PSL_2(\RR) \to \CC$ is right $K$-invariant, with $K = \PSO(2)$. Similarly, when $\sG_0$ acts on the left of $m$ we have that
  \[
    m( S^{k} \, g ) 
    \, = \,
    (-1)^{k} m (g),
    \quad \mbox{ for } g \in \PSL_2(\ZZ), \, k \in \ZZ
  \]
  which means that $m$ is left $\sG_0$-invariant with respect to a character $\sG_0 \to \TT$ and Remark \ref{rmk:NonInvariantCharacters} applies.
  We need to prove that either $m(g h) = m(g)$ or $m(g^{-1}) = m(h)$ for $g \in \PSL_2(\ZZ) \setminus \sG_0$ and 
  $h \in \PSL_2(\ZZ) \setminus \{\Id\}$. Without loss of generality we can assume that $h$ lives in the larger group $\PSL_2(\RR)$ and use the right $K$-invariance of $m$ to assume that
  \[
    h
    \, = \,
    \begin{pmatrix}
      \sqrt{y} & \frac{x}{\sqrt{y}}\\
      0 & \sqrt{y}^{-1}
    \end{pmatrix}
    \quad 
    \mbox{ and }
    \quad
    g 
    \, = \,
    \begin{pmatrix}
      a & b\\
      c & d
    \end{pmatrix},
  \]
  for some $x, y \in \RR$ with $y > 0$ and $a, \, b, \, c, \, d \in \ZZ$. We are going to use the following elementary identities
  \begin{eqnarray*}
    m(g) & = & \sgn(ac+bd),\\
    m(g^{-1}) & = & \sgn(-dc -ab),\\
    m(h) & = & \sgn(x),\\
    m(gh) & = & \sgn \big(ac(x^2 + y^2 ) + (ad + bc)x + bd\big).
  \end{eqnarray*}  
  If $m(g^{-1}) = m(h)$, then the Cotlars's identity is already fulfilled. 
  Thus, we consider the case when $m(g^{-1}) \neq m(h)$, which is equal to saying that $x \, (-ab - dc) < 0$. 
  On the other hand, since $ad - bc = 1$ and $a, b, c, d \in \ZZ$, we get $abcd = bc + (bc)^2 \geq 0$. 
  We also get that $x(ab(d^2 + c^2 ) + dc(a^2 + b^2 )) \geq  0$. 
  This observation gives us the following result
  \[
    \begin{split}
      \big( a c(x^2 + & y^2) + (ad + bc)x + bd \big)(ac + bd)\\
      & = (a^2 c^2 + abcd)(x^2 + y^2 ) + \big(ab(d^2 + c^2 ) + dc(a^2 + b^2 )\big)x + abcd + b^2 d^2 \geq 0.
    \end{split}
  \]
  This translates to $m(gh) = m(g)$, which gives the result.
\end{proof}

%%% alternative proofs
 
The proposition above also admits a geometric proof. As we discussed in the introduction, there is a faithful and transitive isometric action of $\PSL_2(\RR)$ on the Poincar\'e plane $\HH$ through M\"obious transformations. As a lattice of $\PSL_2(\RR)$, $\PSL_2(\ZZ)$ naturally acts on $\HH$. Now consider a fundamental domain  $\sD$ of $\PSL_2(\ZZ) \acts \HH$, it is well known that such a domain $\sD$ can be taken to be a geodesic triangle with finite area but infinite diameter. Let us consider the hyperbolic tessellation:
\[
  \HH
  \, = \, 
  \bigsqcup_{g \in \PSL_2(\ZZ)} g \cdot \sD.
\]
We can build a graph $X$ such that their vertices can be either tiles of the above tessellation or each of the three sides of each tile, let us denote those by $X_{\mathrm{tiles}}$ and $X_{\mathrm{sides}}$ respectively. 
Two vertices of $x_0, x_1 \in X$ are connected by an edge in $\Edg_+(X)$ if $x_0$ can be labeled by a tile and $x_1$ can be labeled by one of its boundary segments. They will be similarly connected by an edge in $\Edg_-(X)$ if the roles of $x_0$ and $x_1$ are reversed. $X$ is a tree, see Figure \ref{fig:SL2Tessel}. 
Obviously the action of the modular group induces an action on $X$, and for every $x\in X_{\mathrm{tiles}}$, its orbit is equal to $X_{\mathrm{tiles}}$; for each $x\in  X_{\mathrm{sides}}$, its orbit is equal to $X_{\mathrm{sides}}$. It is also trivial to see that the stabilizer of each element in $X_{\mathrm{sides}}$ is conjugate to $\{\Id, S\} \cong \ZZ_2$, while the stabilizer of each element in $X_{\mathrm{tiles}}$ is conjugate to $\{\Id, \, (ST), \,(ST)^{-1}\} \cong \ZZ_3$.
By the Bass-Serre theory reviewed in Section \ref{S:BS} it is immediate that $\PSL_2(\ZZ) \cong \ZZ_2 \ast \ZZ_3$.
Furthermore, we can take the fundamental domain $\sD$ such that one of its boundary segments $x_0 \in X_{\mathrm{sides}}$ lays within the geodesic $\{z : \Re\{z\} = 0\}$. 
Then, it is clear that the function $z \mapsto \sgn ( \Re \{z\} )$ takes two values, it is constant on each connected component of $X \setminus \{x_0\}$. Now we can apply Model \ref{mod:model1} to obtain that Cotlar's identity holds. Observe that Proposition \ref{prp:InitialSegment} and the discussion before the theorem implies that $g \mapsto \Re\{g \cdot i\}$ depends only on the initial segment. Therefore \eqref{eq:HilbertContSLR} restricted to the modular group $\PSL_2(\ZZ)$ coincides with a free Hilbert transform for $\ZZ_2 \ast \ZZ_3$.

\begin{figure}[htbp]
	\centering
	\includegraphics[width=0.60\textwidth]{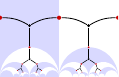}
	\caption{Tesselation associated to the modular group $\PSL_2(\ZZ)$ and its tree $X$.}
	\label{fig:SL2Tessel}
\end{figure}

%%% remark on all the ways it can be defined
%%% hyperbolic space
%%% 1-cocycle
%%% free group Hilbert transform
%%% KAN decompositions

%%% proof for the Bianchi groups 
\textbf{The case of $\PSL_2(\O_{-1})$.}
The purpose of this subsection is to prove Theorem \ref{thm:BoundednessBianchi}. Recall that $\PSL_2(\O_{-1}) \subseteq \PSL_2(\CC)$ is a lattice, we will denote such group by $\Gamma_1$. Recall as well that $\PSL_2(\CC)$ is semisimple and that any element admits a unique $KAN$-decomposition with $K =\PSU(2)$, $A$ the Abelian group of real diagonal matrices and $N$ the nilpotent, in fact Abelian, group of upper triangular complex matrices with ones on the diagonal. More concretely we have that
\[
  \bigg(
    \underbrace{ 
      \begin{matrix}
        a & b\\
        c & d
      \end{matrix}
    }_{g}
  \bigg)
  \, = \,
  \bigg(
    \underbrace{ 
      \begin{matrix}
        \alpha & \beta \\
        - \bar \beta  & \bar \alpha 
      \end{matrix}
    }_{\in K}
  \bigg)
  \bigg(
    \underbrace{ 
      \begin{matrix}
        s & 0 \\
        0 & s^{-1}
      \end{matrix}
    }_{\in A}
  \bigg)
  \bigg(
    \underbrace{ 
      \begin{matrix}
        1 & t\\
        0 & 1
      \end{matrix}
    }_{\in N}
  \bigg),
  \quad
  \mbox{ where } 
  s=\sqrt{|a|^2+|c|^2} 
  \mbox{ and } 
  t=\frac{\bar a b+ \bar c d}{|a|^2+|c|^2}.
\]

Let $g \in \Gamma_{1}$ be a matrix of the form
\begin{equation}
  \label{eq:matrixPSL2}
  g=
  \begin{pmatrix}
    a & b\\
    c & d
  \end{pmatrix}
  =
  \begin{pmatrix}
    a_1 + i \, a_2 & b_1 + i \, b_2\\
    c_1 + i \, c_2 & d_1 + i \, d_2
  \end{pmatrix}.
\end{equation}
The condition that the determinant equals $1$ implies that
\begin{eqnarray}
  a_1 d_1 + b_2 c_2 - ( b_1 c_1 + a_2 d_2) & = & 1, \label{eq: gamma1 1}\\
  a_1 d_2 + a_2 d_1 - b_1 c_2 - b_2 c_1 & = & 0. \label{eq: gamma1 2}
\end{eqnarray}
Using the two identities above, we can elementary prove the following key inequality.
\begin{lemma}
  \label{lem:AuxClaim0}
  Let $g \in \Gamma_1$ be as in \eqref{eq:matrixPSL2}. It holds that
  \[
    0
    \, \leq \,
    \big( a_1 d_1 + b_2 c_2 \big) \, \big( b_1 c_1 + a_2 d_2 \big)
    \, \leq \,
    \Re ( a \bar{c} ) \, \Re ( b \bar{d} )
    \, = \,
    \big( a_1 c_1 + a_2 c_2 \big) \, \big( b_1 d_1 + b_2 d_2 \big)
  \]
\end{lemma}

\begin{proof}
  For the first inequality, let $A = a_1 d_1 + b_2 c_2$ and $B = b_1 c_1 + a_2 d_2$. 
  By equation \eqref{eq: gamma1 1} we have $A = 1 + B$. But this implies that $A B = B^2 + B$. 
  Since $B$ is an integer $AB \geq 0$. For the second inequality, it is enough to show that 
  $0 \leq {\rm (I)} - {\rm (II)}$, where ${\rm (I)} = ( a_1 c_1 + a_2 c_2 ) \, ( b_1 d_1 + b_2 d_2 )$ and
  ${\rm (II)} = ( a_1 d_1 + b_2 c_2 ) \, ( b_1 c_1 + a_2 d_2 )$. We apply \eqref{eq: gamma1 2} to obtain
  \begin{eqnarray*}
    {\rm (I)} 
      & = & a_1 b_1 c_1 d_1 + a_1 b_2 c_1 d_2 + a_2 b_1 c_2 d_1 + a_2 b_2 c_2 d_2\\
      & = & a_1 b_1 c_1 d_1 + b_1 b_2 c_1 c_2 + 
            \big( b_2 c_1 - a_2 d_1 \big) b_2 c_1 + a_2 b_1 c_2 d_1 + a_2 b_2 c_2 d_2,\\
    {\rm (II)} 
      & = & a_1 b_1 c_1 d_1 + a_1 a_2 d_1 d_2 + b_1 b_2 c_1 c_2 + a_2 b_2 c_2 d_2\\
      & = & a_1 b_1 c_1 d_1 + a_2 b_1 c_2 d_1 + 
            \big( b_2 c_1 - a_2 d_1 \big) a_2 d_2 + b_1 b_2 c_1 c_2 + a_2 b_2 c_2 d_2.
  \end{eqnarray*}
  Thus ${\rm (I)} - {\rm (II)} = \big( b_2 c_1 - a_2 d_1 \big) b_2 c_1 - \big( b_2 c_1 - a_2 d_1 \big) a_2 d_2 = (X - Y) X - (X - Y) Y  = (X - Y)^2 \geq 0$,. where $X = b_2 c_1$ and $Y = a_2 d_1$.
\end{proof}

The next lemma follows by elementary calculations.

\begin{lemma}
  \label{claim 1}
  Let $g \in \Gamma_{1}$ as in \eqref{eq:matrixPSL2}. 
  \begin{enumerate}[leftmargin=1.25cm, label={\rm \textbf{(\roman*)}}, ref={\rm {(\roman*)}}]
    \item \label{itm:claim1.1} If we assume that $\Re \{a \bar c\} \neq 0$ and $\Re \{ b\bar d \} \neq 0$, 
    then we get $\, \sgn \big( \Re \{ a\bar c\} \big) = \sgn \big( \Re \{ b\bar d \} \big)$.
    
    \item \label{itm:claim1.2} 
    On the other hand, $\Re \,\{ a\bar  c\}  \, \Re \, \{ b\bar d\} = 0$ if and only if 
    $(a_1 d_1 + b_2 c_2)(b_1c_1+a_2d_2)=0$  and $b_2 c_1 = a_2 d_1$.
    
    \item \label{itm:claim1.3} If $\Re \{\bar a b\} \neq 0$ and $\Re\{\bar c d\} \neq 0$, then 
    $\sgn (\Re \{\bar a b\})= \sgn( \Re \{\bar c d\} ).$

    \item \label{itm:claim1.4}  $\Re \,\{ \bar a b\}  \, \Re \, \{\bar c d\} = 0$ if and only if 
    $\big( a_1 d_1 + b_2 c_2 \big) \big( b_1 c_1 + a_2 d_2 \big) = 0$ and $b_1 c_2 = a_2 d_1$.
    
    \item \label{claim 2}
    It holds that $\Im \big\{ b \bar c- a \bar d \big\}^2 - 4 \Re \{ a\bar c \}  \cdot  \Re \{ b\bar d \} \leq 0.
    $
  \end{enumerate}
\end{lemma}

\begin{proof}
  We will just given an sketch. Observe that, by Lemma \ref{lem:AuxClaim0}, we have that $\Re(a\bar{c}) \, \Re (b \bar{d}) \geq 0$. Therefore, if both terms are different from zero, they have the same sign. The point \ref{itm:claim1.2} follows similarly by Lemma \ref{lem:AuxClaim0} and its proof. Points \ref{itm:claim1.3} and \ref{itm:claim1.4} are a reiteration of the previous two points but changing rows by columns. Lastly, for \ref{claim 2}, we start by noticing that 
  \[
    \Im \{b\bar{c}\} - \Im \{a \bar{d}\}
    = b_2 c_1 - b_1 c_2 - a_2 d_1 + a_1 d_2
    = 2 \big( a_1 d_2 - b_2 c_1 \big),
  \]
  where we have used \eqref{eq: gamma1 2}. Now, we have that 
  \begin{eqnarray}
    & & \big( b_2 c_1 - a_2 d_1 \big)^2 - \big( a_1 c_1 + a_2 c_2 \big)\,\big( b_1 d_1 + b_2 d_2 \big) \nonumber\\
    & = & b_2^2 c_1^2 + a_2^2 d_1^2 - 2 a_2 b_2 c_1 d_1 - a_1 b_1 c_1 d_1 - a_1 b_2 c_1 d_2 
          - a_2 b_1 c_2 d_1 + a_2 b_2 c_2 d_2 \label{eq:Step2}\\
    & = & a_2 d_1 \big( a_2 d_1 - b_2 c_1 - b_1 c_2  \big) 
          - a_1 b_1 c_1 d_1 - b_1 b_2 c_1 c_2 - a_2 b_2 c_2 d_2 \label{eq:Step3} \\
    & = & - a_1 a_2 d_1 d_2 - a_1 b_1 c_1 d_1 - b_1 b_2 c_1 c_2 - a_2 b_2 c_2 d_2 \nonumber\\
    & = & - \big( a_2 d_2 + b_1 c_1 \big) \big( a_1 d_1 + b_2 c_2 \big) \label{eq:Step4}\\
    & = & - X ( 1 + X ) \,\, \leq \, \, 0, \nonumber
  \end{eqnarray}
  where $X = a_2 d_2 + b_1 c_1$. We have used identity \eqref{eq: gamma1 2} in \eqref{eq:Step2}, and \eqref{eq:Step3}. For \eqref{eq:Step4} we use \eqref{eq: gamma1 1}. The last term is negative since $X$ is an integer and $X^2 + X$ is always positive when $X \in \ZZ$.
\end{proof}

%Note that $m (g)=-\tilde m (g^{-1})$ with $ \tilde m (g)= {\rm sgn} ({\Re\, t})={\rm sgn} \,( {\Re \, (\bar a b + \bar c d}))$. By \eqref{eq: symbol bianchi}, it is easy to see $m$ is right $K$-invariant, that is $m(gk)=m(g)$ for any $g\in {\rm PSL} _2(\mathbb{C})$ such that $ \Re \,(a\bar c +b \bar d)\neq  0$ and $k\in {\rm SU}(2)$.
%We are going to start by analyzing the function \eqref{eq:HilbertContSLC}. First, notice that it is clearly right $\PSU(2)$-invariant. Restricting the above function to $\PSL_2(\ZZ[i])$, we can say a bit more.

Now let us consider the symbol given in \eqref{eq:HilbertContSLC}. When it is restricted to $\Gamma_1$, is $0$ over a large subset. We are going to show that this subset is in fact a subgroup $\Gamma_0 \subseteq \Gamma_1$. The following proposition follows from the previous lemma.

\begin{lemma}
  Let $g \in \Gamma_{1}$ be a matrix with coefficients like those in \eqref{eq:matrixPSL2}
  and such that $\Re \{ a \bar c +b \bar d \} = 0$. Then $g \in \Gamma_0^+ \cup \Gamma_0^-$, where
  \begin{eqnarray*}
    \Gamma_0^+
      & = & 
      \left\{
      \begin{pmatrix}
        a_1 & b_2 i\\
        c_2 i & d_1
      \end{pmatrix}: a_1,b_2,c_2, d_1 \in \ZZ, a_1 d_1+b_2 c_2 = 1 \right\}\\
    \Gamma_0^-
      & = & 
      \left\{
      \begin{pmatrix}
        a_2 i & b_1 \\
        c_1 & d_2 i
      \end{pmatrix}: a_2, b_1, c_1, d_2 \in \ZZ, -a_2 d_2 - b_1 c_1 = 1 \right\}
  \end{eqnarray*} 
 It is clear that $\Gamma_0^+$ and $\Gamma_0 = \Gamma_0^+ \cup \Gamma_0^-$ are subgroups of 
 $\Gamma_1$. 
\end{lemma}

\begin{proof}
  %\marginnote{{\color{red} RFR: in the proof of Lemma 5.4, you also need to consider the case when a = 0 or b = 0.}}
  Observe that, if $\Re\{a\overline c\}\neq 0$ and $\Re\{a\overline c\}\neq 0$, by Lemma \ref{claim 1}.\ref{itm:claim1.1}, $\Re\{ a \bar{c}\}$ and $\Re\{b \bar{d}\}$ have the same sign, therefore if $\Re \{ a \bar{c} + b \bar{d} \} = 0$ that is because $\Re \{a \bar{c}\} = \Re \{b \bar{d}\} = 0$. If $\Re\{a\overline c\} \Re\{a\overline c\}=1$, we also have $\Re\{a\overline c\} \Re\{a\overline c\}=0$. Assume that $a_1$ and $a_2$ are coprimes and that $b_1$ and $b_2$ are also coprimes. Since $\Re\{ a \bar{c}\}= a_1 c_1 + a_2 c_2 = 0$, we have that $(a_1, a_2)$ and $(c_1, c_2)$ are perpendicular vectors with integer coordinates. But, since $a_1$ and $a_2$ are coprimes $(c_1, c_2) = \ell (a_2, -a_1)$ for some $\ell \in\mathbb{Z}$. Similarly $(d_1, d_2) = m \, (b_2, -b_1)$  for some $m \in\mathbb{Z}$. Computing the determinant gives
  \[
    \det
    \begin{pmatrix}
      a_1 + i a_2 & b_1 + i b_2 \\
      \ell a_2 - i \ell a_1 & m b_2 - i m b_1  
    \end{pmatrix}
    \, = \,
    \big( m - \ell \big)\big( a_2 b_1 + a_1 b_2 \big) + i \big( m - \ell \big) \big( a_2 b_2 - a_1 b_1 \big)
    \, = \, 1.
  \]
  But this implies that $m - \ell = \pm 1$ and $a_2 b_1 + a_1 b_2 = \pm 1$. Since the imaginary part has to be $0$ and
  $m - \ell = \pm 1$, it follows that $a_2 b_2 = a_1 b_1$. But since $a_1$ and $a_2$, like $b_1$ and $b_2$, are coprimes, we have that $a_2 = b_1$ and $a_1 = b_2$. Therefore $a_1^2 + a_2^2 = 1$ and $b_1^2 + b_2^2 = 1$. Since they are integers, one of $a_1$, $a_2$ and of $b_1$ and $b_2$  has to be $0$. So we have either $a_2=b_1=0$ or $a_1=b_2=0$. The case of non-coprime $a_1$ and $a_2$ can be proved similarly, by noticing that $a_1 + i a_2 = k \alpha_1 + i k \alpha_2$ with $k = \gcd(a_1, a_2)$ and $\alpha_1$ and $\alpha_2$ coprimes. In that case every perpendicular vector to $(a_1,a_2)$ with integer coordinates is of the form $(\ell \alpha_2, - \ell \alpha_1)$.
\end{proof}

Observe that $\Gamma_0^+$ is isomorphic to the subgroup $\PSL_2(\ZZ)$ in $\Gamma_1$. Indeed, if we take the natural embedding $\PSL_2(\ZZ) \subseteq \Gamma_{1}$, we have that
\[
  \Gamma_0^+ \, = \, J \, \PSL_2(\ZZ) \, J^{-1}, 
  \quad 
  \mbox{ where } 
  J = 
  \begin{pmatrix}
    e^{\frac{i \pi}{4}} & 0\\
    0 & e^{-\frac{i \pi}{4}}
  \end{pmatrix}.
\]
The group $\Gamma_0 = \Gamma_0^+ \cup \Gamma_0^-$ is generated by $\Gamma_0^+$ and the diagonal matrix with eigenvalues $i$ and $-i$. It is trivial to check that $\Gamma_0$ it is isomorphic to $\PSL_2(\ZZ) \rtimes \ZZ_2$. With the order two automorphism given by conjugation. 

Now, we introduce the following modification from the symbol in \eqref{eq:HilbertContSLC}.

%\marginnote{{\color{red} RFR: Definition 5.5, the first condition is when $g \not\in \Gamma_0$. What is $\Gamma$ in the definition of the singular elements?}}
\begin{definition}\normalfont
  \label{def:ModifiedBianchi}
  Let us define the symbol $m_2: \Gamma_{1} \to  \CC$ by
  \[
    m_2(g) = 
    \begin{cases}
      \sgn \big( \Re \{ a \, \bar{c} + b \, \bar{d} \} \big) 
        & \mbox{ when } g \not\in \Gamma_0\\
      1 & \mbox{ when } g \in \Gamma_0^-\\
      0 & \mbox{ when } g \in \Gamma_0^+.
    \end{cases}
  \]
  We call the group elements of $\Gamma$ the singular points of $m_2$. 
\end{definition}

\begin{proposition}
  \label{leftrightinv}
  The function $m_2$ given in Definition \ref{def:ModifiedBianchi} is both left and right $\Gamma_0^+$-invariant.
\end{proposition}

%\marginnote{{\color{red} RFR: Remark 5.7 cannot be correct, because multiplication on the left or the  right by $i J$ exchanges $\Gamma_0^{+}$ and $\Gamma_0^{-}$ . As a consequence, Proposition 5.8 is not proved.}}
\begin{proof}
Let $g \in \Gamma_1$ be like in \eqref{eq:matrixPSL2}. We first prove $m_2$ is left $\Gamma_0^+$-invariant. If $h \in \Gamma_0^+$, since $\Gamma_0^+$ is a subgroup of $\Gamma_1$, it is obvious that $m_2(gh)=m_2(h)=0$. If $h\in \Gamma_0^-$, it is easy to check $gh \in \Gamma_0^-$, thus $m_2(gh)=m_2(h)=1$. If $h\notin \Gamma_0$, then $m_2$ is right $K$-invariant and we have $m_2(h)=\Re( t)$, where
\[
  h \, = \,
  \begin{pmatrix}
    s & st\\
    0 & s^{-1}
  \end{pmatrix}
  k \quad
  \mbox{ for some }
  k \in K.
\]
Moreover, we have $gh \notin \Gamma_0$ and 
\begin{equation*}
  m_2(gh) 
  =  
  m_2
  \bigg(
    \begin{pmatrix}
      a_1 & b_2 i\\
      c_2 i & d_1
    \end{pmatrix}
    \begin{pmatrix}
      s & st\\
      0 & s^{-1}
    \end{pmatrix}
  \bigg)
  \, = \,
  \big( a_1 d_1 + b_2 c_2 \big) \, \Re(t) \, = \, \Re (t) = m_2(h).
\end{equation*}
Now we show $m_2$ is right $\Gamma_0^+$-invariant. Similarly as the above, if $h \in \Gamma_0^+$, $m_2(h g) = m_2(h) = 0$; if $h \in \Gamma_0^-$, we have $hg \in \Gamma_0^-$ and  $m_2(hg)=m_2(h)=1$.  If $h \notin \Gamma_0$ and let 
\[
  h \, = \,
  \begin{pmatrix}
    a & b\\
    c & d
  \end{pmatrix}.
\]
Then $h g \notin \Gamma_0$. 
Therefore we have 
\begin{eqnarray*}
  m_2(hg)
   & = & 
     m_2 \bigg(
     \begin{pmatrix}
       a & b\\
       c & d
     \end{pmatrix}
     \begin{pmatrix}
       u & u v\\
       0 & u^{-1}
     \end{pmatrix}
      \bigg)
      \\
   & = & {\rm sgn} \Big[  {\Re \, (a\bar c)} \, u^{2} (1+ \, \Im\, v^2)+ \Im\, (b\bar c - a\bar d)  \, \Im\, v +\Re\, (b \bar d) \, u^{-2}\Big],
\end{eqnarray*}
where
\[
  u=\sqrt{d_1^2+c_2^2}\neq 0 
  \quad \mbox{ and } \quad 
  v=\frac{(b_2d_1-a_1c_2)i}{\sqrt{d_1^2+c_2^2}}
\]
Recall that $m_2(h)= \sgn \,( {\Re \, (a\bar c +b \bar d}))$. If $\Re \, (a\bar c) \neq 0$ and $ \Re \, (b\bar d ) \neq 0$, then by Lemmas \ref{claim 1} and \ref{claim 2}, it is easy to see $m_2(hg)=m_2(h)$. If $\Re \, (a\bar c)=0$ and $\Re \, (b\bar d)\neq 0$, we have $m_2(hg)=\sgn [\Re\, (b \bar d) \, u^{-2}]$ and $m_2(h)=\sgn (\Re \, (b\bar d))$. Thus $m_2(hg)=m_2(h)$. If $\Re \, (a\bar c)\neq 0$ and $\Re \, (b\bar d)= 0$, we have $m_2(hg)=\sgn [ {\Re \, (a\bar c)} \, u^{2} (1+ \, \Im\, v^2)]$ and  $m_2(h)=\sgn (\Re \, (a\bar c))$, which implies $m_2(hg)=m_2(h)$. 
\end{proof}

\begin{proposition}
  The function $m_2$ given in Definition \ref{def:ModifiedBianchi} satisfies the Cotlar
  identity \eqref{eq:CotlarNC} relative to $\Gamma_0^+$, i.e., for any $g \in \Gamma_{1} \setminus \Gamma_0^+$ 
  and $h \in \Gamma_{1}$, it holds that
  \[
    \big( m_2(g^{-1}) - m_2(h) \big) \big( m_2(gh) - m_2(g) \big) \, = \, 0.
  \]
\end{proposition}

\begin{proof}
Since $m_2$ is right $\Gamma_0^+$-invariant by the Proposition \ref{leftrightinv}, it suffices to prove the Cotlar identity for $g$ and $h$ in $\in \Gamma_1 \setminus \Gamma_0^+$. That set decomposes as $\Gamma_1 \setminus \Gamma_0^+ = (\Gamma_1 \setminus \Gamma_0) \sqcup \Gamma_0^-$ and thus we have $4$ potential cases. When $g \in \Gamma_1 \setminus \Gamma_0$, then for every $h \in \Gamma_0$ it holds that $m_2(gh) = m_2(h)$. In the case in which $g \in \Gamma_0^-$ and $h \in \Gamma_0^-$ we have that both $g, h^{-1} \in \Gamma_0^-$ and as such $m_2(g) = m_2(h^{-1})$. Thus, we can assume that $g \in \Gamma_1 \setminus \Gamma_0^+$ and $h \in \Gamma_1 \setminus \Gamma_0$.
First, let us tackle the case of $g \in \Gamma_0^-$. 
In this case we have $m_2(g)=m_2(g^{-1})=1$. If $m_2(g^{-1}) = m_2(h)$, then the identity is satisfied. If $m(g^{-1}) \neq m_2(h)$,  this implies $m_2(h)=-1$, and then $h, g h \notin \Gamma_0$. By the right $K$-invariance of $m_2$ on $\Gamma_1 \setminus \Gamma_0$, we get 
\[
  m_2(gh)
  \, = \,
  m_2 \bigg(
  \begin{pmatrix}
    a_2i & b_1 \\
    c_1 & d_2 i
  \end{pmatrix}
  \begin{pmatrix}
    s & st\\
    0 & s^{-1}
  \end{pmatrix}
  \bigg)
= (a_2d_2 +b_1c_1) \Re \, t= -\Re \, t=-m_2(h)=1.
\]
Therefore, we get $m_2(gh)=m_2(g)$, the Cotlar identity is satisfied. 
Now let us focus on the case of $g \notin \Gamma_0$.
Let $g \notin \Gamma_0$, i.e.  $\Re \,(a\bar c +b \bar d)\neq  0$ and 
\[
  h=
  \begin{pmatrix}
    s & st\\
    0 & s^{-1}
  \end{pmatrix} 
  k
  \notin \Gamma_0
\]
with $k \in \PSU(2)$. We get the following expressions for the four terms appearing in the Cotlar identity:
\begin{eqnarray*}
  m_2(g) & = &  \sgn \,( {\Re \, (a\bar c +b \bar d})),\\
  m_2(g^{-1}) & = & -\sgn \,( {\Re \, (\bar a  b + \bar c d})), \\
  m_2(h) & = &\sgn \, (\Re \, t), \\
  m_2(gh) 
    & = & 
      m_2 \bigg(
      \begin{pmatrix}
        a & b\\
        c & d
      \end{pmatrix}
      \begin{pmatrix}
        s & st\\
        0 & s^{-1}
      \end{pmatrix}  \bigg)\\
    & = & \sgn \Big[  {\Re \, (a\bar c)} \, s^2 (1+|t|^2)+\Re \, (a\bar d +b \bar c)\,\Re \, t 
          + \Im\, (b\bar c - a\bar d)  \, \Im\, t +\Re\, (b \bar d) \, s^{-2}\Big].
\end{eqnarray*}
In the identities for $m_2(h)$ and $m_2(gh)$, we have used the property that $m_2$ is right $K$-invariant on $\Gamma_1 \setminus \Gamma_0$. Now we prove the Cotlar identity. If $m_2(g^{-1}) = m_2(h)$, then the identity is satisfied. If not, that means 
\begin{equation}
  \label{eq: inverse g not equal h}
  \Re \, t \cdot  \Re \, (\bar a  b + \bar c d ) > 0.
\end{equation}
Then we need to show $m(g) = m(gh)$, or in other words, 
\begin{eqnarray*}
  & & \Big[ {\Re \, (a\bar c)} \, s^2 (1+|t|^2) +
            \Re \, (a\bar d +b \bar c)\,\Re \, t + 
            \Im\, (b\bar c -  a\bar d) \, \Im\, t +
            \Re\, (b \bar d) \, s^{-2}\Big] \cdot  {\Re \, (a\bar c +b \bar d}) \\
  & = & {\Re \, (a\bar c +b \bar d})\, \Re \, (a\bar c) \, s^2 (1+ (\Re \, t )^2) + 
        ( {\Re \, (a\bar c +b \bar d})\, \Re \, (a\bar d +b \bar c)\,\Re \, t \\
  & & +  {{ \rm Re} \, (a\bar c +b \bar d})\cdot \,\big[\Re \, (a\bar c) \, s^2 (\Im\, t)^2 +
      \Im\, (b\bar c - a\bar d) \, \Im\, t + \Re\, (b \bar d) \, s^{-2}\big]>0.
\end{eqnarray*}
We first assume that $\Re \, (a\bar c) \neq 0$ and $ \Re \, (b\bar d) \neq 0$. By Lemma \ref{claim 1}, we have 
\[
  \Re \, (a\bar c +b \bar d)\, \Re\, (a\bar c)>0.
\] 
Moreover, Lemma \ref{claim 1}.\ref{claim 2} implies that the determinant of the quadratic function of $\Im \,t$ is negative, together with the inequality above, we see that  
\[
 {{ \rm Re} \, (a\bar c +b \bar d})\cdot \,\big[\Re \, (a\bar c) \, s^2 (\Im\, t)^2 +  \Im\, (b\bar c - a\bar d)\, \Im\, t + \Re\, b \bar d \, s^{-2}\big]\geq 0
\]
Moreover, we claim that
\[
  \Re \, (a\bar c +b \bar d) \Re \, (a\bar d +b \bar c) \, \Re \, t \geq 0.
\]
If the claim is true then we will obtain $m(g) = m(gh)$. So now it remains to prove the claim. Notice that by \eqref{eq: inverse g not equal h} and Lemma \ref{claim 1}, it is enough to show 
\[
\Re\, (\bar a b) \, \Re\, (a\bar c)\,\Re \, (a\bar d +b \bar c) \geq 0.
\]
According to \eqref{eq: gamma1 1}, $\Re \, (a\bar d +b \bar c)= a_1 d_1 +a_2 d_2+b_1 c_1 +b_2 c_2 = 2(b_1 c_1 +a_2 d_2)+1$. On the other hand, \eqref{eq: gamma1 1} and \eqref{eq: gamma1 2} imply that 
\begin{eqnarray*}
\Re\, (\bar a b) \, \Re\, (a\bar c)
  & = & (a_1 b_1+a_2 b_2)(a_1c_1+a_2c_2)\\
  & = & a_1a_2 (b_1c_2+b_2 c_1)+a_1^2b_1 c_1+a_2^2b_2c_2\\
  & = & a_1a_2 (a_1 d_2 +a_2 d_1)+a_1^2 b_1 c_1+a_2^2b_2c_2\\
  & = & a_1^2 (b_1 c_1 +a_2 d_2) +a_2^2 (a_1 d_1 + b_2 c_2)\\
  & = & (b_1 c_1+a_2 d_2)(a_1^2 +a_2^2)+a_2^2.
\end{eqnarray*}
Let $X= b_1c_1+a_2 d_2$, $A= a_1^2+a_2^2$ and $B= a_2^2$. Then $\Re\, (\bar a b) \, \Re\, (a\bar c)\,\Re \, (a\bar d +b \bar c) =(AX+B)(2X+1)$. Since $X\in \ZZ$ and $|\frac{B}{A}|\leq 1$, we get $(AX+B)(2X+1)\geq 0$, which proves the claim. 

Now we deal with the case when $\Re \,( a\bar  c)  \, \Re \, (b\bar d) = 0$. By Lemma \ref{claim 1}, it is equivalent to saying $(a_1d_1+b_2c_2)(b_1c_1+a_2d_2)=0 \text{ and } b_2c_1=a_2d_1$. Without loss of generality, we assume that  $b_1c_1+a_2d_2=0$. Then by \eqref{eq: gamma1 1}, $a_1 d_1 +b_2 c_2=1$. Since $b_2c_1=a_2d_1$,  \eqref{eq: gamma1 2} tells us that $b_1c_2 = a_1 d_2$ and $ \Im\,( b\bar c -  a\bar d )=2(b_2 c_1 -a_2 d_1)=0$. This implies that 
\begin{equation}
  \label{eq: relation 1}
  a_1 b_1 = a_1 b_1 (a_1 d_1 +b_2 c_2)= a_1^2 (b_1 d_1+ b_2 d_2)=a_1^2 \Re\, (b \bar d).
\end{equation}
Similarly, we have 
\begin{equation}
  \label{eq: relation 2}
  c_2 d_2= c_2^2 \, \Re\, (b \bar d), \;\; a_2 b_2 
  =
  b_2^2 \, \Re\, (a \bar c) \;\; \text{ and }\; c_1 d_1 = d_1^2 \, \Re\, (a \bar c).
\end{equation}
Since $\Re \,( a\bar  c+b\bar d)\neq 0$, $\Re \,( a\bar  c)$ and $ \Re \, (b\bar d) $ can not be zero at the same time.  Suppose $\Re \,( a\bar  c)=0$ and $ \Re \, (b\bar d ) \neq  0$. Then $m(g)=   \sgn \, {\Re \, (b \bar d})$, 
$m(gh) =  \sgn \big[ \Re \, (a\bar d +b \bar c)\,\Re \, t + \Re\,( b \bar d) \, s^{-2}\big]= \sgn \big[ \Re \, t + \Re\, (b \bar d) \, s^{-2}\big]$. Applying \eqref{eq: relation 1} and \eqref{eq: relation 2}, we get $\Re \, (\bar a b + \bar c d)=(a_1^2 + c_2^2) \Re\, (b \bar d)$. Recall that we have $\Re \, t \cdot  \Re \, (\bar a  b + \bar c d )>0$ by the assumption $m(g^{-1})\neq m(h)$. Therefore, we have $\Re \, t \cdot \Re\, (b \bar d) >0$, which implies $m(g)=m(gh)$. For the case $\Re \,( a\bar  c)\neq 0$ and $ \Re \, (b\bar d ) =  0$, we omit the proof since  it is similar to the previous case. 
\end{proof}

Now, we can prove Theorem \ref{thm:BoundednessBianchi} by just reducing the $L_p$-boundedness of $m$ to that of $m_2$.
%\marginpar{{\color{blue}  Did we show $\Gamma_0$ and $\Gamma_0^+$ are open?}}
\begin{proof}[Proof (of Theorem \ref{thm:BoundednessBianchi})]
  Notice that $m: \Gamma_{1} \to \CC$ satisfies that
  \[
    m(g) = m_2(g) - \1_{\Gamma_0} + \1_{\Gamma_0^+}.
  \]
  But Fourier multipliers with symbols being characteristic functions of open subgroups are bounded in $L_p$ for every $1 \leq p \leq \infty$. Therefore, up to a finite constant smaller that $2$, the operator $L_p$-norm of $T_m$ is bounded by that of $T_{m_2}$.
\end{proof}

%\textbf{Acknowledgment.} The authors were partly supported by the Spanish Grant \textbf{PID2019-107914GB-I00} ``Fronteras del An\'alisis Arm\'onico'' (MCIN / PI J. Parcet) as well as Severo Ochoa Grant \texttt{CEX2019-000904-S} (ICMAT), funded by \texttt{MCIN/AEI 10.13039/501100011033}.

%\begin{remark} \normalfont
%Note that the Hilbert transform associated with the symbol defined by \eqref{eq: symbol bianchi} is different from the one based on the free product structure that  $\Gamma_1=\G_1 \ast_{{\rm PSL}_2(\mathbb{Z})} \G_2$. We know that $m(g)=0$ for $g\in {\rm PSL}_2(\mathbb{Z})$ according to the discussion in subsection \ref{section: amalgam}. However, clearly from Definition \ref{def 3}, the singular points are elements in $\Gamma$ not ${\rm PSL}_2(\mathbb{Z})$. 
%\end{remark}

%\textbf{Acknowledgement.} The first-named author was partially supported by ANR grant HASCON and by 
%

\begin{small}
  \bibliographystyle{acm}
  \bibliography{bibliography}
\end{small}

\vfill
%\vspace{50pt}

%\

%\hfill \noindent \textbf{Adri\'an M. Gonz\'alez-P\'erez} \\
%\null \hfill Universit\'e Clermont-Auvergne - LMBP \\ 
%\null \hfill 3, Place Vasarely, 63178 Aubi\'ere, France
%\\ \null \hfill\texttt{chadrian.gzl@gmail.com}

%\

%\hfill \noindent \textbf{Javier Parcet} \\
%\null \hfill Consejo Superior de Investigaciones Cient\'ificas - ICMAT\\ 
%\null \hfill 23 Nicol\'as Cabrera, 28049 Madrid, Spain 
%\\ \null \hfill\texttt{parcet@icmat.es}

%\

%\hfill \noindent \textbf{Runlian Xia} \\
%\null \hfill Consejo Superior de Investigaciones Cient\'ificas - ICMAT\\ 
%\null \hfill 23 Nicol\'as Cabrera, 28049 Madrid, Spain 
%\\ \null \hfill\texttt{ }

\begin{tabularx}{\textwidth}{ | X | X}
  \textbf{Adri\'an M. Gonz\'alez-P\'erez}, & \textbf{Javier Parcet} \\
  \texttt{adrian.gonzalez@uam.es}          & \texttt{parcet@icmat.es}\\
  Dept. of Mathematics, UAM,               & CSIC - ICMAT\\
  7 Francisco Tom\'as y Valiente,          & 23 Nicol\'as Cabrera,\\
  28049 Madrid, Spain.                     & 28049 Madrid, Spain\\
                                           & \\
  \texttt{agonzalez-perez@impan.pl}        & \\
  IMPAN                                    & \\
  Jana i J{\c e}drzeja {\'S}niadeckich 8,  & \\
  00-656 Warszawa, Poland                  & 
\end{tabularx}

\vspace{5pt}

\begin{tabularx}{\textwidth}{ | X  X}
  \textbf{Runlian Xia}                     & \\
  \texttt{Runlian.Xia@glasgow.ac.uk}       & \\
  University of Glasgow                    & \\
 University Avenue,                    & \\
  Glasgow G12 8QQ, UK                      &
\end{tabularx}

\end{document}